\documentclass[12pt,reqno, final]{amsart}

\usepackage{amssymb,amsmath,graphicx,
	amsfonts,euscript}
\usepackage{color}
\usepackage{cite}
\allowdisplaybreaks

\newcommand{\be}{\begin{eqnarray}}
\newcommand{\ee}{\end{eqnarray}}
\newcommand{\bes}{\begin{eqnarray*}}
\newcommand{\ees}{\end{eqnarray*}}

\newcommand{\om}{\omega}

\newcommand{\na}{\nabla}

\newcommand{\p}{\partial}

\newtheorem{thm}{Theorem}[section]

\newtheorem{lem}{Lemma}[section]

\newcommand{\beq}{\begin{equation}}
\newcommand{\eeq}{\end{equation}}
\newcommand{\ben}{\begin{eqnarray}}
\newcommand{\een}{\end{eqnarray}}
\newcommand{\beno}{\begin{eqnarray*}}
\newcommand{\eeno}{\end{eqnarray*}}

\numberwithin{equation}{section}

\begin{document}
	
	\title[ANISOTROPIC 2D BOUSSINESQ SYSTEM]{STABILITY AND LARGE-TIME BEHAVIOR FOR THE 2D BOUSSINESQ SYSTEM WITH VERTICAL DISSIPATION AND HORIZONTAL THERMAL DIFFUSION}
	
	\author[O. Ben Said and M. Ben Said]{Oussama Ben Said$^{1}$ and Mona Ben Said$^{2}$}
	
	\address{$^1$ Department of Mathematics, Oklahoma State University, Stillwater, OK 74078, USA}
	
	\email{obensai@ostatemail.okstate.edu}

	\address{$^{2}$ Laboratoire Analyse, Géométrie et Applications, Université Paris 13, 99 Avenue Jean Baptiste Clément
93430 Villetaneuse, France}
	
	\email{bensaid@univ-paris13.fr }

	\vskip .2in
	\begin{abstract}
This paper addresses the stability and large-time behavior problem on the perturbations near the hydrostatic balance of the two dimensional Boussinesq system, taking into account vertical dissipation and horizontal thermal diffusion. The spatial framework $\Omega$ is defined as $ \mathbb{T}\times\mathbb{R}$, where $\mathbb{T}$ spans $[0, 1]$, representing the 1D periodic box, while $\mathbb{R}$ denotes the whole line. The results outlined in this article confirm the fact that the temperature can actually have a stabilizing effect on the buoyancy-driven fluids. The stability and long-time behavior issues discussed here are difficult due to the lack of the horizontal dissipation and vertical thermal diffusion. By formulating in the appropriate energy functional and implementing the orthogonal decomposition of the velocity and the temperature into their horizontal averages and oscillation parts, we are able to make up for the missing regularization and establish the nonlinear stability in the Sobolev space $H^2(\Omega)$ and acheive the algebraic decay rates for the oscillation parts in the $H^1$-norm.

	\end{abstract}

	\maketitle

\section{Introduction}
This paper focuses on the following 2D anisotropic Boussinesq system 
\beq \label{y}
\begin{cases}
	\partial_t U + U\cdot \nabla U= -\nabla {P}+ \nu\, \partial_{22}U + g_0\Theta {\mathbf e}_2, 
	\quad x\in\Omega, \,\,t>0, \\
	\partial_t \Theta + U \cdot \nabla \Theta = \eta\, \partial_{11}\Theta, \\
	\nabla \cdot U=0, 
\end{cases}
\eeq
where $U$ denotes the fluid velocity, $P$ the pressure, $\Theta$ the temperature, $\nu>0$ and $\eta>0$ are parameters representing the kinematic viscosity and the thermal diffusivity,
respectively. Here ${\mathbf e}_2=(0,1)$ 
is the unit vector in the vertical direction, $g_0$ is a non zero constant and the spatial domain $\Omega$ is taken to be
$$
\Omega=\mathbb{T}\times\mathbb{R},
$$
with $\mathbb{T} = [0,1]$ being a 1D periodic box and $\mathbb{R}$ being the whole line. This partially dissipated system 
models anisotropic buoyancy-driven fluids in the circumstance when the horizontal dissipation and the vertical thermal diffusion are negligible \cite{Ped}.
\vskip .1in

The Boussinesq systems stand out as the most commonly employed models for studying atmospheric and oceanographic flows (see, e.g., \cite{Blu}, \cite{Hol}, \cite{Maj}). Recent research has been focused on addressing two fundamental challenges related to these equations, namely, global existence and regularity
problem and the stability problem on perturbations
near various physically relevant equilibrium states (see, e.g., \cite{BPW}, \cite{ABPW}, \cite{CCL},\cite{DWZZ},  \cite{DWXZ}, \cite{Kis}, \cite{Luk1},\cite{Luk2}, \cite{Luk3},  \cite{Lad}, \cite{Li}, \cite{LL},\cite{LWXZZ}, \cite{LWZ}, \cite{TWZZ}).
\vskip .1in 
This work intends to show the $H^2(\Omega)-$stability and examine the the large-time behavior of perturbations near the hydrostatic equilibrium $(U_{he}, \Theta_{he})$ with
$$
U_{he} =0, \quad \Theta_{he} =g_0x_2.
$$
For the velocity $U_{he}$, the momentum equation is fulfilled when the pressure
gradient is balanced by the buoyancy force, namely 
$$
-\na P_{he} + g_0\Theta_{he} \, \mathbf e_2 =0 \quad\mbox{or}\quad P_{he} = \frac12g_0^2 x_2^2. 
$$
To examine the stability problem, we need first to write down the equations for the perturbation $(u, p, \theta)$, where
$$
u = U -U_{he}, \quad p= P- P_{he} \quad\mbox{and} \quad \theta = \Theta- \Theta_{he}. 
$$
It is evident from equations (\ref{y}) that $(u, p, \theta)$ satisfies the following anisotropic Boussinesq equations with vertical  dissipation and horizontal thermal diffusion
\begin{align}\label{oussama}
    \begin{cases}
      \partial_t u+u\cdot \nabla u= -\nabla p+ \nu\, \partial_{22}u+ g_0\theta \mathbf e_2, 
	\quad x\in\Omega, \,\,t>0, \\
       \partial_t \theta + u \cdot \nabla \theta + g_0u_2= \eta\, \partial_{11} \theta, \\
        \nabla \cdot u=0, \\
        u(x,0) =u_0(x), \quad \theta(x,0) =\theta_0(x).    
    \end{cases}
\end{align}
The difference between the original system (\ref{y}) and the system governing the perturbations (\ref{oussama}) is that the temperature equation in (\ref{oussama}) contains $g_0u_2$. 
Without this extra term, the $L^2$-norm of the velocity
$u$ to (\ref{y}) can grow in time due to the buoyancy forcing term $g_0\theta{\mathbf e}_2$. With even full dissipation and thermal diffusion, as taken in \cite{BrS}, solutions of the 3D Boussinesq equations can actually grow in time. This term 
in (\ref{oussama}) contributes to balancing $g_0\theta{\mathbf e}_2$ in the energy estimates. Consequently, the buoyancy forcing ceases to have a destabilizing impact in (\ref{oussama}). In cases where dissipation is degenerate and is only one-directional as in (\ref{y}), it is not clear how the
solution would behave.
\vskip .1in 
When the spacial domain is the whole space $\mathbb{R}^2$, the lack of the horizontal dissipation complicates the control of the growth of the vorticity $\om = \na\times u$, satisfying
\beq\label{vor}
\p_t \om + u\cdot\na \om = \nu\, \p_{22} \om + g_0\p_1 \theta, \quad x\in \mathbb R^2, \, \,t>0.
\eeq
More precisely, it is feasible to derive a uniform bound on the $L^2$-norm of the vorticity $\om$ itself. Nonetheless, controlling the $L^2$-norm of the gradient of the vorticity, $\na\om$, does not seem achievable. In particular, if the temperature $\theta$ is zero, (\ref{vor}) reduces to the 2D Navier-Stokes equation with degenerate dissipation, 
\beq\label{vor0}
\p_t \om + u\cdot\na \om = \nu\, \p_{22} \om, \quad x\in \mathbb R^2, \, \,t>0.
\eeq
While (\ref{vor0}) always has a unique global solution $\om$ for any initial data $\om_0\in H^1(\mathbb R^2)$, the question of whether $\|\nabla \omega(t)\|_{L^2}$ for the solution $\omega$ of (\ref{vor0}) grows with respect to $t$ remains an open problem.

Furthermore, when there is no dissipation at all, namely when $\nu=0$, (\ref{vor0}) takes the form of the 2D Euler vorticity equation 
\begin{align}
\p_t \om + u\cdot\na \om =0, \quad x\in \mathbb R^2, \, \,t>0.\label{21}
\end{align}
As pointed out in many works (see, e.g., \cite{Den}, \cite{Kis}, \cite{Zla}), $\na\om(t)$ of (\ref{21}) can grow even double exponentially in time. Particularly, the 
velocity of the 2D Euler equations in the Sobolev space $H^2$ is not stable. Conversely, solutions to the 2D Navier-Stokes equations with full dissipation 
$$
\p_t \om + u\cdot\na \om = \nu\, \Delta \om, \quad x\in \mathbb R^2, \, \,t>0
$$
decays algebraically in time, as shown by Schonbek (see, e.g., \cite{Sch0}, \cite{Sch}). The abscence of the horizontal dissipation in (\ref{vor0}) hinders our ability to follow the approach used for the fully dissipative 
Navier-Stokes equations. Specifically, when applying the energy method to bound $\|\na \om(t)\|_{L^2}$, namely
\begin{align}
\frac12\,\frac{d}{dt} \|\na \om(t)\|_{L^2}^2 +  \nu \|\p_2 \na \om(t)\|_{L^2}^2 = - \int \na \om \cdot\na u\cdot \na \om\, dx, \label{(1.4)}
\end{align}
the one-directional dissipation is not enough to control the nonlinearity. The challenge lies in acquiring a suitable upper bound for the term on the right-hand side of (\ref{(1.4)}). To effectively leverage the anisotropic dissipation, we naturally decompose this term further into four component terms.
\ben
\int \na \om \cdot\na u\cdot \na \om\, dx &=& \int \p_1 u_1 \, (\p_1 \om)^2\,dx + \int \p_1 u_2\, \p_1 \om\, \p_2 \om\,dx \label{hard}\\
&& + \int \p_2 u_1 \, \p_1 \om\, \p_2 \om\,dx + \int \p_2 u_2\, (\p_2\om)^2\,dx. \notag
\een 
Without horizontal dissipation, establishing a time-integrable upper bound for the first two terms in (\ref{hard}) is not possible. 

\vskip .1in 
When dealing with the stability problem on (\ref{oussama}), we come across the same nonlinear term presented in (\ref{hard}). Fortunately, the smoothing and stabilization effect of the temperature through the coupling and interaction allows us to solve the stability problem in (\ref{oussama}). To reveal these effects, we start by eliminating the 
pressure term in (\ref{oussama}). Applying 
the Helmholtz-Leray projection $\mathbb P = I - \na\Delta^{-1}\na\cdot$ to the velocity equation in (\ref{oussama}), we get 
\beq\label{j8}
\p_t u = \nu \p_{22} u + \mathbb P (g_0\theta \mathbf e_2) -\mathbb P(u\cdot\na u).
\eeq
Using the definition of the Leray projection $\mathbb P$, 
\beq\label{j9}
\mathbb P (g_0\theta \mathbf e_2) = g_0\theta \mathbf e_2 - \na \Delta^{-1} \na\cdot (g_0\theta \mathbf e_2) = g_0\left[ \begin{array}{c} -\p_1 \p_2 \Delta^{-1} \theta \\ \theta - \p_2^2 \Delta^{-1} \theta\end{array}\right].
\eeq
Substituting (\ref{j9}) into (\ref{j8}) and expressing (\ref{j8}) in terms of its component equations, yields 
\beq\label{newu}
\begin{cases}
	\p_t u_1 = \nu \, \p_{22} u_1 - g_0\p_1\p_2 \Delta^{-1} \theta + N_1, \\
	\p_t u_2 = \nu \, \p_{22} u_2 + g_0\p_1\p_1 \Delta^{-1} \theta + N_2,
\end{cases}
\eeq
where $N_1$ and $N_2$ represent the nonlinear terms, 
$$
N_1 = -( u\cdot\na u_1 -\p_1 \Delta^{-1} \na\cdot(u\cdot\na u)), \quad 
N_2 = -( u\cdot\na u_2 -\p_2 \Delta^{-1} \na\cdot(u\cdot\na u)).
$$
Differentiating the first equation of (\ref{newu}) with respect to $t$, we get
$$
\p_{tt} u_1 = \nu \p_{22} \p_t u_1- g_0\p_1\p_2 \Delta^{-1} \p_t \theta + \p_t N_1.
$$
Using the equation of $\theta$ in (\ref{oussama}), we substitute $\p_t \theta$ in the above equation with $\eta \, \p_{11} \theta - g_0u_2 - u\cdot\na \theta$ to write
$$
\p_{tt} u_1 = \nu \p_{22} \p_t u_1 + g_0^2\p_1\p_2 \Delta^{-1} u_2 -g_0\eta \, \p_{11} \p_1\p_2\Delta^{-1} \theta + g_0\p_1 \p_2 \Delta^{-1}(u\cdot\na \theta) + \p_t N_1.
$$
Additionally, replacing $g_0\p_1\p_2\Delta^{-1} \theta$ by the first component equation of (\ref{newu}), namely
$$
- g_0\p_1\p_2 \Delta^{-1} \theta = \p_t u_1 -\nu \,\p_{22} u_1-N_1,
$$
we find 
\beno
\p_{tt} u_1 &=& \nu \p_{22} \p_t u_1 + g_0^2\p_1\p_2 \Delta^{-1} u_2 + \eta \, \p_{11} (\p_t u_1 -\nu \,\p_{22} u_1-N_1)\\
&& + g_0\,\p_1 \p_2 \Delta^{-1}(u\cdot\na \theta) + \p_t N_1,
\eeno
which in turn gives, due to the divergence-free condition $\p_2 u_2 = -\p_1 u_1$, 
\beq\label{u1t}
\p_{tt} u_1 -( \eta \p_{11} + \nu \p_{22})  \p_t u_1 + \nu \eta \p_{11} \p_{22} u_1 + g_0^2\p_{11}\Delta^{-1} u_1 = N_3,
\eeq
where $N_3$ is the nonlinear term, 
$$
N_3=(\p_t  -\eta \p_{11}) N_1 + g_0\p_1 \p_2 \Delta^{-1}(u\cdot\na \theta).
$$
Following the same procedure, we can easily show that $u_2$ and $\theta$ satisfy 
\ben
&& \p_{tt} u_2 -( \eta \p_{11} + \nu \p_{22})  \p_t u_2 + \nu \eta \p_{11} \p_{22} u_2 + g_0^2\p_{11}\Delta^{-1} u_2 = N_4, \label{u2t}\\
&& \p_{tt} \theta -( \eta \p_{11} + \nu \p_{22})  \p_t \theta + \nu \eta \p_{11} \p_{22} \theta + g_0^2\p_{11}\Delta^{-1} \theta = N_5 \notag
\een
with 
\beno
&& N_4 = (\p_t -\eta \p_{11}) N_2 - g_0\p_1 \p_1 \Delta^{-1}(u\cdot\na \theta),\\
&& N_5 = -(\p_t-\nu \p_{22})(u\cdot\na \theta) -g_0N_2.
\eeno 
Then, merging (\ref{u1t}) and (\ref{u2t}) and expressing them into the velocity vector form, we have reformulated (\ref{oussama}) into the following new system 
\beq\label{bb22}
\begin{cases}
	\p_{tt} u -( \eta \p_{11} + \nu \p_{22})  \p_t u + \nu \eta \p_{11} \p_{22} u + g_0^2\p_{11}\Delta^{-1} u = N_6,\\
	\p_{tt} \theta -( \eta \p_{11} + \nu \p_{22})  \p_t \theta + \nu \eta \p_{11} \p_{22} \theta + g_0^2\p_{11}\Delta^{-1} \theta = N_5,
\end{cases}
\eeq
where 
$$
N_6=(N_3, N_4) = -(\p_t -\eta \p_{11})\mathbb P(u\cdot\na u) + g_0\na^\perp \p_1 \Delta^{-1}(u\cdot\na \theta)
$$
with $\na^\perp =(\p_2, -\p_1)$. By applying the curl of the velocity equation, we can likewise transform (\ref{bb2}) into a system of $\om$ and $\theta$, 
$$
\begin{cases}
	\p_{tt} \om -( \eta \p_{11} + \nu \p_{22})  \p_t \om + \nu \eta \p_{11} \p_{22} \om + g_0^2\p_{11}\Delta^{-1} \om = N_7,\\
	\p_{tt} \theta -( \eta \p_{11} + \nu \p_{22})  \p_t \theta + \nu \eta \p_{11} \p_{22} \theta + g_0^2\p_{11}\Delta^{-1} \theta = N_5,
\end{cases}
$$
with
$$
N_7 = - (\p_t -\eta \p_{11})(u\cdot\na \om) -g_0\p_1 (u\cdot\na \theta).
$$
Remarkably, we observe that all physical quantities $u, \theta$ and $\om$ obey the same damped degenerate wave 
equation, differing only in their respective nonlinear terms. Compared to the original system of $(u,\theta)$ in (\ref{oussama}), the wave equations (\ref{bb22}) reveal the underlying smoothing and stabilization hidden in (\ref{oussama}). In (\ref{oussama}), where horizontal dissipation is absent in the velocity field, the wave structure implies that the temperature can stabilize the fluids by creating the horizontal regularization through the coupling and interaction. By taking advantage of these effects, the stability problem on (\ref{oussama}) was recently established by Ben Said and al in \cite{BPW} when the spacial domain is the whole plane $\mathbb{R}^2$. However, the large time behaviour of the solution in  $\mathbb{R}^2$ remains a mystery. When the spatial domain is $\Omega = \mathbb{T}\times\mathbb{R},$ this paper also proves the stability of (\ref{oussama}). Additionally, we analyze the large-time behavior of the solutions.
The core idea involves breaking down a physical quantity into its horizontal average and the associated oscillation. Specifically, for a function $f = f(x_1, x_2)$ defined 
on $\mathbb{T}\times \mathbb{R}$ and integrable in $x_1$ over the 1D periodic box $\mathbb{T} = [0,1]$, we define its horizontal average $\overline{f}$ by
\begin{align}
	\overline{f}(x_2)=\int_{\mathbb{T}}f(x_1,x_2)dx_1,  \label{b}
\end{align}
and we write,
\begin{align}
	f=\overline{f}+\widetilde{f}.\label{bbs}
\end{align}
Note here that, the horizontal average $\overline{f}$ corresponds to the zeroth Fourier mode of f while $\widetilde{f}$
contains all non-zero Fourier modes. 

The decomposition (\ref{bbs}) possesses distinct properties. To begin with, this decomposition is orthogonal
in the Sobolev space $H^k(\Omega)$ for any 
non-negative integer. This implies that the $H^k-$norms of $\overline{f}$ and $\widetilde{f}$
are bounded by the $H^k-$norm of $f$. Furthermore, this decomposition commutes with derivatives,
and $\overline{f}$ and $\widetilde{f}$ of a divergence-free vector field $f$ are also divergence-free. An essential property to
be frequently used in our estimates is that $\widetilde f$ admits strong versions of the Poincar\'{e} type 
inequality
$$
\|\widetilde{f}\|_{L^2(\Omega)} \le C\, \|\p_1 \widetilde{f}\|_{L^2(\Omega)}, \quad \|\widetilde{f}\|_{L^\infty(\Omega)} \le C\, \|\p_1 \widetilde{f}\|_{H^1(\Omega)}.
$$ 
Applying this decomposition to the velocity field $u$ and the temperature $\theta$, namely writing $$
u=\overline{u}+\widetilde{u}, \quad \theta=\overline{\theta}+\widetilde{\theta}
$$ and exploiting the aforementioned properties we can effectively handle the nonlinear terms in (\ref{hard}) appropriately, even when there is only vertical dissipation.
More precisely, the following theorems hold. Theorem \ref{TH} establishes the $H^2$-stability while Theorem \ref{TH1} specifies the decay rates of the oscillation part $(\widetilde u, \widetilde \theta)$. 
\begin{thm}\label{TH}
	Let $\mathbb{T}=[0,1]$ be a 1D periodic box and let $\Omega=\mathbb{T}\times\mathbb{R}$. Assume $u_0,\theta_0\in H^2(\Omega)$ and $\nabla\cdot u_0 = 0.$ Then there exists $\varepsilon = \varepsilon(\nu,\eta)>0$ such that, if 
	\begin{align*}
		\|u_0\|_{H^2}+\|\theta_0\|_{H^2}\le \varepsilon,
	\end{align*}
	then (\ref{oussama}) has a unique global solution $(u, \theta)$ that remains uniformly bounded for all time, for any $t\ge0$,
	\begin{align*}
		\|u(t)\|_{H^2}^2+\|\theta(t)\|_{H^2}^2&+2\nu\int_0^t\|\partial_{2}u(\tau)\|_{H^2}^2d\tau\\&+2\eta\int_0^t\|\partial_{1}\theta(\tau)\|_{H^2}^2d\tau + C(\nu,\eta)\int_0^t\|g_0\partial_1u_2\|_{L^2}^2d\tau\le C\varepsilon^2
	\end{align*}
where $C(\nu,\eta)$ and $C>0$ are constants.
\end{thm}
Theorem \ref{TH} states that any small initial perturbation, in the $H^2$-sense, leads to a unique global, in time, solution of (\ref{oussama}) that remains small in $H^2$ for all time t. Furthermore, it implies that the time-integral of $\|\partial_1u_2(\tau)\|_{L^2}^2$ is finite.
 
 \vskip .1in 
The following Theorem asserts that the oscillation portion $(\widetilde{u},\widetilde{\theta})$ decays to zero algebraically in time in the $H^1$-norm. This result aligns with the stratification phenomenon of buoyancy
driven fluids. Additionally, it affirms the observation derived from the numerical simulations presented in \cite{DWZZ}, indicating that the temperature becomes horizontally
homogeneous and stratify in the vertical direction over time.
\begin{thm}\label{TH1} Let $u_0, \theta_0\in H^{2}(\Omega)$ with $\nabla\cdot u_0=0\,.$ Assume that $(u_0,\theta_0)$ satisfies \begin{align*}
		\|u_0\|_{H^2}+\|\theta_0\|_{H^2}\le \varepsilon,
	\end{align*}for sufficiently small $\varepsilon>0$. Let $(u,\theta)$ be the corresponding solution of (\ref{oussama}) with $g_0$ negative constant. Then the oscillation part $(\widetilde{u},\widetilde{\theta})$ satisfies the following algebraic decay in time,
	\begin{align*}
		\|\widetilde{u}\|_{H^1}+\|\widetilde{\theta}\|_{H^1} \le c (1 +t)^{-\frac{1}{2}},
	\end{align*}
	for some constant $c>0$ and for all $t\ge0.$ In addition, $(\widetilde{u},\widetilde{\theta})$ has
	the asymptotic behavior, as $t\to \infty$, 
	$$
	t\, (\|\widetilde{u}(t)\|_{H^1}^2+\|\widetilde{\theta}(t)\|_{H^1}^2) \to 0. 
	$$
\end{thm}
\vskip .1in 
According to Theorem \ref{TH1}, the solution $(u,\theta)$ of (\ref{oussama}) approaches its horizontal average $(\overline{u},\overline{\theta})$ asymptotically, and eventually, the Boussinesq equations (\ref{oussama}) evolves to the following 1D system
$$
\begin{cases}
	\partial_t\overline{u}+\overline{u\cdot\nabla\widetilde{u}}+\begin{pmatrix}
		0\\\partial_2\overline{p}
	\end{pmatrix}=g_0\begin{pmatrix}
		0\\\overline{\theta}
	\end{pmatrix}+\nu\partial_2^2\overline{u},\\\partial_t\overline{\theta}+\overline{u\cdot\nabla\widetilde{\theta}}=0.
\end{cases}
$$
\vskip .1in 

We briefly outline the proofs for Theorem \ref{TH} and Theorem \ref{TH1}. As the local, in time, well-posedness
on (\ref{oussama}) in the Sobolev setting $H^2(\Omega)$ can be established using standard approaches such as Friedrichs' Fourier cutoff (see, e.g.,
\cite{MaBe}), the proof of Theorem \ref{TH} is essentially reduced to demonstrating the global, in time, \textit{a
priori} bound on the solution in $H^2(\Omega)$. To do so, we make use of the bootstrapping argument (see \cite{Tao}, p 20). To set it up, we introduce the following energy functional for the
$H^2$-solution,
\ben
E(t) &=& \max_{0\leq \tau \leq t} (\|u(\tau)\|_{H^2}^{2} +\|\theta(\tau)\|_{H^2}^{2})  + 2 \nu \int_{0}^{t} \|\partial_2 u\|_{H^2}^{2}d\tau \notag\\
&& + 2\eta \int_{0}^t \|\partial_1 \theta\|_{H^2}^{2} d\tau + \delta \int_{0}^{t}\|g_0\partial_1 u_2\|_{L^2}^{2}\, d\tau, \label{ee}
\een
where $\delta>0$ is a suitably selected small parameter. Our central objective, is to show that, for a constant $C$ uniform and for all $t > 0$,
\beq\label{ine}
E(t) \le C\, E(0) + C\, E(t)^{\frac32}.
\eeq
To prove (\ref{ine}), we should make full use of the extra regularization resulting from the wave structure in (\ref{bb22}). Furthermore, the 
control on the time integral of the horizontal derivative of the velocity field, namely
\beq\label{cru}
\int_0^t \|g_0\p_1 u_2(\tau)\|_{L^2}^2\,d\tau 
\eeq
plays an improtant role our proof. Note here, that the uniform boundedness of (\ref{cru}) is 
not a consequence of the vertical dissipation in the velocity equation but due to the interaction
with the temperature equation. In fact, using the equation of $\theta$ in (\ref{ff}), we represent $g_0\p_1 u_2$ as,
$$
g_0\partial_1 u_2= -\partial_t \partial_1 \theta - \partial_1 (u \cdot \nabla \theta) + \eta \partial_{111} \theta, 
$$
then 
$$
\|g_0\partial_1 u_2\|_{L^2}^{2} = -g_0\int \partial_t \partial_1 \theta\,  \partial_1 u_2\;dx - g_0\int \partial_1 u_2\, \partial_1 (u \cdot \nabla \theta)\;dx + \eta g_0\int \partial_1 u_2\,\partial_{111} \theta\;dx. 
$$
Hence, the time integrability of $\|g_0\partial_1 u_2\|_{L^2}^{2}$ is converted to the time integrability 
of other terms. This phenomenon of extra regularization and time integrability, resulting from the coupling, is also observed in other models of partial differential equations, such as 
the Oldroyd-B system (see \cite{ConW}, \cite{ER}). 

Once (\ref{ine}) is proven, the bootstrapping argument then gives that, if  
$$
E(0)=\|(u_0, \theta_0)\|_{H^2}^2 \le \varepsilon^2
$$
for some sufficiently small $\varepsilon>0$, then $E(t)$ remains uniformly small for all time, 
namely 
\begin{align}
E(t)  \le C\, \varepsilon^2\label{E(t)}
\end{align}
for a constant $C>0$ and for all $t\ge 0$. In particular, (\ref{E(t)}) yields the desired global $H^2$-bound on the
solution $(u,\theta)$.
We leave details on the application of the bootstrapping argument in the proof of Theorem \ref{TH} in Section \ref{ssp}. 
\vskip .1in 
To demonstrate the algebraic decay rates on the $H^1$-norm of the oscillation component, as stated in Theorem \ref{TH1}, we  initially take the difference 
of (\ref{oussama}) and its horizontal average, to write down the system
governing the oscillation part $(\widetilde u, \widetilde \theta)$ 
\begin{align}\label{pp}
	\begin{cases}
		\partial_t\widetilde{u}+\widetilde{u\cdot\nabla\widetilde{u}}+\widetilde{u_2}\partial_2\overline{u}-\nu\partial_{2}^2\widetilde{u}+\nabla \widetilde{p}=g_0\widetilde{\theta}e_2,\\\partial_t\widetilde{\theta}+\widetilde{u\cdot\nabla\widetilde{\theta}}+\widetilde{u_2}\partial_2\overline{\theta}-\eta\partial_{1}^2\widetilde{\theta}+ g_0\widetilde{u_2}=0.
	\end{cases}
\end{align}
Controling the $H^1$-norm of $(\widetilde{u},\widetilde{\theta})$ naturally involves estimating the $L^2-$norms $\|(\widetilde u, \widetilde \theta)\|_{L^2}$ 
and $\|(\na \widetilde u, \na\widetilde \theta)\|_{L^2}$. Here, one major difficulty is that the 
equation of $\widetilde u$ has only vertical dissipation, however, the aforementioned 
Poincar\'{e} inequality can only bound a function in terms of its horizontal derivatives. Consequently, some of the  nonlinear parts  associated with $\widetilde u$ 
can not be bounded suitably and require the upper bounds involving $\|\widetilde u_2\|_{L^2}$. 
To handle these terms, we seek extra smoothing and stabilizing effect on $\widetilde u_2$ 
by exploiting the coupling in (\ref{pp}). Specifically, we introduce the following extra term along 
with the $H^1$-norm to form a Lyapunov functional, 
$$
-\delta (\widetilde u_2, \widetilde \theta),
$$
where $\delta>0$ is a small constant and $(\widetilde u_2, \widetilde \theta)$ denotes the $L^2$-inner 
product. The time derivative of this terme produces $\delta \|\widetilde u_2\|_{L^2}^2$,
which help balance $\|\widetilde u_2\|_{L^2}^2$ from the nonlinearity. Then, applying anisotropic inequalities presented in section \ref{pre}, we demonstrate the following energy inequality.
\begin{align*}
	\frac{d}{dt}\Big(\|\widetilde{u}\|_{H^1}^2 +\|\widetilde{\theta}\|_{H^1}^2-\delta(\widetilde{u_2},\widetilde{\theta})\Big)+\nu\|\partial_2\widetilde{u}\|_{H^1}^2+\eta\|\partial_1\widetilde{\theta}\|_{H^1}^2+\frac{\delta}{4}\|\widetilde{u_2}\|_{L^2}^2\le 0, 
\end{align*}
resulting the desired algebraic decay stated in Theorem \ref{TH1}. More details are given in Section \ref{decay}.
\vskip .1in 
The subsequent sections are organized as follows: Section \ref{pre} presents various anisotropic inequalities and some crucial properties related to the 
orthogonal decomposition, including the Poincar\'{e} type inequality for the oscillation portion $\widetilde f$. Section \ref{ssp} is dedicated to the proof of Theorem \ref{TH} and Section \ref{decay} proves Theorem \ref{TH1}.
 
\vskip .3in 
\section{Preliminaries} 
\label{pre}
This Section serves as preparation for the proof of Theorems \ref{TH} and \ref{TH1}.
Lemma \ref{l1} through Lemma \ref{l4} provide several frequently used facts on the orthogonal decomposition. While Lemma \ref{special5} presents a precise decay rate for a nonnegative integrable function,
which is also monotonic in a generalized sense.

\vskip .1in 
We start first, by presenting some basic properties of $\overline{f}$ and $\widetilde{f}$. 

\begin{lem}\label{l1}
Let $\Omega=\mathbb{T}\times \mathbb{R}$. Assume that $f$ defined on $\Omega$ is sufficiently
regular, say $f\in H^2(\Omega).$ Let  $\overline{f}$ and $\widetilde{f}$ be defined as in (\ref{b}) and (\ref{bbs}). Then 
\\
(a) The average operator $\overline{f}$ and the oscillation operator $\widetilde{f}$ commute with 
partial derivatives, 
\begin{align*}
\overline{\partial_1f}=\partial_1\overline{f}=0,\quad \overline{\partial_2f}=\partial_2\overline{f},\quad \widetilde{\partial_1f}=\partial_1\widetilde{f}, \quad  \widetilde{\partial_2f}=\partial_2\widetilde{f}, \quad \overline{\widetilde{f}}=0. 
\end{align*}
(b) If $f$ is a divergence-free vector field, namely $\nabla\cdot f=0,$ then $\overline{f}$ and $\widetilde{f}$ are also divergence-free,
\begin{align*}
\nabla\cdot \overline{f}=0\,\quad\text{and}\quad\nabla\cdot \widetilde{f}=0.
\end{align*}
(c) $\overline{f}$ and $\widetilde{f}$ are orthogonal in $\dot H^k$ for any integer $k\ge 0$, namely
\begin{align*}
(\overline{f},\widetilde{f})_{\dot H^k(\Omega)}:=\int_{\Omega}\overline{D^k f}\cdot \widetilde{D^k f}dx=0,\quad\|f\|_{\dot H^k(\Omega)}^2=\|\overline{f}\|_{\dot H^k(\Omega)}^2+\|\widetilde{f}\|_{\dot H^k(\Omega)}^2.
\end{align*}
In particular,
\begin{align*}
\|\overline{f}\|_{\dot H^k(\Omega)}\le \|f\|_{\dot H^k(\Omega)}\,\quad\text{and}\quad \|\widetilde{f}\|_{\dot H^k(\Omega)}\le \|f\|_{\dot H^k(\Omega)}.
\end{align*} 
The orthogonality is actually more general and holds for any integrable functions, 
$$
\int_{\Omega} \overline{f} \cdot \widetilde g \,dx =0. 
$$
\end{lem}
Lemma \ref{l1} can be proven easily using the definition of $\overline{f}$ and $\widetilde{f}$.

\vskip .1in 
The next Lemma compares the 1D Sobolev inequalities on the whole line $\mathbb R$ and 
on bounded domains. 

\begin{lem}\label{delta216}
For any 1D function $f \in H^1(\mathbb{R})$,
\begin{align*}
\|f\|_{L^{\infty}(\mathbb{R})} \le \sqrt{2} \,\|f\|_{L^{2}(\mathbb{R})}^{\frac{1}{2}}\,\|f'\|_{L^{2}(\mathbb{R})}^{\frac{1}{2}}.
\end{align*}
For any bounded domain such as $\mathbb{T} = [0,1]$ and $f \in H^1(\mathbb{T})$,
\begin{align*}
\|f\|_{L^{\infty}(\mathbb{T})} \le \sqrt{2} \,\|f\|_{L^{2}(\mathbb{T})}^{\frac{1}{2}}\,\|f'\|_{L^{2}(\mathbb{T})}^{\frac{1}{2}} + \|f\|_{L^{2}(\mathbb{T})},
\end{align*}
in particular, if the function f has mean zero such as the oscillation part $\tilde{f}$, 
\begin{align*}
\|f\|_{L^{\infty}(\mathbb{T})} \le C \,\|f\|_{L^{2}(\mathbb{T})}^{\frac{1}{2}}\,\|f'\|_{L^{2}(\mathbb{T})}^{\frac{1}{2}}.
\end{align*}
\end{lem}

\vskip .1in 
The following lemma presents anisotropic upper bounds for triple products as 
well as for the $L^\infty$-norm on the domain $\Omega$. Anisotropic Sobolev inequalities are powerful tools for dealing with anisotropic models. The whole space version of these type of inequalities has previously been 
used in \cite{CaoWu} in the 2D cases and in \cite{WuZhu} in the 3D case. 

\begin{lem}\label{l2}Let $\Omega=\mathbb{T}\times \mathbb{R}.$ For any $f,g,h\in L^2(\Omega)$ with $\partial_1f\in L^2(\Omega)$ and $\partial_2g\in L^2(\Omega),$ then
\begin{align} \label{uu}
\Big|\int_{\Omega}fgh\,dx\Big|\le C\|f\|_{L^2}^{\frac12}(\|f\|_{L^2}+\|\partial_1f\|_{L^2})^{\frac12}\|g\|_{L^2}^{\frac12}\|\partial_2g\|_{L^2}^{\frac12}\|h\|_{L^2}\,.
\end{align}
For any $f \in H^2(\Omega),$ we have
\begin{align*}
\|f\|_{L^{\infty}(\Omega)}\le &C\|f\|_{L^2}^{\frac14}(\|f\|_{L^2}+\|\partial_1f\|_{L^2})^{\frac14}\|\partial_2f\|_{L^2}^{\frac14}\nonumber\\&\times(\|\partial_2f\|_{L^2}+\|\partial_1\partial_2f\|_{L^2})^{\frac14}.
\end{align*}
\end{lem}

\vskip .1in 
Replacing $f$ in Lemma \ref{l2} by its oscillation portion $\widetilde{f}$, the lower-order part in (\ref{uu}) can be dropped, as presented in the next Lemma.
\begin{lem}\label{l3}Let $\Omega=\mathbb{T}\times \mathbb{R}.$ For any $f,g,h\in L^2(\Omega)$ with $\partial_1f\in L^2(\Omega)$ and $\partial_2g\in L^2(\Omega),$ then
\begin{align}
\Big|\int_{\Omega}\widetilde{f}gh\,dx\Big|\le C\|\widetilde{f}\|_{L^2}^{\frac12}\|\partial_1\widetilde{f}\|_{L^2}^{\frac12}\|g\|_{L^2}^{\frac12}\|\partial_2g\|_{L^2}^{\frac12}\|h\|_{L^2}.\label{a}
\end{align}
For any $f\in H^2(\Omega)$, we have
\begin{align*}
\|\widetilde{f}\|_{L^{\infty}(\Omega)}\le C\|\widetilde{f}\|_{L^2}^{\frac14}\|\partial_1\widetilde{f}\|_{L^2}^{\frac14}\|\partial_2\widetilde{f}\|_{L^2}^{\frac14}\|\partial_1\partial_2\widetilde{f}\|_{L^2}^{\frac14}\,.
\end{align*}
\end{lem}
The subsequent Lemma states that the oscillation component $\widetilde{f}$ verifies a strong Poincaré type inequality with the upper bound expressed in terms of $\partial_1\widetilde{f}$ rather than $\nabla\widetilde{f}$.
\begin{lem}\label{l4}Let  $\overline{f}$ and $\widetilde{f}$ be defined as in (\ref{b}) and (\ref{bbs}). If $\|\partial_1\widetilde{f}\|_{L^2(\Omega)}<\infty$, then 
\begin{align*}
\|\widetilde{f}\|_{L^2(\Omega)}\le C\|\partial_1\widetilde{f}\|_{L^2(\Omega)},
\end{align*}
where $C$ is a pure constant. In addition, if $\|\partial_1\widetilde{f}\|_{H^1(\Omega)}<\infty$, then
\begin{align*}
\|\widetilde{f}\|_{L^{\infty}(\Omega)}\le C\|\partial_1\widetilde{f}\|_{H^1(\Omega)}.
\end{align*}
\end{lem}

\vskip .1in
As a direct consequence of Lemma \ref{l4} and the inequality  (\ref{a}), one has
\begin{align}\label{ohno}
\Big|\int_{\Omega}\widetilde{f}gh\,dx\Big|\le C\|\partial_1\widetilde{f}\|_{L^2}\|g\|_{L^2}^{\frac12}\|\partial_2g\|_{L^2}^{\frac12}\|h\|_{L^2}.
\end{align}
We refer the readers to \cite{DWXZ} for detailed proofs of Lemmas \ref{l1}, \ref{l2}, \ref{l3} and \ref{l4}.
\vskip .1in 
The last lemma precises an explicit decay rate in (\ref{special2}) for functions 
that are integrable and are decreasing in a general sense, namely (\ref{special1}).

\begin{lem}\label{special5}
Let $f = f(t)$ be a nonnegative function satisfying , for two constants $C_0 > 0$ and $C_1 > 0$,
\begin{align}\label{special1}
\int_0^{\infty} f(\tau) d\tau < C_0 \quad \text{and} \quad f(t) \le C_1 f(s) \quad \text{for any}\quad 0 \le s <t.
\end{align}
Then, for $C_2 = \max\lbrace2C_1f(0), 4C_0C_1\rbrace$ and for any $t>0$,
\begin{align}\label{special2}
f(t) \le C_2(1+t)^{-1}.
\end{align}
Furthermore, $f(t)$ has the following large-time asymptotic behavior, 
$$
\lim_{t\to +\infty} t \, f(t) =0. 
$$
\end{lem}
A detailed proof of Lemma \ref{special5} can be found in \cite{LWZ}.
\section{The $H^2$ Nonlinear Stability} 
\label{ssp}
This section proves Theorem. \ref{TH}.
\vskip .1in 
\begin{proof}
The proof is naturally divided into two major parts. The first part is for the existence,
while the second part is for the uniqueness of solutions to (\ref{oussama}).\\

To prove the global existence of solutions, it suffices to establish the energy in-
equality in (\ref{ine}) with $E(t)$ being defined in (\ref{ee}). 
This process consists of two main parts. The first is to estimate the $H^2$-norm of $(u, \theta)$ while the second is to estimate $\|\p_1 u_2\|_{L^2}^2$ and its time integral. 

\vskip .1in 
Note that, for a divergence-free vector field $u$, namely $\na\cdot u=0$, we have 
$$
\|\na u\|_{L^2} = \|\om\|_{L^2}, \quad \|\Delta u\|_{L^2} = \|\na \om\|_{L^2},
$$
where $\om=\na\times u$ is the vorticity. Then, the $H^2$-norm of $u$ is equivalent to the 
sum of the $L^2$-norms of $u$, $\om$ and $\na \om$. \\ Taking the $L^2$-inner product of $(u, \theta)$
with the first two equations in (\ref{oussama}), we find that the $L^2$-norm of $(u, \theta)$ obeys
\ben
&&\|u(t)\|_{L^2}^{2} +\|\theta(t)\|_{L^2}^{2}  + 2 \nu \int_{0}^{t} \|\partial_2 u(\tau)\|_{L^2}^{2}\, d\tau
+ 2\eta \int_{0}^t \|\partial_1 \theta(\tau)\|_{L^2}^{2} d\tau \notag\\
&& = \|u_0\|_{L^2}^{2} +\|\theta_0\|_{L^2}^{2}.\label{l2estimate}
\een
Next, we estimate the $L^2$-norm of $(\om, \na \theta)$. We make use of the vorticity equation 
and the temperature equation, 
\begin{equation}\label{vorticityequation}
\begin{aligned}
&\partial_t \omega +u\cdot \nabla \omega=  \nu \partial_{22}\omega + g_0\partial_1 \theta, \\
&\partial_t \theta + u \cdot \nabla \theta + g_0u_2= \eta \partial_{11} \theta.
\end{aligned}
\end{equation}
Dotting the equations of $\om$ and $\na\theta$ by $(\om, \na\theta)$, yields
\begin{equation}
\begin{aligned}
\frac{1}{2}\frac{d}{dt}(\|\omega\|_{L^2}^{2}+\|\nabla \theta\|_{L^2}^{2})+ \nu \|\partial_2 \omega \|_{L^2}^{2}+ \eta  \|\partial_1 \nabla \theta \|_{L^2}^{2}= I_1 + I_2, \label{coo}
\end{aligned}
\end{equation}
where 
$$
 I_1 =g_0\int (\partial_1 \theta\, \omega- \nabla u_2\cdot \nabla \theta)\, dx, \quad I_2= -\int \nabla \theta \cdot \nabla u \cdot \nabla \theta\, dx.
$$
Then, expressing $\om$ and $u$ in terms of the stream function $\psi$, namely $\om =\Delta \psi$ and $u= \na^\perp \psi:=(-\p_2 \psi, \p_1 \psi)$, we get 
\beno
I_1 &=& g_0\int (\partial_1 \theta\, \omega- \nabla u_2\cdot \nabla \theta)\, dx 
= g_0\int (\p_1 \theta \Delta \psi - \na \p_1 \psi\cdot \nabla \theta)\, dx\\
&=& g_0\int (-\theta \, \Delta\p_1 \psi + \Delta \p_1 \psi \, \theta)\,dx =0. 
\eeno
We further write $I_2$ into four terms,
\be
I_2 &=&  -\int (\partial_1 u_1 (\partial_1 \theta)^2+ \partial_1 u_2 \partial_1 \theta \partial_2 \theta +\partial_2 u_1 \partial_1 \theta \partial_2 \theta+ \partial_2 u_2 (\partial_2 \theta)^2)\,dx\nonumber\\
&:=& I_{21} + I_{22} + I_{23} + I_{24}.\label{S}
\ee 
The key point here is to obtain
upper bounds for the terms on the right-hand side of (\ref{S}) that are time integrable. By Lemmas \ref{l1}, \ref{l3} and Young’s inequality, $I_{21}$, $I_{22}$ and $I_{23}$ can be bounded as follows
\begin{align}\label{fogai11}
I_{21}:&=-\int \partial_1u_1(\partial_1\theta)^2dx=-\int \partial_1\widetilde{u_1}(\partial_1\theta)^2dx\nonumber\\&\le c\|\partial_1 \theta\|_{L^2}^{\frac{1}{2}} \|\partial_2 \partial_1 \theta \|_{L^2}^{\frac{1}{2}} \|\partial_1 \theta\|_{L^2}^{\frac{1}{2}} \|\partial_1 \partial_1\theta \|_{L^2}^{\frac{1}{2}}\|\partial_1\widetilde{u_1}\|_{L^2}\nonumber\\
&\le c\, \|u\|_{H^2}\, \|\partial_1\theta \|_{H^2}^2,
\end{align}
\begin{align}\label{fogai12}
I_{22}:&=-\int \partial_1u_2\partial_1\theta\partial_2\theta dx=-\int \partial_1u_2\partial_1\widetilde{\theta}\partial_2\theta dx\nonumber\\&\le c\, \|\partial_1\widetilde{\theta}\|_{L^2}^{\frac12}\|\partial_1\partial_1\widetilde{\theta}\|_{L^2}^{\frac12}\|\partial_2\theta\|_{L^2}^{\frac12}\|\partial_2\partial_2\theta\|_{L^2}^{\frac12}\|\partial_1u_2\|_{L^2}\nonumber\\&\le c\, \|\partial_1\theta\|_{H^2}\|\theta\|_{H^2}\|\partial_1u_2\|_{L^2}\nonumber\\&\le c\,\|\theta\|_{H^2}\Big(\|\partial_1\theta\|_{H^2}^2+\|\partial_1u_2\|_{L^2}^2\Big),
\end{align}
\begin{align}\label{foga13}
I_{23}:&=-\int \partial_2u_1\partial_1\theta\partial_2\theta dx= -\int \partial_2u_1\partial_1\widetilde{\theta}\partial_2\theta dx \nonumber\\&\le c\, \|\partial_1\widetilde{\theta}\|_{L^2}^{\frac12} \|\partial_1\partial_1\widetilde{\theta}\|_{L^2}^{\frac12}\|\partial_2u_1\|_{L^2}^{\frac12}\|\partial_2\partial_2u_1\|_{L^2}^{\frac12}\|\partial_2\theta\|_{L^2}\nonumber\\&\le c\, \|\partial_1\theta\|_{H^2}\|\partial_2u\|_{H^2}\|\theta\|_{H^2}\nonumber\\&\le c\,\|\theta\|_{H^2}\Big(\|\partial_1\theta\|_{H^2}^2+\|\partial_2u\|_{H^2}\Big).
\end{align} 
Using the divergence-free condition $\na\cdot u=0$, integration by parts and Lemmas \ref{l1} and \ref{l4}, we obtain 
\begin{align}\label{foga14}
I_{24}:&=-\int \partial_2u_2(\partial_2\theta)^2dx\nonumber=\int \partial_1u_1(\partial_2\theta)^2dx\\& =  - 2\int  \widetilde{u_1}\, \partial_2 \theta\, \partial_1\partial_2\theta\,dx
\nonumber\\&\le  c\,  \|\partial_2 \theta\|_{L^2} \| \widetilde{u_1}\|_{L^2}^{\frac{1}{2}} \|\partial_1 \widetilde{u_1} \|_{L^2}^{\frac{1}{2}}  \| \partial_1\partial_2\theta\|_{L^2}^{\frac{1}{2}} \|\partial_2\partial_1\partial_2\theta\|_{L^2}^{\frac{1}{2}}\nonumber\\&\le  c\,  \|\theta\|_{H^2} \underbrace{ \|\partial_1 \widetilde{u_1} \|_{L^2}}_{= \|\partial_2 \widetilde{u_2} \|_{L^2}}  \| \partial_1\theta\|_{H^2}
\nonumber\\&\le  c\,  \|\theta\|_{H^2} \Big( \|\partial_2 u \|_{H^2}^2+  \| \partial_1\theta\|_{H^2}^2\Big).
\end{align}
Hence, collecting the upper bounds on $I_2$ and inserting them in (\ref{coo}), we find 
\ben
&& \frac{d}{dt}(\|\nabla u\|_{L^2}^{2}+\|\nabla \theta\|_{L^2}^{2})+ 2\nu \|\partial_2 \nabla u \|_{L^2}^{2}+ 2\eta  \|\partial_1 \nabla \theta \|_{L^2}^{2} \notag\\
&& \qquad\qquad \leq c\, \|(u,\theta)\|_{H^2}  \Big( \|\partial_2 u \|_{H^2}^2+  \| \partial_1\theta\|_{H^2}^2+\|\partial_1u_2\|_{L^2}^2\Big). \label{h1sum}
\een
Thus, integrating  (\ref{h1sum}) over $[0, t]$ and combining with $(\ref{l2estimate})$, we get 
\ben
&& \|(u, \theta)\|_{H^1}^{2}+2 \nu  \int_{0}^{t}\|\partial_2  u(s) \|_{H^1}^{2}ds + 2 \eta  \int_{0}^{t}\|\partial_1  \theta(s) \|_{H^1}^{2}ds \notag\\
&& \leq \|(u_0, \theta_0)\|_{H^1}^{2} +  \,c\, \int_0^t \|(u,\theta)\|_{H^2}  \Big( \|\partial_2 u \|_{H^2}^2+  \| \partial_1\theta\|_{H^2}^2+\|\partial_1u_2\|_{L^2}^2\Big)\,d\tau \nonumber\\
&& \le E(0) + \,c\, E(t)^{\frac{3}{2}}. \label{h1final}
\een
\vskip .1in 
To bound the $H^2$-norm of $(u,\theta)$, it then remains to control the $L^2$-norm of $(\na \om, \Delta \theta)$. Applying $\nabla$ to the first equation of (\ref{vorticityequation}) then dotting with $\nabla \omega$, and applying $\Delta$ to the second equation of (\ref{vorticityequation}) then dotting with $\Delta \theta$, we find 
\beq\label{h2energy}
\frac{1}{2}\frac{d}{dt}(\|\nabla \omega\|_{L^2}^{2} +\|\Delta \theta(t)\|_{L^2}^{2})+  \nu \|\partial_2 \nabla \omega\|_{L^2}^{2}+ \eta \|\partial_1 \Delta \theta\|_{L^2}^{2} 
= J_1 + J_2 + J_3,
\eeq
with
\beno 
&& J_1 = g_0\int (\nabla \partial_1 \theta \cdot \nabla \omega- \Delta u_2 \Delta \theta)\,dx, \\
&& J_2 = -\int \nabla \omega \cdot \nabla u \cdot \nabla \omega\, dx, \\ 
&& J_3 = - \int \Delta \theta \cdot \Delta (u\cdot \nabla \theta) dx.
\eeno 
Similarly, we need to obtaining an upper bound for  that is time integrable
for each term in (\ref{h2energy}). Writing $\om$ and $u$ in terms of the stream function $\psi$, namely $\om =\Delta \psi$ and $u= \na^\perp \psi:=(-\p_2 \psi, \p_1 \psi)$, we have 
\beno 
J_1 &=& g_0\int (\nabla \partial_1 \theta \cdot \nabla \omega- \Delta u_2 \Delta \theta)\,dx
= g_0\int (\nabla \partial_1 \theta \cdot \nabla \omega- \Delta \p_1 \psi \Delta \theta)\,dx\\
&=& g_0\int (\nabla \partial_1 \theta \cdot \nabla \omega-  \p_1 \om\, \Delta \theta)\,dx
 = g_0\int (\nabla \partial_1 \theta \cdot \nabla \omega + \p_1 \na\om\, \cdot\na \theta)\,dx\\
&=& g_0\int \p_1(\na\theta \cdot \na \om)\,dx =0. 
\eeno 
After integration by parts, we decompose $J_3$ it into four pieces,
\begin{align}
J_3 &= -\int \Delta \theta\,  \Delta u_1 \, \partial_1 \theta\, dx - \int \Delta \theta \, \Delta u_2\, \partial_2 \theta\, dx\nonumber\\
&-2\int \Delta \theta\, \nabla u_1\cdot \partial_1\nabla\theta\, dx -2\int \Delta \theta \,\nabla u_2  \cdot \partial_2 \nabla \theta\, dx\nonumber\\
&:= J_{31} + J_{32} + J_{33} + J_{34}.\label{dd3}
\end{align}
To deal with $J_{31}$, we make use of the orthogonal decompositions $u =\overline{u}+\widetilde{u}$ and  $\theta= \overline{\theta}+\widetilde{\theta}$ to write
\begin{align}\label{usaj21}
J_{31}:&=-\int \Delta\theta\Delta u_1\partial_1\theta dx=-\int \Delta\theta\Delta u_1\partial_1\widetilde{\theta} dx\nonumber\\&=-\int \Delta\theta\partial_{11} u_1\partial_1\widetilde{\theta} dx-\int \Delta\theta\partial_{22} u_1\partial_1\widetilde{\theta} dx\nonumber\\&=\int \Delta\theta\partial_{12} u_2\partial_1\widetilde{\theta} dx-\int \Delta\theta\partial_{22} u_1\partial_1\widetilde{\theta} dx\nonumber\\&:=J_{311}+J_{312}.
\end{align}
Applying Lemma \ref{l3} we obtain,
\begin{align}
J_{311}&:=\int \Delta\theta\partial_{12} u_2\partial_1\widetilde{\theta} dx\nonumber\\&\le c\|\partial_1\widetilde{\theta}\|_{L^2}^{\frac12}\|\partial_1\partial_1\widetilde{\theta}\|_{L^2}^{\frac12}\|\partial_{12}u_2\|_{L^2}^{\frac12}\|\partial_2\partial_{12}u_2\|_{L^2}^{\frac12}\|\Delta\theta\|_{L^2}\nonumber\\&\le c\|\theta\|_{H^2}\|\partial_1\theta\|_{H^2}\|\partial_2u\|_{H^2}\nonumber\\&\le c\|\theta\|_{H^2}\Big(\|\partial_1\theta\|_{H^2}^2+\|\partial_2u\|_{H^2}^2\Big),\label{s1}
\end{align}
\begin{align}\label{s2}
J_{312}&:=-\int \Delta\theta\partial_{22} u_1\partial_1\widetilde{\theta} dx\nonumber\\&\le c\|\partial_1\widetilde{\theta}\|_{L^2}^{\frac12}\|\partial_1\partial_1\widetilde{\theta}\|_{L^2}^{\frac12}\|\partial_{22}u_1\|_{L^2}^{\frac12}\|\partial_2\partial_{22}u_1\|_{L^2}^{\frac12}\|\Delta\theta\|_{L^2}\nonumber\\&\le c\|\theta\|_{H^2}\|\partial_1\theta\|_{H^2}\|\partial_2u\|_{H^2}\nonumber\\&\le c\|\theta\|_{H^2}\Big(\|\partial_1\theta\|_{H^2}^2+\|\partial_2u\|_{H^2}^2\Big).
\end{align}
Inserting the upper bounds for $J_{311}$ and $J_{312}$ in (\ref{usaj21}) yields
\begin{align}
J_{31}\le  c\|\theta\|_{H^2}\Big(\|\partial_{1}\theta\|_{H^2}^{2}+\|\partial_{2}u\|_{H^2}^{2}\Big).\label{d31}
\end{align}
To deal with $J_{32}$, we divide it first into two terms,
\begin{align}\label{DI*}
J_{32}&=-\int \Delta \theta  \Delta u_2\partial_2 \theta dx\nonumber\\&
=-\int \partial_1 \partial_1 \theta \,\Delta u_2 \,\partial_2 \theta dx -\int \partial_2 \partial_2 \theta \, \Delta u_2\, \partial_2 \theta \, dx \nonumber\\&
=-\int \partial_1 \partial_1 \theta\, \Delta u_2\partial_2\, \theta \,dx +\frac{1}{2}\int \Delta  \partial_2 u_2\, (\partial_2 \theta)^2\, dx \nonumber\\&
=-\int \partial_1 \partial_1 \theta\, \Delta u_2\,\partial_2 \theta\, dx -\frac{1}{2}\int \Delta  \partial_1 u_1\, (\partial_2 \theta)^2 dx \nonumber\\&
=-\int \partial_1 \partial_1 \theta\, \Delta u_2\,\partial_2 \theta\, dx +\int \Delta  u_1\, \partial_2 \theta\, \partial_1 \partial_2 \theta dx\nonumber\\&=J_{321}+J_{322}.
\end{align}
Invoking the decompositions of $u$ and $\theta$, we can rewrite $J_{321}$ as,
\begin{align}\label{usa104}
J_{321}&:=-\int \partial_1 \partial_1 \theta\, \Delta u_2\,\partial_2 \theta\, dx\nonumber\\&=-\int \partial_1 \partial_1 \widetilde{\theta}\, \partial_{11}\widetilde{u_2}\,\partial_2 \overline{\theta}\, dx-\int \partial_1 \partial_1 \widetilde{\theta}\, \partial_{11}\widetilde{u_2}\,\partial_2 \widetilde{\theta}\, dx-\int \partial_1 \partial_1 \widetilde{\theta}\, \partial_{22} u_2\,\partial_2 \theta\, dx \nonumber\\&:= J_{3211} +  J_{3212} +  J_{3213}.
\end{align}
The three terms in $J_{321}$ can be bounded as follows. By interation by parts, Lemma \ref{l1}, Hölder's inequality, Lemma \ref{delta216} and Young's inequality,
\begin{align}\label{1usa104*}
J_{3211}&:= -\int \partial_1 \partial_1 \widetilde{\theta}\, \partial_{11}\widetilde{u_2}\,\partial_2 \overline{\theta}\, dx\nonumber\\&=\int \partial_1 \partial_{11} \widetilde{\theta}\, \partial_{1}\widetilde{u_2}\,\partial_2 \overline{\theta}\, dx\nonumber\\&=\int_{\mathbb{R}} \partial_{2}\overline{\theta}\Big(\int_{\mathbb{T}}\partial_1 \partial_{11} \widetilde{\theta}\, \partial_{1}\widetilde{u_2} dx_1\Big)dx_2\nonumber\\&\le \int_{\mathbb{R}} |\partial_{2}\overline{\theta}|\|\partial_{1} \widetilde{u_2}\|_{L^2_{x_1}}\|\partial_1\partial_{11}\widetilde{\theta} \|_{L^2_{x_1}}dx_2\nonumber\\&\le  \|\partial_{2}\overline{\theta}\|_{L^{\infty}_{x_2}}\|\partial_{1} \widetilde{u_2}\|_{L^2_{x_2}L^2_{x_1}}\|\partial_1\partial_{11}\widetilde{\theta} \|_{L^2_{x_2}L^2_{x_1}}\nonumber\\& \le c\|\partial_{2}\overline{\theta}\|_{H^1}\|\partial_{1} \widetilde{u_2}\|_{L^2}\|\partial_1\partial_{11}\widetilde{\theta} \|_{L^2}\nonumber\\&\le c\|\theta\|_{H^2}\Big(\|\partial_{1} u_2\|_{L^2}^2+\|\partial_1\theta \|_{H^2}^2\Big).
\end{align}
By lemma \ref{l3} and then lemma \ref{l4}
\begin{align}\label{2usa104}
J_{3212}&:=-\int \partial_1 \partial_1 \widetilde{\theta}\, \partial_{11}\widetilde{u_2}\,\partial_2 \widetilde{\theta}\, dx\nonumber\\&\le c\underbrace{\|\partial_{2}\widetilde{\theta}\|_{L^2}^{\frac12}}_{\le \|\partial_1\partial_{2}\widetilde{\theta}\|_{L^2}^{\frac12}} \|\partial_1\partial_{2}\widetilde{\theta}\|_{L^2}^{\frac12}\|\partial_{11}\widetilde{\theta}\|_{L^2}^{\frac12}\|\partial_2\partial_{11}\widetilde{\theta}\|_{L^2}^{\frac12}\|\partial_{11}u_2\|_{L^2}\nonumber\\&\le c\|u\|_{H^2}\|\partial_{1}\theta\|_{H^2}^2.
\end{align}
Making use of the divergence-free condition of $u$, Lemmas \ref{l1} and \ref{l3}, we have
\begin{align}\label{3usa104}
J_{3213}&:=-\int \partial_1 \partial_1 \widetilde{\theta}\, \partial_{22} u_2\,\partial_2 \theta\, dx\nonumber\\&=-\int \partial_1 \partial_1 \widetilde{\theta}\, \partial_{21} \widetilde{u_1}\,\partial_2 \theta\, dx\nonumber\\&\le c \|\partial_{21} \widetilde{u_1}\|_{L^2}^{\frac{1}{2}} \|\partial_1\partial_{21} \widetilde{u_1}\|_{L^2}^{\frac{1}{2}}\|\partial_{11} \widetilde{\theta}\|_{L^2}^{\frac{1}{2}} \|\partial_2\partial_{11} \widetilde{\theta}\|_{L^2}^{\frac{1}{2}}\|\partial_2\theta\|_{L^2}\nonumber\\&\le c\|\theta\|_{H^2}\|\partial_{1}\theta\|_{H^2}\|\partial_{2}u\|_{H^2}\nonumber\\&\le c\|\theta\|_{H^2}\Big(\|\partial_{1}\theta\|_{H^2}^2+\|\partial_{2}u\|_{H^2}^2\Big).
\end{align}
Inserting (\ref{1usa104*}), (\ref{2usa104}) and (\ref{3usa104}) in (\ref{usa104}) we obtain
\begin{align}
J_{321}\le c\|(u,\theta)\|_{H^2}\Big(\|\partial_{1}\theta\|_{H^2}^2+\|\partial_{2}u\|_{H^2}^2+\|\partial_{1}u_2\|_{L^2}^2\Big).\label{di6}
\end{align}
We now turn to $J_{322}$. We further decompose it into two
terms,
\begin{align}
J_{322}&:=\int \Delta  u_1\, \partial_2 \theta\, \partial_1 \partial_2 \theta dx\nonumber\\&=\int \partial_{11}  \widetilde{u_1}\, \partial_2 \theta\, \partial_1 \partial_2 \widetilde{\theta} dx+\int \partial_{22}  u_1\, \partial_2 \theta\, \partial_1 \partial_2 \widetilde{\theta} dx\nonumber\\&=J_{3221}+J_{3222}.\label{di*}
\end{align}
Due to the divergence-free condition of $u$ and Lemma \ref{l3},
\begin{align}\label{di1*}
J_{3221}&:=\int \partial_{11}  \widetilde{u_1}\, \partial_2 \theta\, \partial_{12} \widetilde{\theta} dx\nonumber\\&=-\int \partial_{12}  \widetilde{u_2}\, \partial_2 \theta\, \partial_{12}  \widetilde{\theta} dx\nonumber\\&\le c \|\partial_{12}  \widetilde{u_2}\|_{L^2}^{\frac12}\|\partial_1\partial_{12}  \widetilde{u_2}\|_{L^2}^{\frac12}\|\partial_{12}  \widetilde{\theta}\|_{L^2}^{\frac12}\|\partial_2\partial_{12}  \widetilde{\theta}\|_{L^2}^{\frac12}\|\partial_2\theta\|_{L^2}\nonumber\\&\le c \|\theta\|_{H^2}\|\partial_2u\|_{H^2}\|\partial_1\theta\|_{H^2}\nonumber\\&\le c \|\theta\|_{H^2}\Big(\|\partial_2u\|_{H^2}^2+\|\partial_1\theta\|_{H^2}^2\Big).
\end{align}
By Lemma \ref{l3},
\begin{align}\label{di2*}
J_{3222}&:=\int \partial_{22}  u_1\, \partial_2 \theta\, \partial_{12} \widetilde{\theta} dx\nonumber\\&\le c \|\partial_{12}  \widetilde{\theta}\|_{L^2}^{\frac12}\|\partial_1\partial_{12}  \widetilde{\theta}\|_{L^2}^{\frac12}\|\partial_{22}  u_2\|_{L^2}^{\frac12}\|\partial_2\partial_{22}  u_2\|_{L^2}^{\frac12}\|\partial_2\theta\|_{L^2}\nonumber\\&\le c \|\theta\|_{H^2}\|\partial_2u\|_{H^2}\|\partial_1\theta\|_{H^2}\nonumber\\&\le c \|\theta\|_{H^2}\Big(\|\partial_2u\|_{H^2}^2+\|\partial_1\theta\|_{H^2}^2\Big).
\end{align}
Combining the estimates (\ref{di1*}) and (\ref{di2*}) and inserting them in (\ref{di*}) we find
\begin{align}\label{di4}
J_{322}\le c \|\theta\|_{H^2}\Big(\|\partial_2u\|_{H^2}^2+\|\partial_1\theta\|_{H^2}^2\Big).
\end{align}
Putting (\ref{di6}) and (\ref{di4}) in (\ref{DI*}) we obtain
\begin{align}
J_{32}\le  c \|(u,\theta)\|_{H^2}\Big(\|\partial_2u\|_{H^2}^2+\|\partial_1\theta\|_{H^2}^2+\|\partial_1u_2\|_{L^2}^2\Big).\label{d32}
\end{align}
The next term $J_{33}$ is naturally split into two parts, 
\begin{align}
J_{33}&:=-2\int \Delta\theta\nabla u_1\cdot\partial_1\nabla\theta dx\nonumber\\&=-2\int \Delta\theta\partial_1 u_1\partial_1\partial_1\theta dx-2\int \Delta\theta\partial_2 u_1\partial_1\partial_2\theta dx\nonumber\\&:=J_{331}+J_{332}.\label{L}
\end{align}
All terms can be bounded suitably. In fact, due to the divergence-free condition of $u$ and Lemma \ref{l3},
\begin{align}\label{L1}
J_{331}&:=-2\int \Delta\theta\partial_1 u_1\partial_{11}\theta dx\nonumber\\&=2\int \Delta\theta\partial_2 u_2\partial_1\partial_1\widetilde{\theta} dx\nonumber\\&\le c\|\partial_{11}\widetilde{\theta}\|_{L^2}^{\frac12}\|\partial_1\partial_{11}\widetilde{\theta}\|_{L^2}^{\frac12}\|\partial_{2}u_2\|_{L^2}^{\frac12}\|\partial_2\partial_{2}u_2\|_{L^2}^{\frac12}\|\Delta\theta\|_{L^2}\nonumber\\&\le c \|\theta\|_{H^2}\|\partial_2u\|_{H^2}\|\partial_1\theta\|_{H^2}\nonumber\\&\le c \|\theta\|_{H^2}\Big(\|\partial_2u\|_{H^2}^2+\|\partial_1\theta\|_{H^2}^2\Big).
\end{align}
$J_{332}$ can be bounded similarly, by $\partial_1\theta=\partial_1\widetilde{\theta}$ and Lemma \ref{l3},
\begin{align}\label{L2}
J_{332}&:=-2\int \Delta\theta\partial_2 u_1\partial_{12}\theta dx\nonumber\\&=-2\int \Delta\theta\partial_2 u_1\partial_{12}\widetilde{\theta} dx\nonumber\\&\le c\|\partial_{12}\widetilde{\theta}\|_{L^2}^{\frac12}\|\partial_1\partial_{12}\widetilde{\theta}\|_{L^2}^{\frac12}\|\partial_{2}u_1\|_{L^2}^{\frac12}\|\partial_2\partial_{2}u_1\|_{L^2}^{\frac12}\|\Delta\theta\|_{L^2}\nonumber\\&\le c \|\theta\|_{H^2}\|\partial_2u\|_{H^2}\|\partial_1\theta\|_{H^2}\nonumber\\&\le c \|\theta\|_{H^2}\Big(\|\partial_2u\|_{H^2}^2+\|\partial_1\theta\|_{H^2}^2\Big).
\end{align}
Inserting these upper bounds in (\ref{L}) we get
\begin{align}
J_{33}\le c\|\theta\|_{H^2}\Big(\|\partial_1\theta\|_{H^2}^2+\|\partial_2u\|_{H^2}^2\Big).\label{d33}
\end{align}
To estimate $J_{34}$, we first invoke the decompositions $u = \overline{u} + \widetilde{u}$, $\theta = \overline{\theta} + \widetilde{\theta}$ and Lemma \ref{l1}, to write $J_{34}$ as
\begin{align}
J_{34}&:=-2\int \Delta\theta\nabla u_2\cdot\partial_2\nabla\theta dx\nonumber\\&=-2\int (\partial_1 u_2\partial_1\partial_2\theta\Delta\theta+\partial_2 u_2\partial_2\partial_2\theta\Delta\theta) dx\nonumber\\&=-2\int \partial_1 \widetilde{u_2}\partial_1\partial_2\widetilde{\theta}\Delta\theta-2\int\partial_2 u_2\partial_2\partial_2\theta\Delta\theta dx\nonumber\\&=-2\int  \partial_1\widetilde{u_2}\partial_1\partial_2\widetilde{\theta}\partial_{11}\theta dx-2\int  \partial_1\widetilde{u_2}\partial_1\partial_2\widetilde{\theta}\partial_{22}\theta dx-2\int\partial_2 u_2\partial_2\partial_2\theta\Delta\theta dx\nonumber\\&:=J_{341}+J_{342}+J_{343}.\label{24}
\end{align}
We start with $J_{341}$. By integration by parts, Lemmas \ref{l1} and \ref{l3} we have 
\begin{align}
J_{341}&:=-2\int  \partial_1\widetilde{u_2}\partial_1\partial_2\widetilde{\theta}\partial_{11}\theta dx\nonumber\\&=2\int  \widetilde{u_2}\partial_1\partial_1\partial_2\widetilde{\theta}\partial_{11}\widetilde{\theta} dx\nonumber\\&\le c\|\partial_{11}\widetilde{\theta}\|_{L^2}^{\frac12}\|\partial_1\partial_{11}\widetilde{\theta}\|_{L^2}^{\frac12}\|\partial_{12}\widetilde{\theta}\|_{L^2}^{\frac12}\|\partial_2\partial_{11}\widetilde{\theta}\|_{L^2}^{\frac12}\|\partial_1\widetilde{u_2}\|_{L^2}\nonumber\\&\le c\|u\|_{H^2}\|\partial_1\theta\|_{H^2}^2.\label{241}
\end{align}
Using the decomposition $\theta = \overline{\theta} + \widetilde{\theta}$ we write $J_{342}$ as,
\begin{align}\label{11usa104}
J_{342}&:=-2\int  \partial_1\widetilde{u_2}\partial_1\partial_2\widetilde{\theta}\partial_{22}\theta dx\nonumber\\&=-2\int  \partial_1\widetilde{u_2}\partial_1\partial_2\widetilde{\theta}\partial_{22}\overline{\theta} dx-2\int  \partial_1\widetilde{u_2}\partial_1\partial_2\widetilde{\theta}\partial_{22}\widetilde{\theta} dx\nonumber\\&:=J_{3421}+J_{3422}.
\end{align}
We start with $J_{3421}$. Due to integration by parts, Lemma \ref{l1}, Hölder's inequality, Lemma \ref{delta216} and Young's inequality,
\begin{align}\label{1usa104}
J_{3421}&:= -2\int  \partial_1\widetilde{u_2}\partial_1\partial_2\widetilde{\theta}\partial_{22}\overline{\theta} dx\nonumber\\&=2\int (\partial_2\partial_1\widetilde{u_2}\partial_1\partial_2\widetilde{\theta}+\partial_1\widetilde{u_2}\partial_2\partial_1\partial_2\widetilde{\theta})\partial_{2}\overline{\theta}\, dx\nonumber\\&=2\int_{\mathbb{R}} \partial_{2}\overline{\theta}\Big(\int_{\mathbb{T}}(\partial_2\partial_1\widetilde{u_2}\partial_1\partial_2\widetilde{\theta}+\partial_1\widetilde{u_2}\partial_2\partial_1\partial_2\widetilde{\theta}) dx_1\Big)dx_2\nonumber\\&\le c\int_{\mathbb{R}} |\partial_{2}\overline{\theta}|\Big(\|\partial_2\partial_1\widetilde{u_2} \|_{L^2_{x_1}}\|\partial_1\partial_2\widetilde{\theta}\|_{L^2_{x_1}}+\|\partial_1\widetilde{u_2} \|_{L^2_{x_1}}\|\partial_2\partial_1\partial_2\widetilde{\theta}\|_{L^2_{x_1}}\Big)dx_2\nonumber\\&\le c \|\partial_{2}\overline{\theta}\|_{L^{\infty}_{x_2}}\Big(\|\partial_2\partial_1\widetilde{u_2} \|_{L^2_{x_2}L^2_{x_1}}\|\partial_1\partial_2\widetilde{\theta}\|_{L^2_{x_2}L^2_{x_1}}+\|\partial_1\widetilde{u_2} \|_{L^2_{x_2}L^2_{x_1}}\|\partial_2\partial_1\partial_2\widetilde{\theta}\|_{L^2_{x_2}L^2_{x_1}}\Big)\nonumber\\& \le c\|\partial_{2}\overline{\theta}\|_{H^1}\Big(\|\partial_2\partial_1\widetilde{u_2} \|_{L^2}\|\partial_1\partial_2\widetilde{\theta}\|_{L^2}+\|\partial_1\widetilde{u_2} \|_{L^2}\|\partial_2\partial_1\partial_2\widetilde{\theta}\|_{L^2}\Big)\nonumber\\&\le c\|\theta\|_{H^2}\Big(\|\partial_{2} u\|_{H^2}^2+\|\partial_1\theta \|_{H^2}^2+\|\partial_{1} u_2\|_{L^2}^2\Big).
\end{align}
For $J_{3422}$, we apply Lemma \ref{l3} then Young's inequality,
\begin{align}\label{12usa104}
J_{3422}&:= -2\int  \partial_1\widetilde{u_2}\partial_1\partial_2\widetilde{\theta}\partial_{22}\widetilde{\theta} dx\nonumber\\&\le c\|\partial_{22}\widetilde{\theta}\|_{L^2}^{\frac12}\|\partial_1\partial_{22}\widetilde{\theta}\|_{L^2}^{\frac12}\|\partial_{12}\widetilde{\theta}\|_{L^2}^{\frac12}\|\partial_2\partial_{12}\widetilde{\theta}\|_{L^2}^{\frac12}\|\partial_1\widetilde{u_2}\|_{L^2}\nonumber\\&\le c\|\theta\|_{H^2}^{\frac12}\|u\|_{H^2}^{\frac12}\|\partial_1\theta\|_{H^2}^{\frac32}\|\partial_1u_2\|_{L^2}^{\frac12}\nonumber\\&\le c\|\theta\|_{H^2}^{\frac12}\|u\|_{H^2}^{\frac12}\Big(\|\partial_1\theta\|_{H^2}^{2}+\|\partial_1u_2\|_{L^2}^{2}\Big).
\end{align}
In view of (\ref{11usa104}), (\ref{1usa104}) and (\ref{12usa104}) we have
\begin{align}
J_{342}\le c\|(u,\theta)\|_{H^2}\Big(\|\partial_{2} u\|_{H^2}^2+\|\partial_1\theta \|_{H^2}^2+\|\partial_{1} u_2\|_{L^2}^2\Big).\label{J342*}
\end{align}
Writing $J_{343}$
more explicitly and using $\partial_1\theta=\partial_1\widetilde{\theta}$, we have
\begin{align}
J_{343}&:=-2\int\partial_2 u_2\partial_2\partial_2\theta\Delta\theta dx\nonumber\\&=-2\int\partial_2 u_2\partial_{22}\theta\partial_{11}\widetilde{\theta} dx-2\int\partial_2 u_2\partial_{22}\theta\partial_{22}\theta dx\nonumber\\&:=J_{3431}+J_{3432}.\label{bb}
\end{align}
From Lemma \ref{l3}, $J_{3431}$ can be bounded as,
\begin{align}
J_{3431}&:=-2\int\partial_2 u_2\partial_{22}\theta\partial_{11}\widetilde{\theta} dx\nonumber\\&\le c\|\partial_{11}\widetilde{\theta}\|_{L^2}^{\frac12}\|\partial_1\partial_{11}\widetilde{\theta}\|_{L^2}^{\frac12}\|\partial_2u_2\|_{L^2}^{\frac12}\|\partial_2\partial_2u_2\|_{L^2}^{\frac12}\|\partial_{22}\theta\|_{L^2}\nonumber\\&\le c\|\theta\|_{H^2}\|\partial_1\theta\|_{H^2}\|\partial_2u\|_{H^2}\nonumber\\&\le c\|\theta\|_{H^2}\Big(\|\partial_1\theta\|_{H^2}^2+\|\partial_2u\|_{H^2}^2\Big).\label{bb1}
\end{align}
The estimate of $J_{3432}$ is slightly more delicate. Due to the decomposition $\theta=\widetilde{\theta}+\overline{\theta}$, we write $J_{3432}$ as,
\begin{align}
J_{3432}&:=-2\int\partial_2 u_2\partial_{22}\theta\partial_{22}\theta dx\nonumber\\&=-2\int\partial_2 u_2\partial_{22}\overline{\theta}\partial_{22}\overline{\theta} dx-4\int\partial_2u_2\partial_{22}\overline{\theta}\partial_{22}\widetilde{\theta} dx-2\int\partial_2 u_2\partial_{22}\widetilde{\theta}\partial_{22}\widetilde{\theta}dx\nonumber\\&:=J_{34321}+J_{34322}+J_{34323}.\label{aa}
\end{align}
By $\nabla\cdot u=0$ and Lemma \ref{l1}, the first term $J_{34321}$ is clearly zero,
\begin{align}
J_{34321}=-2\int\partial_2 u_2\partial_{22}\overline{\theta}\partial_{22}\overline{\theta} dx=2\int\partial_1 \widetilde{u_1}\partial_{22}\overline{\theta}\partial_{22}\overline{\theta} dx=0.\label{aa1}
\end{align}
Applying Lemmas \ref{l3} and \ref{l4} and Young's inequality,
\begin{align}
J_{34322}&:=-4\int\partial_2u_2\partial_{22}\overline{\theta}\partial_{22}\widetilde{\theta} dx\nonumber\\&\le c \| \partial_{22}\overline{\theta}\|_{L^2}\| \partial_{22}\widetilde{\theta}\|_{L^2}^{\frac12}\|\partial_1\partial_{22}\widetilde{\theta}\|_{L^2}^{\frac12}\| \partial_{2}u_2\|_{L^2}^{\frac12}\|\partial_2\partial_{2}u_2\|_{L^2}^{\frac12}\nonumber\\&\le c \| \partial_{22}\overline{\theta}\|_{L^2}\|\partial_1\partial_{22}\widetilde{\theta}\|_{L^2}^{\frac12}\|\partial_1\partial_{22}\widetilde{\theta}\|_{L^2}^{\frac12}\| \partial_{2}u_2\|_{L^2}^{\frac12}\|\partial_2\partial_{2}u_2\|_{L^2}^{\frac12}\nonumber\\&\le c\|\theta\|_{H^2}\|\partial_1\theta\|_{H^2}\|\partial_2u\|_{H^2}\nonumber\\&\le c\|\theta\|_{H^2}\Big(\|\partial_1\theta\|_{H^2}^2+\|\partial_2u\|_{H^2}^2\Big),\label{aa2}
\end{align} 
\begin{align}
J_{34323}&:=-4\int\partial_2u_2\partial_{22}\widetilde{\theta}\partial_{22}\widetilde{\theta} dx\nonumber\\&\le c \| \partial_{22}\overline{\theta}\|_{L^2}\| \partial_{22}\widetilde{\theta}\|_{L^2}^{\frac12}\|\partial_1\partial_{22}\widetilde{\theta}\|_{L^2}^{\frac12}\| \partial_{2}u_2\|_{L^2}^{\frac12}\|\partial_2\partial_{2}u_2\|_{L^2}^{\frac12}\nonumber\\&\le c \| \partial_{22}\overline{\theta}\|_{L^2}\|\partial_1\partial_{22}\widetilde{\theta}\|_{L^2}^{\frac12}\|\partial_1\partial_{22}\widetilde{\theta}\|_{L^2}^{\frac12}\| \partial_{2}u_2\|_{L^2}^{\frac12}\|\partial_2\partial_{2}u_2\|_{L^2}^{\frac12}\nonumber\\&\le c\|\theta\|_{H^2}\|\partial_1\theta\|_{H^2}\|\partial_2u\|_{H^2}\nonumber\\&\le c\|\theta\|_{H^2}\Big(\|\partial_1\theta\|_{H^2}^2+\|\partial_2u\|_{H^2}^2\Big).\label{aa3}
\end{align} 
The bounds for $J_{3432}$ in (\ref{aa1}), (\ref{aa2}) and (\ref{aa3}) lead to,
\begin{align}
J_{3432}\le c\|\theta\|_{H^2}\Big(\|\partial_1\theta\|_{H^2}^2+\|\partial_2u\|_{H^2}^2\Big).\label{bb2}
\end{align}
Combining (\ref{bb1}) and (\ref{bb2}) and inserting them in (\ref{bb}) we obtain
\begin{align}
J_{343}\le c\|\theta\|_{H^2}\Big(\|\partial_1\theta\|_{H^2}^2+\|\partial_2u\|_{H^2}^2\Big).\label{242}
\end{align}
Inserting (\ref{241}), (\ref{J342*}) and (\ref{242}) in (\ref{24}) we get
\begin{align}
J_{34}\le c\|(u,\theta)\|_{H^2}\Big(\|\partial_1\theta\|_{H^2}^2+\|\partial_2u\|_{H^2}^2+\|\partial_1u_2\|_{L^2}^2\Big).\label{d34}
\end{align}
Thus, by (\ref{d31}), (\ref{d32}), (\ref{d33}), (\ref{d34}), and (\ref{dd3}),
\begin{align}
J_3\le c\|(u,\theta)\|_{H^2}\Big(\|\partial_2u\|_{H^2}^2&+\|\partial_1u_2\|_{L^2}^2+\|\partial_1\theta\|_{H^2}^2\Big).\label{kk}
\end{align}
As outlined in the introduction, we need the help of an extra regularization term to bound $J_2$, namely,
\beq\label{qqq}
\int_0^t \|\p_1 u_2\|_{L^2}^2\,d\tau.
\eeq
In order to make efficient use of the anisotropic dissipation, we express $J_2$ as follows
\begin{align}
J_2 =&-\int \partial_1 u_1\,  (\partial_1\omega)^2 \, dx - \int \partial_1 u_2\,  \partial_1 \omega \, \partial_2 \omega\,  dx \nonumber\\
& - \int \partial_2 u_1\, \partial_1 \omega \, \partial_2 \omega dx-\int \partial_2 u_2\,  (\partial_2\omega)^2 dx\nonumber\\
=&\int \partial_2 u_2\,  (\partial_1\omega)^2 dx - \int \partial_1 u_2\,  \partial_1 \omega\,  \partial_2 \omega \, dx \nonumber\\
& - \int \partial_2 u_1\,  \partial_1 \omega\,  \partial_2 \omega \, dx - \int \partial_2 u_2\,  (\partial_2\omega)^2\,  dx\nonumber\\
:=& J_{21} + J_{22} + J_{23} + J_{24}. \label{J2}
\end{align}
The terms $J_{21}$ through $J_{24}$ can be bounded  in the following manner. Due to $\nabla\cdot u=0$, integration by parts and Lemmas \ref{l1} and \ref{l3},
\begin{align}\label{ohno1}
J_{21}&:=-\int \partial_1u_1(\partial_1 \omega)^2\,dx\nonumber\\&=\int \partial_2u_2(\partial_1 \omega)^2\,dx\nonumber\\&=-2\int \widetilde{u_2}\partial_1\widetilde{\omega}\partial_2\partial_1\widetilde{\omega}\,dx\nonumber\\&\le c\|\widetilde{u_2}\|_{L^2}^{\frac12}\|\partial_1\widetilde{u_2}\|_{L^2}^{\frac12}\|\partial_2\partial_1\widetilde{\omega}\|_{L^2}^{\frac12}\|\partial_2\partial_1\widetilde{\omega}\|_{L^2}^{\frac12}\|\partial_2\partial_1\widetilde{\omega}\|_{L^2}\nonumber\\&\le c\|u\|_{H^2}\|\partial_1u_2\|_{L^2}^{\frac12}\|\partial_2u\|_{H^2}^{\frac32}\nonumber\\&\le c\|u\|_{H^2}\Big(\|\partial_1u_2\|_{L^2}^{2}+\|\partial_2u\|_{H^2}^{2}\Big).
\end{align}
According to Lemmas \ref{l1} and \ref{l3},
\begin{align}\label{ohno2}
J_{22}&:=-\int \partial_1u_2\partial_1 \omega\partial_2\omega\,dx\nonumber\\&=-\int \partial_1\widetilde{u_2}\partial_1\widetilde{\omega}\partial_2\omega\,dx\nonumber\\&\le c\|\partial_1\widetilde{u_2}\|_{L^2}^{\frac12}\|\partial_1\partial_1\widetilde{u_2}\|_{L^2}^{\frac12}\|\partial_1\widetilde{\omega}\|_{L^2}^{\frac12}\|\partial_2\partial_1\widetilde{\omega}\|_{L^2}^{\frac12}\|\partial_2\omega\|_{L^2}\nonumber\\&\le c\|u\|_{H^2}\|\partial_1u_2\|_{L^2}^{\frac12}\|\partial_2u\|_{H^2}^{\frac32}\nonumber\\&\le c\|u\|_{H^2}\Big(\|\partial_1u_2\|_{L^2}^{2}+\|\partial_2u\|_{H^2}^{2}\Big).
\end{align}
To bound $J_{23}$, we first use the orthogonal decomposition of $u_1$ and $\omega$ and Lemma \ref{l1}, to write $J_{23}$ as
\begin{align}
J_{23}&:=-\int \partial_2u_1\partial_1\omega\partial_2\omega\,dx=-\int \partial_2u_1\partial_1\widetilde{\omega}\partial_2\omega\,dx\nonumber\\&=-\int \partial_2\overline{u_1}\partial_1\widetilde{\omega}\partial_2\overline{\omega}\,dx-\int \partial_2\overline{u_1}\partial_1\widetilde{\omega}\partial_2\widetilde{\omega}\,dx-\int \partial_2\widetilde{u_1}\partial_1\widetilde{\omega}\partial_2 \omega\,dx\nonumber\\&=J_{231}+J_{232}+J_{233}.\label{13-}
\end{align}
According to Lemma \ref{l1}, the first term $J_{231}$ is clearly zero,
\begin{align}
J_{231}&:=-\int \partial_2\overline{u_1}\partial_1\widetilde{\omega}\partial_2\overline{\omega}\,dx=-\int_{\mathbb{R}} \partial_2\overline{u_1}\partial_2\overline{\omega}\int_{\mathbb{T}}\partial_1\widetilde{\omega}\,dx_1dx_2=0.\label{131}
\end{align}
The terms $J_{232}$ and $J_{233}$ can be bounded directly. By Lemma \ref{l3}, 
\begin{align}
J_{232}&:=-\int \partial_2\overline{u_1}\partial_1\widetilde{\omega}\partial_2\widetilde{\omega}\,dx\nonumber\\&\le c\|\partial_2\widetilde{\omega}\|_{L^2}^{\frac12}\|\partial_1\partial_2\widetilde{\omega}\|_{L^2}^{\frac12}\|\partial_1\widetilde{\omega}\|_{L^2}^{\frac12}\|\partial_2\partial_1\widetilde{\omega}\|_{L^2}^{\frac12}\|\partial_2\overline{u_1}\|_{L^2}\nonumber\\&\le c\|u\|_{H^2}\|\partial_2u\|_{H^2}^2,\label{132}
\end{align}
\begin{align}
J_{233}&:=-\int \partial_2\widetilde{u_1}\partial_1\widetilde{\omega}\partial_2\omega\,dx\nonumber\\&\le c\|\partial_2\widetilde{u_1}\|_{L^2}^{\frac12}\|\partial_1\partial_2\widetilde{u_1}\|_{L^2}^{\frac12}\|\partial_1\widetilde{\omega}\|_{L^2}^{\frac12}\|\partial_2\partial_1\widetilde{\omega}\|_{L^2}^{\frac12}\|\partial_2\omega\|_{L^2}\nonumber\\&\le  c\|u\|_{H^2}\|\partial_2u\|_{H^2}^2.\label{133}
\end{align}
Inserting these upper bounds in (\ref{13-})   yields
\begin{align}\label{ohno3}
J_{23}\le c\|u\|_{H^2}\|\partial_2u\|_{H^2}^2.
\end{align}
To deal with $J_{24}$ we use the divergence-free condition of $u$, Lemma \ref{l1}, and the inequality (\ref{ohno}) in Lemma \ref{l4}
\begin{align}\label{ohno4}
J_{24}&:=-\int \partial_2u_2(\partial_2 \omega)^2\,dx\nonumber\\&=-\int \partial_1\widetilde{u_1}(\partial_2\overline{\omega}+\partial_2\widetilde{\omega})^2\,dx\nonumber\\&=-2\int \partial_1\widetilde{u_1}\partial_2\overline{\omega}\partial_2\widetilde{\omega}\,dx-2\int \partial_1\widetilde{u_1}(\partial_2\widetilde{\omega})^2\,dx\nonumber\\&\le c\Big(\|\partial_2\overline{\omega}\|_{L^2}+\|\partial_2\widetilde{\omega}\|_{L^2}\Big)\|\partial_1\widetilde{u_1}\|_{L^2}^{\frac12}\|\partial_2\partial_1\widetilde{u_1}\|_{L^2}^{\frac12}\|\partial_1\partial_2\widetilde{\omega}\|_{L^2}\nonumber\\&\le c\|u\|_{H^2}\|\partial_2u\|_{H^2}^2.
\end{align}
Collecting the bounds for $J_{21}$ through $J_{24}$ obtained in (\ref{ohno1}), (\ref{ohno2}), (\ref{ohno3}) and (\ref{ohno4}), we obtain
\begin{align}
J_2\le c\|u\|_{H^2}\|\partial_2u\|_{H^2}^2\,.\label{j2b}
\end{align}
Inserting $J_1=0$, (\ref{kk}) and (\ref{j2b}) in (\ref{h2energy}), yields
\ben
&& \frac{d}{dt}(\|\Delta u\|_{L^2}^{2}+\|\Delta \theta\|_{L^2}^{2}) + 2\nu \|\partial_2 \Delta u \|_{L^2}^{2}+ 2\eta  \|\partial_1 \Delta \theta \|_{L^2}^{2}\notag \\
&& \leq c\, \|(u,\theta)\|_{H^2}  \Big( \|\partial_2 u \|_{H^2}^2+  \| \partial_1\theta\|_{H^2}^2+\|\partial_1u_2\|_{L^2}^2\Big). \label{honesum}
\een
Integrating (\ref{honesum}) over $[0, t]$, we get
\ben
&& \|\Delta u(t)\|_{L^2}^{2}+\|\Delta \theta(t)\|_{L^2}^{2} +2 \nu  \int_{0}^{t}\|\partial_2  \Delta u \|_{L^2}^{2}d\tau + 2 \eta  \int_{0}^{t}\|\Delta \partial_1  \theta \|_{L^2}^{2}d\tau
\notag \\
&& \le \|\Delta u_0\|_{L^2}^{2}+\|\Delta \theta_0\|_{L^2}^{2} + c\, \int_0^t   \|(u,\theta)\|_{H^2}  \Big( \|\partial_2 u \|_{H^2}^2+  \| \partial_1\theta\|_{H^2}^2+\|\partial_1u_2\|_{L^2}^2\Big) \notag \\
&& \leq E(0)  + \, c\,  E(t)^{\frac{3}{2}}. \label{H2final}
\een
The subsequent step is to control the last piece in $E(t)$ defined by (\ref{ee}), namely
\begin{align}
\int_0^t \|g_0\p_1 u_2\|_{L^2}^2\,d\tau.\label{ff}
\end{align}
Our strategy is to make use of the special structure of the equation for $\theta$ in (\ref{oussama}) and replace $g_0\partial_1u_2$ in (\ref{ff}) via the equation of $\theta$,
\begin{equation}\label{p1u}
g_0\partial_1 u_2= -\partial_t \partial_1 \theta - \partial_1 (u \cdot \nabla \theta) + \eta \partial_{111} \theta. 
\end{equation}
Multiplying (\ref{p1u}) by $g_0\partial_1 u_2$ and then integrating over $\Omega$, we obtain
\begin{align}
\|g_0\partial_1 u_2\|_{L^2}^{2}&= -g_0\int \partial_t \partial_1 \theta\,  \partial_1 u_2\;dx - g_0\int \partial_1 u_2\, \partial_1 (u \cdot \nabla \theta)\;dx + g_0\eta \int \partial_1 u_2\,\partial_{111} \theta\;dx \nonumber\\
&:= K_1 +K_2 + K_3.\label{K**}
\end{align}
We bound $K_3$ as follows, 
\beq\label{k1up}
|K_3|\le \,\eta \|g_0\p_1 u_2\|_{L^2}\, \|\p_{111} \theta\|_{L^2} \le \frac12 \|g_0\p_1 u_2\|_{L^2}^2 + c\, \|\p_1\theta\|_{H^2}^2, 
\eeq
the term with unfavorable derivative $\p_1 u_2$ will be then absorbed by the left-hand side of (\ref{k1up}). 
\\
For $K_1$, we first shift the time derivative  
\beq\label{k3up}
K_1 = - g_0\frac{d}{dt} \int \partial_1 \theta\, \partial_1 u_2\, dx+g_0\int  \partial_1 \theta\, \partial_1 \partial_t u_2\, dx := K_{11} + K_{12}.
\eeq
Using the equation for the second component of the velocity, we write 
\be
K_{12} &=& -g_0\int \p_1\p_1\theta \, \p_t u_2\,dx \\
&=& -g_0\int \partial_{11}  \theta  (-(u\cdot \nabla u_2)- \partial_2 p+\nu \partial_{22}u_2+ g_0\theta)\;dx\nonumber\\
&= & g_0\int \partial_{11}  \theta\, (u\cdot \nabla u_2)\;dx\;+g_0\int \partial_{11}  \theta \;\partial_2 p\;dx\nonumber\\
&& - g_0\nu \int \partial_{11}  \theta\; \partial_{22}u_2\;dx\; - g_0^2\int \partial_{11}  \theta\; \theta\;dx.\label{eqp1}
\ee
Then, we apply the divergence operator to the velocity equation to express the pressure term as
\begin{equation}
p= -\Delta^{-1}\nabla \cdot (u\cdot \nabla u)+ g_0\Delta^{-1} \partial_2 \theta.\label{eqp}
\end{equation}
Inserting (\ref{eqp}) in (\ref{eqp1}), we obtain \begin{align}
K_{12} &= g_0\int \partial_{11}  \theta\, (u\cdot \nabla u_2)\;dx\;+g_0\int \partial_{11}  \theta\, (-\partial_2 \Delta^{-1}\nabla \cdot (u\cdot \nabla u))\;dx\nonumber\\ 
&-g_0\nu \int \partial_{11}  \theta\; \partial_{22}u_2\;dx\; -g_0^2\int \partial_{11}  \theta\; \partial_{11} \Delta^{-1}\theta\;dx\nonumber\\
&:= K_{121}+ K_{122}+ K_{123}+K_{124}.\label{HA1}
\end{align}
Due to Hölder's inequality and the fact that the double Riesz transform $\p_{11} \Delta^{-1}$ is bounded on $L^q$ for any $1<q<\infty$ (see, e.g., \cite{Stein}), we have 
\begin{align}
K_{124} := -g_0^2\int \p_1\theta\, \partial_{11} \Delta^{-1} \p_1 \theta\;dx \le c\, \|\p_1\theta\|_{L^2}\|\partial_{11} \Delta^{-1}\p_1\theta\|_{L^2}\le c\, \|\p_1\theta\|_{L^2}^2. \label{HA2}
\end{align}
Thanks to Hölder's inequality, 
\begin{align}
K_{123}:=-g_0\nu \int \partial_{11}  \theta\; \partial_{22}u_2\;dx \le c\, \|\p_{11}\theta\|_{L^2} \, \|\p_{22} u_2\|_{L^2}.\label{HA3}
\end{align}
By integration by parts, Hölder's inequality and the boundedness of the double Riesz transform, 
\begin{align}
K_{122} &:=g_0\int \partial_{11}  \theta\, (-\partial_2 \Delta^{-1}\nabla \cdot (u\cdot \nabla u))\;dx\nonumber\\&= g_0\int \partial_{1} \theta\; \partial_{12} \Delta^{-1}\nabla \cdot (u\cdot \nabla u)\;dx\nonumber\\
&\leq c \|\partial_{1} \theta\|_{L^2}\; \| \Delta^{-1}\partial_{12} \nabla \cdot (u\cdot \nabla u)\|_{L^2}\nonumber\\
&\leq \,c\, \|\partial_{1} \theta\|_{L^2}\,\|\partial_{2} (u\cdot \nabla u)\|_{L^2} \nonumber\\  
&\leq  \,c\, \|\partial_{1} \theta\|_{L^2}\,\|\partial_{2} u\cdot \nabla u+ u\cdot \nabla \partial_2 u\|_{L^2}  \nonumber\\
&\leq  \,c\, \|\partial_{1} \theta\|_{L^2} \left(\;\|\partial_{2}u\|_{L^4} \;\| \nabla u\|_{L^4}+ \|u\|_{\infty}\|\nabla \partial_2 u\|_{L^2}\right) \nonumber\\             
&\leq  \,c\,\|\partial_{1} \theta\|_{L^2} \,\|\partial_{2}u\|_{H^1}\;  \| \nabla u\|_{H^1}+
\,c\,\|\partial_{1} \theta\|_{L^2} \, \|u\|_{H^2}\|\nabla \partial_2 u\|_{L^2}.\label{HA44}
\end{align}
To deal with $K_{121}$, we rewrite it as
\begin{align}
K_{121} &= g_0\int \partial_{11}  \theta (u_1 \partial_1 u_2+ u_2 \partial_2 u_2) dx\nonumber\\
&=g_0\int \partial_{11}  \widetilde{\theta}\; u_1\; \partial_1 u_2 \;dx\;+ g_0\int \partial_{11}\widetilde{\theta}\; u_2\; \partial_2 u_2\; dx\nonumber\\
&=g_0\int \partial_{11}  \widetilde{\theta}\; \widetilde{u_1}\; \partial_1 \widetilde{u_2} \;dx\;+g_0\int \partial_{11}  \widetilde{\theta}\; \overline{u_1}\; \partial_1 \widetilde{u_2} \;dx+ g_0\int \partial_{11}\widetilde{\theta}\; u_2\; \partial_2 u_2\; dx\nonumber\\&:=K_{1211}+K_{1212}+K_{1213}.\label{HA111}
\end{align}
By Lemma \ref{l3}, the divergence-free condition of $u$ and Lemma \ref{l4}, 
\begin{align}
K_{1211} &:=g_0\int \partial_{11}  \widetilde{\theta}\; \widetilde{u_1}\; \partial_1 \widetilde{u_2} \;dx\nonumber\\&\le \, c\, \|\partial_{11} \theta\|_{L^2}^{\frac12}\|\partial_2\partial_{11} \theta\|_{L^2}\underbrace{\|\widetilde{u_1}\|_{L^2}^{\frac12}}_{\le \|\partial_1\widetilde{u_1}\|_{L^2}^{\frac12}}\underbrace{\|\partial_1\widetilde{u_1}\|_{L^2}^{\frac12}}_{=\|\partial_2\widetilde{u_2}\|_{L^2}^{\frac12}}\|\partial_1 \widetilde{u_2}\|_{L^2}\nonumber\\&\le \,c\, \|u\|_{H^2}\, \|\p_2u\|_{H^2}\, \|\partial_{1} \theta\|_{H^2}\nonumber\\&\le \,c\, \|u\|_{H^2}\Big(\|\p_2u\|_{H^2}^2+ \|\partial_{1} \theta\|_{H^2}^2\Big).\label{HA11}
\end{align}
Due to Lemma \ref{l1}, Hölder's inequality, Lemma \ref{delta216} and then Young's inequality,
\begin{align}
K_{1212} &:=g_0\int \partial_{11}  \widetilde{\theta}\; \overline{u_1}\; \partial_1 \widetilde{u_2} \;dx\nonumber\\&=g_0\int_{\mathbb{R}} \overline{u_1}\Big(\int_{\mathbb{T}}\partial_{11}  \widetilde{\theta} \partial_1 \widetilde{u_2}dx_1\Big)dx_2\nonumber\\&\le c\int_{\mathbb{R}} | \overline{u_1}|\|\partial_{11}  \widetilde{\theta}\|_{L^2_{x_1}}\| \partial_1 \widetilde{u_2} \|_{L^2_{x_1}}dx_2\nonumber\\&\le  c\|\overline{u_1}\|_{L^{\infty}_{x_2}}\|\partial_{11}\widetilde{\theta} \|_{L^2_{x_2}L^2_{x_1}}\|\partial_{1} \widetilde{u_2}\|_{L^2_{x_2}L^2_{x_1}}\nonumber\\& \le c\| \overline{u_1}\|_{H^1}\|\partial_{11}\widetilde{\theta} \|_{L^2}\|\partial_{1} \widetilde{u_2}\|_{L^2}\nonumber\\&\le c\|u\|_{H^2}\Big(\|\partial_{1} u_2\|_{L^2}^2+\|\partial_1\theta \|_{H^2}^2\Big).\label{HA12}
\end{align}
According to Lemma \ref{l3},
\begin{align}
K_{1213}&:=g_0\int \partial_{11}\widetilde{\theta}\; u_2\; \partial_2 u_2\; dx\nonumber\\&\le c\|\partial_{11}\widetilde{\theta}\|_{L^2}^{\frac12}\|\partial_1\partial_{11}\widetilde{\theta}\|_{L^2}^{\frac12}\| \partial_2 u_2\|_{L^2}^{\frac12}\|\partial_2\; \partial_2 u_2\|_{L^2}^{\frac12}\|u_2\|_{L^2}\nonumber\\&\le c\|u\|_{H^2}\|\partial_{2} u\|_{H^2}\|\partial_1\theta \|_{H^2}\nonumber\\&\le c\|u\|_{H^2}\Big(\|\partial_{2} u\|_{H^2}^2+\|\partial_1\theta \|_{H^2}^2\Big).\label{HA13}
\end{align}
Inserting (\ref{HA11}), (\ref{HA12}) and (\ref{HA13}) in (\ref{HA111}) we get
\begin{align}
K_{121}\le c\|u\|_{H^2}\Big(\|\partial_{2} u\|_{H^2}^2+\|\partial_{1} u_2\|_{L^2}^2+\|\partial_1\theta \|_{H^2}^2\Big).\label{HA4}
\end{align}
It then follows from (\ref{HA1}), (\ref{HA2}), (\ref{HA3}), (\ref{HA44}) and (\ref{HA4}) that 
\beq \label{k12up}
|K_{12}| \le \,c\|u\|_{H^2}\Big(\|\partial_{2} u\|_{H^2}^2+\|\partial_{1} u_2\|_{L^2}^2+\|\partial_1\theta \|_{H^2}^2\Big).
\eeq
We now need to bound $K_2$. We first split it into four terms,
\begin{align}
K_2 :&= - g_0\int \partial_1 u_2\, \partial_1 (u \cdot \nabla \theta)\;dx\nonumber\\&=  - g_0\int \partial_1 u_2\, \partial_1 u_1\, \partial_1  \theta\,  dx - g_0\int \partial_1 u_2  u_1 \partial_1 \partial_1  \theta\, dx\nonumber\\
&\;\;\;\;- g_0\int \partial_1 u_2 \partial_1 u_2 \partial_2  \theta\,  dx - g_0\int \partial_1 u_2  u_2 \partial_1 \partial_2  \theta\, dx\nonumber\\:&=K_{21}+K_{22}+K_{23}+K_{24}. \label{K2}
\end{align}
Due to $\partial_1\theta=\partial_1\widetilde{\theta}$, Lemma \ref{l3} and Young's inequality
\begin{align}\label{K21}
K_{21}&:= - g_0\int \partial_1 u_2\, \partial_1 u_1\, \partial_1  \theta\,  dx\nonumber\\&= - g_0\int \partial_1 u_2\, \partial_1 u_1\, \partial_1  \widetilde{\theta}\,  dx\nonumber\\&\le c  \|\partial_1 \widetilde{\theta}\|_{L^2}^{\frac{1}{2}} \|\partial_1 \partial_1 \widetilde{\theta}\|_{L^2}^{\frac{1}{2}} \|\partial_2 u_2\|_{L^2}^{\frac{1}{2}}\|\partial_2 \partial_2 u_2\|_{L^2}^{\frac{1}{2}} \|\partial_1u_2\|_{L^2} \nonumber\\&\le c\|u\|_{H^2}\|\partial_{2} u\|_{H^2}\|\partial_1\theta \|_{H^2}\nonumber\\&\le c\|u\|_{H^2}\Big(\|\partial_{2} u\|_{H^2}^2+\|\partial_1\theta \|_{H^2}^2\Big).
\end{align}
Using Lemma \ref{l1} and invoking the decompositions $u=\overline{u}+\widetilde{u}$ we write $K_{22}$ as
\begin{align}
K_{22}&:=-g_0\int \partial_1 u_2  u_1 \partial_1 \partial_1  \theta\, dx\nonumber\\&=g_0\int \partial_{11}  \widetilde{\theta}\; \widetilde{u_1}\; \partial_1 \widetilde{u_2} \;dx\;+g_0\int \partial_{11}  \widetilde{\theta}\; \overline{u_1}\; \partial_1 \widetilde{u_2} \;dx\nonumber\\&:=K_{221}+K_{222}.\label{ci}
\end{align}
By Lemmas \ref{l3}, \ref{l4} and the divergence-free condition of $u$,
\begin{align}\label{ci1}
K_{221} &:=g_0\int \partial_{11}  \widetilde{\theta}\; \widetilde{u_1}\; \partial_1 \widetilde{u_2} \;dx\nonumber\\&\le \, c\, \|\partial_{11} \theta\|_{L^2}^{\frac12}\|\partial_2\partial_{11} \theta\|_{L^2}\underbrace{\|\widetilde{u_1}\|_{L^2}^{\frac12}}_{\le \|\partial_1\widetilde{u_1}\|_{L^2}^{\frac12}}\underbrace{\|\partial_1\widetilde{u_1}\|_{L^2}^{\frac12}}_{=\|\partial_2\widetilde{u_2}\|_{L^2}^{\frac12}}\|\partial_1 \widetilde{u_2}\|_{L^2}\nonumber\\&\le \,c\, \|u\|_{H^2}\, \|\p_2u\|_{H^2}\, \|\partial_{1} \theta\|_{H^2}\nonumber\\&\le \,c\, \|u\|_{H^2}\Big(\|\p_2u\|_{H^2}^2+ \|\partial_{1} \theta\|_{H^2}\Big).
\end{align}
To bound $K_{222}$, we first use Lemma \ref{l1}, Hölder's inequality and then Lemma \ref{delta216} to obtain
\begin{align}\label{ci2}
K_{222} &:=g_0\int \partial_{11}  \widetilde{\theta}\; \overline{u_1}\; \partial_1 \widetilde{u_2} \;dx\nonumber\\&=g_0\int_{\mathbb{R}} \overline{u_1}\Big(\int_{\mathbb{T}}\partial_{11}  \widetilde{\theta} \partial_1 \widetilde{u_2}dx_1\Big)dx_2\nonumber\\&\le c\int_{\mathbb{R}} | \overline{u_1}|\|\partial_{11}  \widetilde{\theta}\|_{L^2_{x_1}}\| \partial_1 \widetilde{u_2} \|_{L^2_{x_1}}dx_2\nonumber\\&\le c \|\overline{u_1}\|_{L^{\infty}_{x_2}}\|\partial_{11}\widetilde{\theta} \|_{L^2_{x_2}L^2_{x_1}}\|\partial_{1} \widetilde{u_2}\|_{L^2_{x_2}L^2_{x_1}}\nonumber\\& \le c\| \overline{u_1}\|_{H^1}\|\partial_{11}\widetilde{\theta} \|_{L^2}\|\partial_{1} \widetilde{u_2}\|_{L^2}\nonumber\\&\le c\|u\|_{H^2}\Big(\|\partial_{1} u_2\|_{L^2}^2+\|\partial_1\theta \|_{H^2}^2\Big).
\end{align}
Then (\ref{ci1}), (\ref{ci2}) and (\ref{ci}) together leads to
\begin{align}
K_{22}\le c\|u\|_{H^2}\Big(\|\partial_{2} u\|_{H^2}^2+\|\partial_{1} u_2\|_{L^2}^2+\|\partial_1\theta \|_{H^2}^2\Big).\label{K22}
\end{align}
By $\partial_1u_2=\partial_1\widetilde{u_2}$ and $\theta=\widetilde{\theta}+\overline{\theta}$, we rewrite $K_{23}$ as
\begin{align}
K_{23}&:=- g_0\int \partial_1 u_2 \partial_1 u_2 \partial_2  \theta\,  dx \nonumber\\&=- g_0\int \partial_1 \widetilde{u_2} \partial_1 \widetilde{u_2} \partial_2  \overline{\theta}\,  dx - g_0\int \partial_1 \widetilde{u_2} \partial_1 \widetilde{u_2} \partial_2  \widetilde{\theta}\,  dx \nonumber\\&:=K_{231}+K_{232}.\label{di}
\end{align}
To estimate $K_{231}$, we make use of Lemma \ref{l1}, Hölder's inequality and then Lemma \ref{delta216} to get
\begin{align}\label{di1}
K_{231} &:=g_0\int \partial_{1}  \widetilde{u_2}\partial_{1}  \widetilde{u_2} \partial_2\overline{\theta} \;dx\nonumber\\&=g_0\int_{\mathbb{R}} \partial_2\overline{\theta}\Big(\int_{\mathbb{T}}\partial_{1}  \widetilde{u_2}\partial_{1}  \widetilde{u_2}dx_1\Big)dx_2\nonumber\\&\le c\int_{\mathbb{R}} | \partial_2\overline{\theta}|\|\partial_{1}  \widetilde{u_2}\|_{L^2_{x_1}}\|\partial_{1}  \widetilde{u_2}\|_{L^2_{x_1}}dx_2\nonumber\\&\le  c\|\partial_2\overline{\theta}\|_{L^{\infty}_{x_2}}\|\partial_{1}  \widetilde{u_2}\|_{L^2_{x_2}L^2_{x_1}}\|\partial_{1}  \widetilde{u_2}\|_{L^2_{x_2}L^2_{x_1}}\nonumber\\& \le c\| \partial_2\overline{\theta}\|_{H^1}\|\partial_{1}  \widetilde{u_2}\|_{L^2}^2\nonumber\\&\le c\|\theta\|_{H^2}\|\partial_{1} u_2\|_{L^2}^2.
\end{align}
Via Lemma \ref{l3},
\begin{align}\label{di2}
K_{232} &:=g_0\int \partial_{1}  \widetilde{u_2}\partial_{1}  \widetilde{u_2} \partial_2\widetilde{\theta} \;dx\nonumber\\&\le c  \|\partial_2 \widetilde{\theta}\|_{L^2}^{\frac{1}{2}} \|\partial_1 \partial_2 \widetilde{\theta}\|_{L^2}^{\frac{1}{2}} \| \partial_1 u_2\|_{L^2}^{\frac{1}{2}}\|\partial_2 \partial_1 u_2\|_{L^2}^{\frac{1}{2}}\|\partial_1 u_2\|_{L^2}\nonumber\\&\le c\|\theta\|_{H^2}^{\frac{1}{2}}\|u\|_{H^2}^{\frac{1}{2}} \|\partial_1\theta\|_{H^2}^{\frac{1}{2}}\|\partial_1 u_2\|_{L^2}^{\frac{3}{2}}\nonumber\\&\le c\|(u,\theta)\|_{H^2}\Big(\|\partial_1\theta\|_{H^2}^2+\|\partial_1 u_2\|_{L^2}^2\Big).
\end{align}
Inserting the bounds for $K_{231}$ and $K_{232}$ in (\ref{di}), we find
\begin{align}
K_{23}\le c\|(u,\theta)\|_{H^2}\Big(\|\partial_1\theta\|_{H^2}^2+\|\partial_1 u_2\|_{L^2}^2\Big).\label{K23}
\end{align}
The last term $K_{24}$ can also be bounded due to the fact that $\overline{u_2}=0$, Lemmas \ref{l3} and \ref{l4}
\begin{align}
K_{24}&:= -g_0 \int \partial_1 u_2  u_2 \partial_1 \partial_2  \theta\, dx\nonumber\\&= -g_0 \int \partial_1 u_2  \widetilde{u_2} \partial_1 \partial_2  \theta\, dx\nonumber\\&\le c \underbrace{\|\widetilde{u_2}\|_{L^2}^{\frac{1}{2}}}_{\le \|\partial_1\widetilde{u_2}\|_{L^2}^{\frac{1}{2}}} \|\partial_1 \widetilde{u_2}\|_{L^2}^{\frac{1}{2}} \|\partial_1 u_2\|_{L^2}^{\frac{1}{2}} \|\partial_2 \partial_1 u_2\|_{L^2}^{\frac{1}{2}} \|\partial_1 \partial_2 \theta\|_{L^2}\nonumber\\&\le c\|\theta\|_{H^2}^{\frac{1}{2}}\|u\|_{H^2}^{\frac{1}{2}} \|\partial_1\theta\|_{H^2}^{\frac{1}{2}}\|\partial_1 u_2\|_{L^2}^{\frac{3}{2}}\nonumber\\&\le c\|(u,\theta)\|_{H^2}\Big(\|\partial_1\theta\|_{H^2}^2+\|\partial_1 u_2\|_{L^2}^2\Big).\label{K24}
\end{align}
Inserting (\ref{K21}), (\ref{K22}), (\ref{K23}), (\ref{K24}), in (\ref{K2}) we obtain 
\ben
K_2 \le c\|(u,\theta)\|_{H^2}\Big(\|\partial_1\theta\|_{H^2}^2+\|\partial_1 u_2\|_{L^2}^2+\|\partial_2 u\|_{H^2}^2\Big). \label{k2up}
\een 
Collecting the bounds obtained above for $K_1$ through $K_3$ in (\ref{k1up}), (\ref{k3up}), (\ref{k12up}) and (\ref{k2up}) and inserting them in (\ref{K**}), we get 
\begin{align}
\frac12 \, \|g_0\p_1 u_2\|_{L^2}^2 &\le \,c\, \|\p_1\theta\|_{H^2}^2 -g_0\frac{d}{dt}\int \p_1\theta\, \p_1 u_2\,dx \nonumber\\
&+c\|(u,\theta)\|_{H^2}\Big(\|\partial_1\theta\|_{H^2}^2+\|\partial_2 u\|_{H^2}^2+\|\partial_1 u_2\|_{L^2}^2\Big).\label{mo}
\end{align}
Integrating (\ref{mo}) over the time interval $[0, t]$, we find
\ben
\int_0^t \|g_0\p_1 u_2\|_{L^2}^2 \,d\tau &\le& \,c\, \int_0^t \|\p_1\theta\|_{H^2}^2 \,d\tau
-2g_0\int \p_1\theta\, \p_1 u_2\,dx + 2g_0\int \p_1\theta_0\, \p_1 u_{02}\,dx \notag \\
&& + \,c\,\int_0^t \|(u,\theta)\|_{H^2}\Big(\|\partial_1\theta\|_{H^2}^2+\|\partial_2 u\|_{H^2}^2+\|\partial_1 u_2\|_{L^2}^2\Big) \,d\tau \notag\\
&\le& \,c\, \int_0^t \|\p_1\theta\|_{H^2}^2 \,d\tau + \,c\, \int_0^t \|\p_2 u\|_{H^2}^2\,d\tau
 + \,c\, (\|u\|_{H^1}^2 + \|\theta\|_{H^1}^2)\notag \\
&& + \,c\, (\|u_0\|_{H^1}^2 + \|\theta_0\|_{H^1}^2) + c\, E(t)^{\frac32}. \label{defb}
\een

To conclude, we combine the $H^1$-bound in (\ref{h1final}), the homogeneous $H^2$-bound in (\ref{H2final})
and the bound for the extra regularization term in (\ref{defb}). When doing so, we need eliminate 
the quadratic terms on the right-hand side of (\ref{defb}) by the corresponding terms 
on the left-hand side, then it suffices to multiply both sides of (\ref{defb}) by a suitable small coefficient
$\delta>0$. Taking $(\ref{h1final}) + (\ref{H2final}) + \delta\,(\ref{defb})$, leads to
\ben
&& \|u(t)\|_{H^2}^{2} +\|\theta(t)\|_{H^2}^{2} + 2 \nu \int_{0}^{t} \|\partial_2 u\|_{H^2}^{2}d\tau  
+ 2\eta \int_{0}^t \|\partial_1 \theta\|_{H^2}^{2} d\tau+ \delta \int_{0}^{t}\|g_0\partial_1 u_2\|_{L^2}^{2} \notag \\
&&\leq \,  E(0) + \,c\,E(t)^{\frac{3}{2}}\;+\, c\,\delta\, (\|u(t)\|_{H^{2}}^{2} +\|\theta(t)\|_{H^{2}}^{2}) + \,c\,\delta\, (\|u_0\|_{H^{2}}^{2} + \|\theta_0\|_{H^{2}}^{2}) \notag\\
&&\quad  + \,c\,\delta \int_{0}^{t}\|\partial_2 u \|_{H^2}^{2} d\tau 
+ \,c\,\delta\int_{0}^{t}\|\partial_1 \theta  \|_{H^2}^{2} d\tau 
+ \,c\,\delta\, E(t)^{\frac{3}{2}}. \label{gir}
\een
If  $\delta>0$ is chosen to be sufficiently small, say 
$$
c\, \delta \le \frac12, \quad c\, \delta \le \nu, \quad c\, \delta \le \eta,
$$
then (\ref{gir}) gives 
\begin{equation}\label{ggd}
E(t)\leq \,C_1\,  E(0)+ \, C_2\, E(t)^{\frac{3}{2}},
\end{equation}
where $C_1$ and $C_2$ are positive constants. The proof of the desired stability result, is then completed by applying the bootstrapping argument on (\ref{ggd}). Indeed, if the initial data $(u_0, \theta_0)$, is sufficiently small, say,
\begin{align}
E(0)=\|(u_0, \theta_0)\|_{H^2}^2 \le \varepsilon^2 := \frac1{16 C_1C_2^2},\label{E0}
\end{align}
then (\ref{ggd}) implies 
$$
\|(u(t), \theta(t))\|_{H^2}^2 \le 2 C_1\, \varepsilon^2.
$$
To initiate the bootstrapping argument, we make the ansatz that, for $t\le T$
\beq\label{lov1}
E(t) \le \frac1{4 C_2^2},
\eeq
and we then show that $E(t)$ actually admits an even smaller bound by taking the initial $H^2$-norm $E(0)$ sufficiently small. In fact,  Inserting (\ref{lov1}) in (\ref{ggd}) yields 
\beno
E(t) &\leq& \,C_1\,  E(0)+ \, C_2\, E(t)^{\frac{3}{2}} 
\\
&\le& C_1 \, \varepsilon^2 + \, C_2\, \frac1{2C_2}\, E(t).
\eeno
That is, 
$$
\frac12 E(t) \le C_1 \, \varepsilon^2\quad\mbox{or}\quad E(t) \le 2\, C_1 \, \frac{1}{16 C_1 C_2^2} = \frac{1}{8 C_2^2} = 2C_1\, \epsilon^2, \quad\mbox{for all $t\le T$}.
$$
The bootstrapping argument then assesses that (\ref{lov1}) holds for all time when $E(0)$ satisfies (\ref{E0}). This establishes the global stability.
\vskip .1in 
Finally, we establish the uniqueness of $H^2$-solutions to (\ref{oussama}). Assume 
that $(u^{(1)}, p^{(1)}, \theta^{(1)})$ and $(u^{(2)}, p^{(2)}, \theta^{(2)})$ are two solutions of (\ref{oussama}) with one of them in the $H^2$-regularity class say $(u^{(1)}, \theta^{(1)})\in L^{\infty}(0, T; H^2) $.  The difference between the two solutions $({u}^*, {p}^*, {\theta}^*)$ with 
$$
{u}^*=u^{(2)}- u^{(1)},\quad {p}^*=p^{(2)}- p^{(1)} \quad \text{and} \quad {\theta}^*=\theta^{(2)}- \theta^{(1)} 
$$
verifies 
\begin{equation}\label{uniqueness}\begin{cases}
\begin{aligned}
&\partial_t {u}^*+u^{(2)}\cdot \nabla {u}^*+ {u}^* \cdot \nabla u^{(1)}+\nabla {p}^*= \nu \partial_{22}{u}^*+ g_0{\theta}^*{\mathbf e}_2, \\
&\partial_t {\theta}^* + u^{(2)} \cdot \nabla {\theta}^* + {u}^*\cdot \nabla \theta^{(1)} + g_0{u_2}^*\;= \eta \partial_{11} {\theta}^*, \\
&\nabla \cdot {u}^*=0,\\
& u^*(x,0) =0, \quad \theta^*(x,0) =0.
\end{aligned}
\end{cases}
\end{equation}
We estimate the difference $({u}^*, {p}^*, {\theta}^*)$ in $L^{2}(\Omega)$. Taking the $L^2$-inner product of (\ref{uniqueness}) with $({u}^*, {\theta}^*)$ and applying the divergence-free condition, we get
\begin{align}
\frac{1}{2}\frac{d}{dt}\|({u}^*, {\theta}^*)\|_{L^2}^{2}+ \nu \|\partial_{2}{u}^*\|_{L^2}^{2}+ \eta \|\partial_{1}{\theta}^*\|_{L^2}^{2}&= -\int {u}^* \cdot \nabla u^{(1)}\cdot {u}^*\; dx - \int {u}^* \cdot \nabla \theta^{(1)}\cdot {\theta}^*\; dx\nonumber\\&=I_1+I_2.\label{=1}
\end{align}
Due to Lemma \ref{l2} and the uniformly global bound for $\|{u}^{(1)}\|_{H^2}$,
\begin{align}\label{unique11}
I_1&:=-\int{u}^{*}\cdot\nabla u^{(1)}\cdot{u}^{*}\,dx\nonumber\\&\le c\underbrace{\|\nabla{u}^{(1)}\|_{L^2}^{\frac12}\Big(\|\nabla{u}^{(1)}\|_{L^2}+\|\partial_1\nabla{u}^{(1)}\|_{L^2}\Big)^{\frac12}}_{\le c}\|{u}^{*}\|_{L^2}^{\frac12}\|\partial_2{u}^{*}\|_{L^2}^{\frac12}\|{u}^{*}\|_{L^2}\nonumber\\&\le c\|{u}^{*}\|_{L^2}^{\frac32}\|\partial_2{u}^{*}\|_{L^2}^{\frac12}\nonumber\\&\le c\|{u}^{*}\|_{L^2}^2+\frac{\nu}{4}\|\partial_2{u}^{*}\|_{L^2}^2.
\end{align}
Similarly, by Lemma \ref{l2} and the uniformly global bound for $\|{\theta}^{(1)}\|_{H^2}$,
\begin{align}\label{unique21}
I_2&:=-\int{u}^{*}\cdot\nabla \theta^{(1)}\cdot{\theta}^{*}\,dx\nonumber\\&\le c\underbrace{\|\nabla{\theta}^{(1)}\|_{L^2}^{\frac12}\Big(\|\nabla{\theta}^{(1)}\|_{L^2}+\|\partial_1\nabla{\theta}^{(1)}\|_{L^2}\Big)^{\frac12}}_{\le c}\|u^{*}\|_{L^2}^{\frac12}\|\partial_2u^{*}\|_{L^2}^{\frac12}\|\theta^{*}\|_{L^2}\nonumber\\&\le c\|u^{*}\|_{L^2}^{\frac12}\|\partial_2u^{*}\|_{L^2}^{\frac12}\|\theta^{*}\|_{L^2}\nonumber\\&\le c\|\theta^{*}\|_{L^2}\Big(\|u^{*}\|_{L^2}+\|\partial_2u^{*}\|_{L^2}\Big)\nonumber\\&\le c\|\theta^{*}\|_{L^2}^2+c\|u^{*}\|_{L^2}^2+\frac{\nu}{4}\|\partial_2u^{*}\|_{L^2}^2.
\end{align}
Putting the estimates (\ref{unique11}) and (\ref{unique21}) in (\ref{=1}) leads to
\begin{align*}
\frac12&\frac{d}{dt}\|({u}^{*},{\theta}^{*})\|^2_{L^2}+\nu\|\partial_2{u}^{*}\|_{L^2}^2+\eta\|\partial_1{\theta}^{*}\|_{L^2}^2\\&\le c\Big(\|{u}^{*}\|_{L^2}^2+\|{\theta}^{*}\|_{L^2}^2\Big)+\frac{\nu}{2}\|\partial_2{u}^{*}\|_{L^2}^2
\end{align*}
or
\begin{align}
\frac{d}{dt}\|({u}^{*},{\theta}^{*})\|^2_{L^2}+\nu\|\partial_2{u}^{*}\|_{L^2}^2+\eta\|\partial_1{\theta}^{*}\|_{L^2}^2\le c\|({u}^{*},{\theta}^{*})\|^2_{L^2}.\label{i1}
\end{align}
Grönwall’s inequality then implies,
$$
\|{u}^*(t)\|_{L^2}= \|{\theta}^*(t)\|_{L^{2}} =0.
$$
In other words, these two solutions coincide.
This finishes the proof of Theorem \ref{TH}.\end{proof}
\section{Decay Rates Result} 
\label{decay}
This section is devoted to the proof the decay rates presented in Theorem \ref{TH1}.
\begin{proof}[Proof of Theorem \ref{TH1}]
Taking the
average of the system (\ref{oussama}) and using the fact that $\overline{u\cdot\nabla\overline{u}}=0$, we write the equations of $(\overline{u},\overline{\theta})$, 
\begin{align}\label{horizontalaverage}
\begin{cases}
	\partial_t\overline{u}+\overline{u\cdot\nabla\widetilde{u}}+\begin{pmatrix}
	0\\\partial_2\overline{p}
	\end{pmatrix}=g_0\begin{pmatrix}
	0\\\overline{\theta}
	\end{pmatrix}+\nu\partial_2^2\overline{u}\,,\\\partial_t\overline{\theta}+\overline{u\cdot\nabla\widetilde{\theta}}=0,
\end{cases}
\end{align}
where $g_0$ is a negative constant.
By subtracting (\ref{horizontalaverage}) from (\ref{oussama}), we get
\begin{align}\label{*}
\begin{cases}
	\partial_t\widetilde{u}+\widetilde{u\cdot\nabla\widetilde{u}}+\widetilde{u_2}\partial_2\overline{u}-\nu\partial_2^2\widetilde{u}+\nabla \widetilde{p}=g_0\widetilde{\theta}e_2\,,\\\partial_t\widetilde{\theta}+\widetilde{u\cdot\nabla\widetilde{\theta}}+\widetilde{u_2}\partial_2\overline{\theta}-\eta\partial_1^2\widetilde{\theta}+ g_0\widetilde{u_2}=0.
\end{cases}
\end{align}
Taking the $L^2$-inner product of $(\widetilde{u},\widetilde{\theta})$ with (\ref{*}) yields, 
\begin{align}
&\frac12\frac{d}{dt}\Big(\|\widetilde{u}\|_{L^2}^2+\|\widetilde{\theta}\|_{L^2}^2\Big)+\nu\|\partial_2\widetilde{u}\|_{L^2}^2+\eta\|\partial_1\widetilde{\theta}\|_{L^2}^2\nonumber\\&=-\int\widetilde{u\cdot\nabla\widetilde{u}}\cdot\widetilde{u}dx-\int\widetilde{ u_2}\partial_2\overline{u}\cdot\widetilde{u}dx-\int\widetilde{u\cdot\nabla\widetilde{\theta}}\cdot\widetilde{\theta}dx-\int \widetilde{u_2}\partial_2\overline{\theta}\cdot\widetilde{\theta}dx\nonumber\\&:=A_1+A_2+A_3+A_4.\label{AA}
\end{align}
Now, we estimate $A_1$ through $A_4$. The first term $A_1$ is clearly zero due to $\nabla\cdot u=0$ and Lemma \ref{l1},
\begin{align}
A_1:=-\int\widetilde{u\cdot\nabla\widetilde{u}}\cdot\widetilde{u}dx=\underbrace{-\int u\cdot\nabla\widetilde{u}\cdot\widetilde{u}dx}_{=0}+\underbrace{\int\overline{u\cdot\nabla\widetilde{u}}\cdot\widetilde{u}dx}_{=0}=0.\label{A_11}
\end{align}
Likewise, \begin{align}A_3:=\int\widetilde{u\cdot\nabla\widetilde{\theta}}\cdot\widetilde{\theta}dx =0.\label{A_13}\end{align}
To bound $A_2$ we first write it as,
\begin{align}
A_2&:=-\int \widetilde{u_2}\partial_2\overline{u}\cdot\widetilde{u}dx\nonumber\\&:=-\int \widetilde{u_2}\partial_2\overline{u_1}\widetilde{u_1}dx-\int \widetilde{u_2}\partial_2\overline{u_2}\widetilde{u_2}dx\nonumber\\&:=A_{21}+A_{22}.\label{decA2}
\end{align}
Due to the fact that $\overline{u_2}=0$ we have,
\begin{align}
A_{22}=-\int \widetilde{u_2}\partial_2\overline{u_2}\widetilde{u_2}dx=0.\label{A_122}
\end{align}
Applying Lemmas \ref{l3} and \ref{l4}, the divergence-free condition of $u$ and then Young's inequality leads to
\begin{align}
A_{21}&:=-\int \widetilde{u_2}\partial_2\overline{u_1}\widetilde{u_1}dx\nonumber\\& \le c\|\partial_2\overline{u}\|_{L^2}\|\widetilde{u_2}\|_{L^2}^{\frac12}\|\partial_2\widetilde{u_2}\|_{L^2}^{\frac12}\underbrace{\|\widetilde{u_1}\|_{L^2}^{\frac12}}_{\le \|\partial_1\widetilde{u_1}\|_{L^2}^{\frac12}= \|\partial_2\widetilde{u_2}\|_{L^2}^{\frac12}}\underbrace{\|\partial_1\widetilde{u_1}\|_{L^2}^{\frac12}}_{\;\;=\|\partial_2\widetilde{u_2}\|_{L^2}^{\frac12}}\nonumber\\& \le c\|u\|_{H^2}\|\widetilde{u_2}\|_{L^2}^{\frac12}\|\partial_2\widetilde{u}\|_{L^2}^{\frac32}\nonumber\\&\le c\|u\|_{H^2}\Big(\|\widetilde{u_2}\|_{L^2}^{2}+\|\partial_2\widetilde{u}\|_{L^2}^{2}\Big).\label{A_121}
\end{align}
Inserting (\ref{A_122}) and (\ref{A_121}) in (\ref{decA2}) we get
\begin{align}
A_2\le c\|u\|_{H^2}\Big(\|\widetilde{u_2}\|_{L^2}^{2}+\|\partial_2\widetilde{u}\|_{L^2}^{2}\Big). \label{A_12}
\end{align}
The last term $A_4$ can be bounded via Lemma \ref{l1}, Hölder's inequality, and Lemmas \ref{delta216} and  \ref{l4},
\begin{align}
A_4&:=-\int \widetilde{u_2}\partial_2\overline{\theta}\cdot\widetilde{\theta}dx\nonumber\\&=-\int_{\mathbb{R}} \partial_2\overline{\theta}\Big(\int_{\mathbb{T}}  \widetilde{\theta} \widetilde{u_2}dx_1\Big)dx_2\nonumber\\&\le c\|\partial_2\overline{\theta}\|_{L^{\infty}_{x_2}}\|\widetilde{u_2}\|_{L^2}\|\widetilde{\theta}\|_{L^2}\nonumber\\&\le c \| \partial_2\overline{\theta}\|_{H^{1}}\|\widetilde{u_2}\|_{L^2}\|\partial_1\widetilde{\theta}\|_{L^2}\nonumber \\&\le c \| \theta\|_{H^{2}}\|\widetilde{u_2}\|_{L^2}\|\partial_1\widetilde{\theta}\|_{L^2}\nonumber\\&\le c\|\theta\|_{H^2}\Big(\|\widetilde{u_2}\|_{L^2}^2+\|\partial_1\widetilde{\theta}\|_{L^2}^{2}\Big).\label{A_14}
\end{align}
Combining the estimates of $A_1$ through $A_4$, we get
\begin{align}
\frac12\frac{d}{dt}\Big(\|\widetilde{u}\|_{L^2}^2&+\|\widetilde{\theta}\|_{L^2}^2\Big)+\nu\|\partial_2\widetilde{u}\|_{L^2}^2+\eta\|\partial_1\widetilde{\theta}\|_{L^2}^2\nonumber\\&\le c\|(u,\theta)\|_{H^2}\Big(\|\widetilde{u_2}\|_{L^2}^2+\|\partial_2\widetilde{u}\|_{L^2}^2+\|\partial_1\widetilde{\theta}\|_{L^2}^2\Big).\label{M2}
\end{align}
Applying $\nabla$ to (\ref{*}), we write
\begin{align}\label{15}
\begin{cases}
	\partial_t\nabla\widetilde{u}+\nabla(\widetilde{u\cdot\nabla\widetilde{u}})+\nabla(\widetilde{u_2}\partial_2\overline{u})-\nu\partial_2^2\nabla\widetilde{u}+\nabla\nabla \widetilde{p}=g_0\nabla(\widetilde{\theta}e_2)\,,\\\partial_t\nabla\widetilde{\theta}+\nabla(\widetilde{u\cdot\nabla\widetilde{\theta}})+\nabla(\widetilde{u_2}\partial_2\overline{\theta})-\eta\partial_1^2\nabla\widetilde{\theta}+ g_0\nabla\widetilde{u_2}=0.
\end{cases}
\end{align}
Dotting (\ref{15}) by $(\nabla\widetilde{u},\nabla\widetilde{\theta})$, we get
\begin{align}
&\frac12\frac{d}{dt}\Big(\|\nabla\widetilde{u}(t)\|_{L^2}^2+\|\nabla\widetilde{\theta}(t)\|_{L^2}^2\Big)+\nu\|\partial_2\nabla\widetilde{u}\|_{L^2}^2+\eta\|\partial_1\nabla\widetilde{\theta}\|_{L^2}^2\nonumber\\&=-\int\nabla(\widetilde{u\cdot\nabla\widetilde{u}})\cdot\nabla\widetilde{u}dx-\int \nabla(\widetilde{u_2}\partial_2\overline{u})\cdot\nabla\widetilde{u}dx\nonumber\\&\quad-\int\nabla(\widetilde{u\cdot\nabla\widetilde{\theta}})\cdot\nabla\widetilde{\theta}dx-\int \nabla(\widetilde{u_2}\partial_2\overline{\theta})\cdot\nabla\widetilde{\theta}dx\nonumber\\&:=B_1+B_2+B_3+B_4.\label{B}
\end{align}
The terms $ B_1$ through $B_4$ can be bounded as follows. We start with $B_1$. Using Lemma \ref{l1}, we write $B_1$ as,
\begin{align}
B_1&:=-\int\nabla(\widetilde{u\cdot\nabla\widetilde{u}})\cdot\nabla\widetilde{u}dx\nonumber\\&=-\int\nabla(u\cdot\nabla\widetilde{u})\cdot\nabla\widetilde{u}\;dx+\underbrace{\int\nabla(\overline{u\cdot\nabla\widetilde{u}})\cdot\nabla\widetilde{u}dx}_{=0}\nonumber\\&=-\int \partial_1u_1\partial_1\widetilde{u} \cdot\partial_1\widetilde{u}dx-\int \partial_1u_2\partial_2\widetilde{u}\cdot\partial_1\widetilde{u}dx\notag\\
&\quad -\int \partial_2u_1\partial_1\widetilde{u}\cdot\partial_2\widetilde{u}dx-\int \partial_2u_2\partial_2\widetilde{u}\cdot\partial_2\widetilde{u}dx\nonumber\\&:=B_{11}+B_{12}+B_{13}+B_{14}.\label{B_1}
\end{align}
Further, we divide the first term $B_{11}$ into the following two integrals,
\begin{align}
B_{11}&:=-\int \partial_1u_1\partial_1\widetilde{u}\cdot\partial_1\widetilde{u}dx\nonumber\\&=-\int \partial_1\widetilde{u_1}\partial_1\widetilde{u_1}\partial_1\widetilde{u_1}dx-\int \partial_1\widetilde{u_1}\partial_1\widetilde{u_2}\partial_1\widetilde{u_2}dx\nonumber\\&:=B_{111}+B_{112}.\label{B_{11=}}
\end{align}
By the divergence-free condition of $u$ and Lemma \ref{l3}
\begin{align}
B_{111}&:=-\int \partial_1\widetilde{u_1}\partial_1\widetilde{u_1}\partial_1\widetilde{u_1}dx\nonumber\\&=\int \partial_2\widetilde{u_2}\partial_2\widetilde{u_2}\partial_2\widetilde{u_2}dx\nonumber\\&\le c\|\partial_2\widetilde{u_2}\|_{L^2}\|\partial_2\widetilde{u_2}\|_{L^2}^{\frac12}\|\partial_1\partial_2\widetilde{u_2}\|_{L^2}^{\frac12}\|\partial_2\widetilde{u_2}\|_{L^2}^{\frac12}\|\partial_2\partial_2\widetilde{u_2}\|_{L^2}^{\frac12}
\nonumber\\&\le c\|u\|_{H^2}\|\partial_2\widetilde{u}\|_{H^1}^2.\label{B_{111}}
\end{align}
Due to $\nabla\cdot u=0$, integration by parts and Lemma, \ref{l3}\begin{align}
B_{112}&:=-\int \partial_1\widetilde{u_1}\partial_1\widetilde{u_2}\partial_1\widetilde{u_2}dx\nonumber\\&=\int \partial_2\widetilde{u_2}\partial_1\widetilde{u_2}\partial_1\widetilde{u_2}dx\nonumber\\&=2\int \widetilde{u_2}\partial_2\partial_1\widetilde{u_2}\partial_1\widetilde{u_2}dx\nonumber\\&\le c\|\partial_2\partial_1\widetilde{u_2}\|_{L^2}\|\widetilde{u_2}\|_{L^2}^{\frac12}\|\partial_2\widetilde{u_2}\|_{L^2}^{\frac12}\|\partial_1\widetilde{u_2}\|_{L^2}^{\frac12}\|\partial_1\partial_1\widetilde{u_2}\|_{L^2}^{\frac12}
\nonumber\\&\le c\|u\|_{H^2}\|\widetilde{u_2}\|_{L^2}^{\frac12}\|\partial_2\widetilde{u}\|_{H^1}^{\frac32}\nonumber\\&\le c\|u\|_{H^2}\Big(\|\widetilde{u_2}\|_{L^2}^{2}+\|\partial_2\widetilde{u}\|_{H^1}^{2}\Big).\label{B_{112}}
\end{align}
Inserting the upper bound for $B_{111}$ and $B_{112}$ in (\ref{B_{11=}}) we get
\begin{align}
B_{11}\le c\|u\|_{H^2}\Big(\|\widetilde{u_2}\|_{L^2}^{2}+\|\partial_2\widetilde{u}\|_{H^1}^{2}\Big).\label{B_{11}}
\end{align}
To deal with $B_{12},$ we write it first as
\begin{align}
B_{12}&:=-\int \partial_1u_2\partial_2\widetilde{u}\cdot\partial_1\widetilde{u}dx\nonumber\\&=-\int \partial_1\widetilde{u_2}\partial_2\widetilde{u_1}\partial_1\widetilde{u_1}dx-\int \partial_1\widetilde{u_2}\partial_2\widetilde{u_2}\partial_1\widetilde{u_2}dx\nonumber\\&:=B_{121}+B_{122}.\label{B_{12=}} 
\end{align}
For $B_{121},$ we use the divergence-free condition of $u$ and Lemma \ref{l3}
\begin{align}
B_{121}&:=-\int \partial_1\widetilde{u_2}\partial_2\widetilde{u_1}\partial_1\widetilde{u_1}dx=\int \partial_1\widetilde{u_2}\partial_2\widetilde{u_1}\partial_2\widetilde{u_2}dx\nonumber\\&\le c\|\partial_1\widetilde{u_2}\|_{L^2}\|\partial_2\widetilde{u_1}\|_{L^2}^{\frac12}\|\partial_1\partial_2\widetilde{u_1}\|_{L^2}^{\frac12}\|\partial_2\widetilde{u_2}\|_{L^2}^{\frac12}\|\partial_2\partial_2\widetilde{u_2}\|_{L^2}^{\frac12}\nonumber\\&\le c\|u\|_{H^2}\|\partial_2\widetilde{u}\|_{H^1}^{2}.\label{B_{121}}
\end{align}
The second piece $B_{122}$ can be bounded using integrating by parts, Lemma \ref{l3} and then Young's inequality
\begin{align}
B_{122}&:=-\int \partial_1\widetilde{u_2}\partial_2\widetilde{u_2}\partial_1\widetilde{u_2}dx\nonumber\\&\nonumber=2\int \partial_2\partial_1\widetilde{u_2}\widetilde{u_2}\partial_1\widetilde{u_2}dx\nonumber\\&\le  c\|\partial_2\partial_1\widetilde{u_2}\|_{L^2}\|\widetilde{u_2}\|_{L^2}^{\frac12}\|\partial_2\widetilde{u_2}\|_{L^2}^{\frac12}\|\partial_1\widetilde{u_2}\|_{L^2}^{\frac12}\|\partial_1\partial_1\widetilde{u_2}\|_{L^2}^{\frac12}
\nonumber\\&\le c\|u\|_{H^2}\|\widetilde{u_2}\|_{L^2}^{\frac12}\|\partial_2\widetilde{u}\|_{H^1}^{\frac32}\nonumber\\&\le c\|u\|_{H^2}\Big(\|\widetilde{u_2}\|_{L^2}^{2}+\|\partial_2\widetilde{u}\|_{H^1}^{2}\Big)\label{B_{122}}
\end{align}
Combining (\ref{B_{121}}) and (\ref{B_{122}}) and inserting them in (\ref{B_{12=}}) we obtain
\begin{align}
B_{12}\le c\|u\|_{H^2}\Big(\|\widetilde{u_2}\|_{L^2}^{2}+\|\partial_2\widetilde{u}\|_{H^1}^{2}\Big).\label{B_{12}}
\end{align}
The term $B_{13}$ is naturally divided into two integrals,
\begin{align}
B_{13}&:=-\int \partial_2u_1\partial_1\widetilde{u}\cdot\partial_2\widetilde{u}dx\nonumber\\&=-\int \partial_2u_1\partial_1\widetilde{u_1}\partial_2\widetilde{u_1}dx-\int \partial_2u_1\partial_1\widetilde{u_2}\partial_2\widetilde{u_2}dx\nonumber\\&:=B_{131}+B_{132}.\label{B_{13=}} 
\end{align}
Due to $\nabla\cdot u=0$ and Lemma \ref{l3},
\begin{align}
B_{131}&:=-\int \partial_2u_1\partial_1\widetilde{u_1}\partial_2\widetilde{u_1}dx\nonumber\\&=\int \partial_2u_1\partial_2\widetilde{u_2}\partial_2\widetilde{u_1}dx\nonumber\\&\le c\|\partial_2u_1\|_{L^2}\|\partial_2\widetilde{u_2}\|_{L^2}^{\frac12}\|\partial_1\partial_2\widetilde{u_2}\|_{L^2}^{\frac12}\|\partial_2\widetilde{u_1}\|_{L^2}^{\frac12}\|\partial_2\partial_2\widetilde{u_1}\|_{L^2}^{\frac12}\nonumber\\&\le c\|u\|_{H^2}\|\partial_2\widetilde{u}\|_{H^1}^2.\label{B_{131}}
\end{align}
Integrating by parts, making use of Lemma \ref{l3} and then Young's inequality
\begin{align}
B_{132}&:=-\int \partial_2u_1\partial_1\widetilde{u_2}\partial_2\widetilde{u_2}dx\nonumber\\&=\int \partial_1\partial_2\widetilde{u_1}\widetilde{u_2}\partial_2\widetilde{u_2}dx+\int \partial_2u_1\widetilde{u_2}\partial_1\partial_2\widetilde{u_2}dx\nonumber\\&\le c\|\partial_1\partial_2\widetilde{u_1}\|_{L^2}\|\widetilde{u_2}\|_{L^2}^{\frac12}\|\partial_2\widetilde{u_2}\|_{L^2}^{\frac12}\|\partial_2\widetilde{u_2}\|_{L^2}^{\frac12}\|\partial_1\partial_2\widetilde{u_2}\|_{L^2}^{\frac12}\nonumber\\&\quad+c\|\partial_2u_1\|_{L^2}\|\widetilde{u_2}\|_{L^2}^{\frac12}\|\partial_2\widetilde{u_2}\|_{L^2}^{\frac12}\|\partial_1\partial_2\widetilde{u_2}\|_{L^2}^{\frac12}\|\partial_1\partial_1\partial_2\widetilde{u_2}\|_{L^2}^{\frac12}\nonumber\\&\le c\|u\|_{H^2}\|\widetilde{u_2}\|_{L^2}^{\frac12}\|\partial_2\widetilde{u}\|_{H^1}^{\frac32}\nonumber\\&\le c\|u\|_{H^2}\Big(\|\widetilde{u_2}\|_{L^2}^2+\|\partial_2\widetilde{u}\|_{H^1}^2\Big).\label{B_{132}}
\end{align}
Inserting the estimates (\ref{B_{131}}) and (\ref{B_{132}}) in (\ref{B_{13=}}) we get 
\begin{align}
B_{13}\le c\|u\|_{H^2}\Big(\|\widetilde{u_2}\|_{L^2}^2+\|\partial_2\widetilde{u}\|_{H^1}^2\Big).\label{B_{13}}
\end{align}
The last term $B_{14}$ can be bounded directly via Lemma \ref{l3},
\begin{align}
B_{14}&:=-\int \partial_2u_2\partial_2\widetilde{u}\cdot\partial_2\widetilde{u}dx \nonumber\\&\le c\|\partial_2u_2\|_{L^2}\|\partial_2\widetilde{u}\|_{L^2}^{\frac12}\|\partial_1\partial_2\widetilde{u}\|_{L^2}^{\frac12}\|\partial_2\widetilde{u}\|_{L^2}^{\frac12}\|\partial_2\partial_2\widetilde{u}\|_{L^2}^{\frac12}\nonumber\\&\le c\|u\|_{H^2}\|\partial_2\widetilde{u}\|_{H^1}^2. \label{B_{14}} 
\end{align}
Collecting the upper bounds obtained in (\ref{B_{11}}), (\ref{B_{12}}), (\ref{B_{13}}) and (\ref{B_{14}}) and inserting them in (\ref{B_1}), yields
\begin{align}
B_1 \le c\|u\|_{H^2}\Big(\|\widetilde{u_2}\|_{L^2}^2+\|\partial_2\widetilde{u}\|_{H^1}^2\Big).\label{B_11}
\end{align}
The next term $B_2$ is naturally split into four parts, 
\begin{align}
B_2&:=-\int \nabla(\widetilde{u_2}\partial_2\overline{u})\cdot\nabla\widetilde{u}dx\nonumber\\&=-\int\partial_1\widetilde{u_2}\partial_2\overline{u}\cdot\partial_1\widetilde{u}dx -\int\partial_2\widetilde{u_2}\partial_2\overline{u}\cdot\partial_2\widetilde{u}dx\nonumber\\&\quad -\int \widetilde{u_2}\partial_1\partial_2\overline{u}\cdot\partial_1\widetilde{u}dx-\int \widetilde{u_2}\partial_2\partial_2\overline{u}\cdot\partial_2\widetilde{u}dx\nonumber\\&:=B_{21}+B_{22}+B_{23}+B_{24}.\label{B_2}
\end{align}
We rewrite $B_{21}$ as,
\begin{align}
B_{21}&:=-\int\partial_1\widetilde{u_2}\partial_2\overline{u}\cdot \partial_1\widetilde{u}dx\nonumber\\&=-\int\partial_1\widetilde{u_2}\partial_2\overline{u_1} \partial_1\widetilde{u_1}dx-\int\partial_1\widetilde{u_2}\partial_2\overline{u_2} \partial_1\widetilde{u_2}dx\nonumber\\&:=B_{211}+B_{212}.\label{B_{21=}}
\end{align}
Clearly, due to $\overline{u_2}=0$, 
\begin{align}
B_{212}:=-\int\partial_1\widetilde{u_2}\partial_2\overline{u_2} \partial_1\widetilde{u_2}dx=0.\label{B_{211}}
\end{align}
By the divergence-free condition of $u$, integration by parts, Lemma \ref{l1}, Hölder's inequality and then Lemma \ref{delta216} 
\begin{align}
B_{211}&:=-\int\partial_1\widetilde{u_2}\partial_2\overline{u_1} \partial_1\widetilde{u_1}dx\nonumber\\&=\int\partial_1\widetilde{u_2}\partial_2\overline{u_1} \partial_2\widetilde{u_2}dx\nonumber\\&=\int\widetilde{u_2}\partial_2\overline{u_1} \partial_1\partial_2\widetilde{u_2}dx\nonumber\\&=\int_{\mathbb{R}} \partial_2\overline{u_1}\Big(\int_{\mathbb{T}}  \widetilde{u_2} \partial_1\partial_2\widetilde{u_2}dx_1\Big)dx_2\nonumber\\&\le c\|\partial_2\overline{u_1}\|_{L^{\infty}_{x_2}}\|\widetilde{u_2} \|_{L^2}\|\partial_1\partial_2\widetilde{u_2}\|_{L^2}\nonumber\\&\le  c\|\partial_2\overline{u_1}\|_{H^1}\|\widetilde{u_2} \|_{L^2}\|\partial_1\partial_2\widetilde{u_2}\|_{L^2}\nonumber \\&\le c \| u\|_{H^{2}}\|\widetilde{u_2}\|_{L^2}\|\partial_2\widetilde{u}\|_{H^1}\nonumber\\&\le c\|\theta\|_{H^2}\Big(\|\widetilde{u_2}\|_{L^2}^2+\|\partial_2\widetilde{u}\|_{H^1}^{2}\Big).
\label{B_{212}}
\end{align}
It then follows from (\ref{B_{211}}), (\ref{B_{212}}) and (\ref{B_{21=}}) that
\begin{align}
B_{21}\le  c\|u\|_{H^2}\Big(\|\partial_2\widetilde{u}\|_{H^1}^{2}+\|\widetilde{u_2}\|_{L^2}^{2}\Big).\label{B_{21}}
\end{align}
According to Lemma \ref{l3},
\begin{align}
B_{22}&:=-\int\partial_2\widetilde{u_2}\partial_2\overline{u}\cdot\partial_2\widetilde{u}dx\nonumber\\&\le c\|\partial_2\overline{u}\|_{L^2}\|\partial_2\widetilde{u_2}\|_{L^2}^{\frac12}\|\partial_1\partial_2\widetilde{u_2}\|_{L^2}^{\frac12}\|\partial_2\widetilde{u}\|_{L^2}^{\frac12}\|\partial_2\partial_2\widetilde{u}\|_{L^2}^{\frac12}\nonumber\\&\le c\|u\|_{H^2}\|\partial_2\widetilde{u}\|_{H^1}^2.\label{B_{22}}
\end{align}
By definition of $\overline{u}$,
\begin{align}
B_{23}&:=-\int \widetilde{u_2}\partial_1\partial_2\overline{u}\cdot\partial_1\widetilde{u}dx = 0.\label{B_{24}}
\end{align}
Due Lemma \ref{l3} and Young's inequality, 
\begin{align}
B_{24}&:=-\int \widetilde{u_2}\partial_2\partial_2\overline{u}\cdot\partial_2\widetilde{u}dx\nonumber\\&\le c\|\partial_2 \partial_2 \overline{u}\|_{L^2}\|\widetilde{u_2}\|_{L^2}^{\frac12}\|\partial_2\widetilde{u_2}\|_{L^2}^{\frac12}\|\partial_2\widetilde{u}\|_{L^2}^{\frac12}\|\partial_1\partial_2\widetilde{u}\|_{L^2}^{\frac12}\nonumber\\&\le c\|u\|_{H^2}\|\widetilde{u_2}\|_{L^2}^{\frac12}\|\partial_2\widetilde{u}\|_{H^1}^{\frac32}\nonumber\\&\le c\|u\|_{H^2}\Big(\|\widetilde{u_2}\|_{L^2}^2+\|\partial_2\widetilde{u}\|_{H^1}^2\Big). \label{B_{23}}
\end{align}
Collecting the estimates (\ref{B_{21}}), (\ref{B_{22}}), (\ref{B_{24}}), (\ref{B_{23}}) and (\ref{B_2}),  we get
\begin{align}
B_2\le c\|u\|_{H^2}\Big(\|\widetilde{u_2}\|_{L^2}^2+\|\partial_2\widetilde{u}\|_{H^1}^2\Big).\label{B_22}
\end{align}
To bound $B_3$, we first write $u =\overline{u}+ \widetilde{u}$ and use Lemma \ref{l1}
\begin{align}
B_3&:=-\int\nabla(\widetilde{u\cdot\nabla\widetilde{\theta}})\cdot\nabla\widetilde{\theta}dx\nonumber\\&=-\int\nabla(u\cdot\nabla\widetilde{\theta})\cdot\nabla\widetilde{\theta}dx+\underbrace{\int\nabla(\overline{u\cdot\nabla\widetilde{\theta}})\cdot\nabla\widetilde{\theta}dx}_{=0}\nonumber\\&=-\int \partial_1\widetilde{\theta}\partial_1\widetilde{u_1}\partial_1\widetilde{\theta} dx-\int \partial_2\widetilde{\theta}\partial_1\widetilde{u_2}\partial_1\widetilde{\theta} dx\nonumber\\&\quad-\int \partial_1\widetilde{\theta}\partial_2u_1\partial_2\widetilde{\theta} dx-\int \partial_2\widetilde{\theta}\partial_2\widetilde{u_2}\partial_2\widetilde{\theta} dx\nonumber\\&:=B_{31}+
B_{32}+B_{33}+B_{34}.\label{B_3}\end{align}
All terms in (\ref{B_3}) can be bounded suitably. In fact, by Lemma \ref{l3},
\begin{align}
B_{31}&:=-\int \partial_1\widetilde{\theta}\partial_1\widetilde{u_1}\partial_1\widetilde{\theta} dx\nonumber\\&\le c\|\partial_1\widetilde{u_1}\|_{L^2}\|\partial_1\widetilde{\theta}\|_{L^2}^{\frac12}\|\partial_1\partial_1\widetilde{\theta}\|_{L^2}^{\frac12}\|\partial_1\widetilde{\theta}\|_{L^2}^{\frac12}\|\partial_2\partial_1\widetilde{\theta}\|_{L^2}^{\frac12}
\nonumber\\&\le c\|u\|_{H^2}\|\partial_1\widetilde{\theta}\|_{H^1}^{2}.\label{B_{31}}
\end{align}
For $B_{32}$, $B_{33}$ and $B_{34}$ we use Lemmas \ref{l3} and \ref{l4}, 
\begin{align}
B_{32}&:=-\int \partial_2\widetilde{\theta}\partial_1\widetilde{u_2}\partial_1\widetilde{\theta} dx\nonumber\\&\le c\|\partial_1\widetilde{u_2}\|_{L^2}\underbrace{\|\partial_2\widetilde{\theta}\|_{L^2}^{\frac12}}_{\le \|\partial_1\partial_2\widetilde{\theta}\|_{L^2}^{\frac12}} \|\partial_1\partial_2\widetilde{\theta}\|_{L^2}^{\frac12}\|\partial_1\widetilde{\theta}\|_{L^2}^{\frac12}\|\partial_2\partial_1\widetilde{\theta}\|_{L^2}^{\frac12}
\nonumber\\&\le c\|u\|_{H^2}\|\partial_1\widetilde{\theta}\|_{H^1}^{2},\label{B_{32}}
\end{align}
\begin{align}
B_{33}&:=-\int \partial_1\widetilde{\theta}\partial_2u_1\partial_2\widetilde{\theta} dx\nonumber\\&\le c\|\partial_2u_1\|_{L^2}\underbrace{\|\partial_2\widetilde{\theta}\|_{L^2}^{\frac12}}_{\le \|\partial_1\partial_2\widetilde{\theta}\|_{L^2}^{\frac12}} \|\partial_1\partial_2\widetilde{\theta}\|_{L^2}^{\frac12}\|\partial_1\widetilde{\theta}\|_{L^2}^{\frac12}\|\partial_2\partial_1\widetilde{\theta}\|_{L^2}^{\frac12}
\nonumber\\&\le c\|u\|_{H^2}\|\partial_1\widetilde{\theta}\|_{H^1}^{2}, \label{B33}
\end{align}
and
\begin{align}
B_{34}&:=-\int \partial_2\widetilde{\theta}\partial_2\widetilde{u_2}\partial_2\widetilde{\theta} dx\nonumber\\&\le c\|\partial_2\widetilde{\theta}\|_{L^2}\underbrace{\|\partial_2\widetilde{\theta}\|_{L^2}^{\frac12}}_{\le \|\partial_1\partial_2\widetilde{\theta}\|_{L^2}^{\frac12}}\|\partial_1\partial_2\widetilde{\theta}\|_{L^2}^{\frac12}\|\partial_2\widetilde{u_2}\|_{L^2}^{\frac{1}{2}}\|\partial_2\partial_2\widetilde{u_2}\|_{L^2}^{\frac12}
\nonumber\\&\le c\|\theta\|_{H^2}\|\partial_2\widetilde{u}\|_{H^1}\|\partial_1\widetilde{\theta}\|_{H^1}\nonumber\\&\le c\|\theta\|_{H^2}\Big(\|\partial_2\widetilde{u}\|_{H^1}^{2}+\|\partial_1\widetilde{\theta}\|_{H^1}^{2}\Big).\label{B_{34}}
\end{align}
Combining the estimates (\ref{B_{31}}), (\ref{B_{32}}), (\ref{B33}), (\ref{B_{34}}) and (\ref{B_3}), we obtain
\begin{align}
B_3 \le c\|(u,\theta)\|_{H^2}\Big(\|\partial_2\widetilde{u}\|_{H^1}^{2}+\|\partial_1\widetilde{\theta}\|_{H^1}^{2}\Big).\label{B_33}
\end{align}
To deal with $B_4$, we split it into four pieces,
\begin{align}
B_4&:=-\int \nabla(\widetilde{u_2}\partial_2\overline{\theta})\cdot\nabla\widetilde{\theta}dx\nonumber\\&=-\int \partial_1(\widetilde{u_2}\partial_2\overline{\theta})\cdot\partial_1\widetilde{\theta}dx-\int \partial_2(\widetilde{u_2}\partial_2\overline{\theta})\cdot\partial_2\widetilde{\theta}dx\nonumber\\&=-\int \partial_1\widetilde{u_2}\partial_2\overline{\theta}\partial_1\widetilde{\theta}dx-\int \widetilde{u_2}\partial_1\partial_2\overline{\theta}\partial_1\widetilde{\theta}dx\nonumber\\&\quad-\int \partial_2\widetilde{u_2}\partial_2\overline{\theta}\partial_2\widetilde{\theta}dx-\int \widetilde{u_2}\partial_2\partial_2\overline{\theta}\partial_2\widetilde{\theta}dx\nonumber\\&:=B_{41}+B_{42}+B_{43} + B_{44}.\label{B_4}
\end{align}
The terms above can be bounded as follows. Due to the definition of the horizontal average $\overline{\theta}$,
\begin{align} \label{faracha}
B_{42}:= -\int \widetilde{u_2}\partial_1\partial_2\overline{\theta}\partial_1\widetilde{\theta}dx = 0.
\end{align}
For $B_{41}$, we use integration by parts, Lemma \ref{l1}, Hölder's inequality and then Lemma \ref{delta216}  
\begin{align}
B_{41}&:=-\int \partial_1\widetilde{u_2}\partial_2\overline{\theta}\partial_1\widetilde{\theta}dx\nonumber\\&\nonumber=-\int \widetilde{u_2}\partial_2\overline{\theta}\partial_1\partial_1\widetilde{\theta}dx\nonumber\\&=\int_{\mathbb{R}} \partial_2\overline{\theta}\Big(\int_{\mathbb{T}}  \widetilde{u_2} \partial_1\partial_1\widetilde{\theta}dx_1\Big)dx_2\nonumber\\&\le c\|\partial_2\overline{\theta}\|_{L^{\infty}_{x_2}}\|\widetilde{u_2} \|_{L^2}\|\partial_1\partial_1\widetilde{\theta}\|_{L^2}\nonumber\\&\le c \|\partial_2\overline{\theta}\|_{H^1} \| \widetilde{u_2}\|_{L^2} \| \partial_1\partial_1 \widetilde{\theta}\|_{L^2}\nonumber\\&\le c \|\theta\|_{H^2} \| \widetilde{u_2}\|_{L^2} \| \partial_1 \widetilde{\theta}\|_{H^1}\nonumber \\&\le c\|\theta\|_{H^2}\Big(\|\widetilde{u_2}\|_{L^2}^2+\|\partial_1 \widetilde{\theta}\|_{H^1}^{2}\Big).\label{B_{41}}
\end{align}
The other two terms $B_{43}$, $B_{44}$ can be bounded via Lemmas \ref{l3} and \ref{l4},
\begin{align}
B_{43}&:=-\int \partial_2\widetilde{u_2}\partial_2\overline{\theta}\partial_2\widetilde{\theta}dx\nonumber\\&\le c\|\partial_2\overline{\theta}\|_{L^2}\underbrace{\|\partial_2\widetilde{\theta}\|_{L^2}^{\frac12}}_{\le \|\partial_1\partial_2\widetilde{\theta}\|_{L^2}^{\frac12}}\|\partial_1\partial_2\widetilde{\theta}\|_{L^2}^{\frac12}\|\partial_2\widetilde{u_2}\|_{L^2}\|\partial_2\partial_2\widetilde{u_2}\|_{L^2}^{\frac12}
\nonumber\\&\le c\|\theta\|_{H^2}\|\partial_1\widetilde{\theta}\|_{H^1}\|\partial_2\widetilde{u}\|_{H^1}\nonumber\\&\le c\|\theta\|_{H^2}\Big(\|\partial_1\widetilde{\theta}\|_{H^1}^{2}+\|\partial_2\widetilde{u}\|_{H^1}^{2}\Big),\label{B_{42}}
\end{align}
\begin{align}
B_{44}&:=-\int \widetilde{u_2}\partial_2\partial_2\overline{\theta}\partial_2\widetilde{\theta}dx\nonumber\\&\le c\|\partial_2\partial_2\overline{\theta}\|_{L^2}
\underbrace{\|\partial_2\widetilde{\theta}\|_{L^2}^{\frac12}}_{\le \|\partial_1\partial_2\widetilde{\theta}\|_{L^2}^{\frac12}}\|\partial_1\partial_2\widetilde{\theta}\|_{L^2}^{\frac12}\|\widetilde{u_2}\|_{L^2}^{\frac12}\|\partial_2\widetilde{u_2}\|_{L^2}^{\frac12}\nonumber\\&\le c\|\theta\|_{H^2}\|\widetilde{u_2}\|_{L^2}^{\frac12}\|\partial_2\widetilde{u}\|_{H^1}^{\frac12}\|\partial_1\widetilde{\theta}\|_{H^1}\nonumber\\&\le c\|\theta\|_{H^2}\Big(\|\partial_2\widetilde{u}\|_{H^1}^{2}+\|\widetilde{u_2}\|_{L^2}^2+\|\partial_1\widetilde{\theta}\|_{H^1}^{2}\Big).\label{B_{43}}
\end{align}
Inserting all the bounds obtained above for $B_{41}$ through  $B_{44}$ in (\ref{B_4}) leads to
\begin{align}
B_4 \le c\|(u,\theta)\|_{H^2}\Big(\|\partial_2\widetilde{u}\|_{H^1}^{2}+\|\partial_1\widetilde{\theta}\|_{H^1}^{2}+\|\widetilde{u_2}\|_{L^2}^2\Big).\label{B_44}
\end{align}
Collecting (\ref{B_11}), (\ref{B_22}), (\ref{B_33}) and (\ref{B_44}) gives 
\begin{align}
\frac12\frac{d}{dt}\Big(\|\nabla\widetilde{u}(t)\|_{L^2}^2&+\|\nabla\widetilde{\theta}(t)\|_{L^2}^2\Big)+\nu\|\partial_2\nabla\widetilde{u}\|_{L^2}^2+\eta\|\partial_1\nabla\widetilde{\theta}\|_{L^2}^2\nonumber\\&\le c\|(u,\theta)\|_{H^2}\Big(\|\partial_2\widetilde{u}\|_{H^1}^{2}+\|\partial_1\widetilde{\theta}\|_{H^1}^{2}+\|\widetilde{u_2}\|_{L^2}^2\Big).\label{M1}
\end{align}
Now, to control the norm $\|\widetilde{u_2}\|_{L^2}$ present in (\ref{M2}) and (\ref{M1}), we need to add the following term, 
\begin{align*}
-\frac{d}{dt}\Big(\delta(\widetilde{u_2},\widetilde{\theta})\Big)=-\delta(\partial_t\widetilde{u_2},\widetilde{\theta})-\delta(\widetilde{u_2},\partial_t\widetilde{\theta}), 
\end{align*}
with $\delta>0$ is a small constant to be fixed at the end of the proof. Doing so, we generate an extra regularization term that helps bound $\|\widetilde{u_2}\|_{L^2}$.
Note that, this stabilizing term comes from the interaction between $\widetilde u$ 
and $\widetilde\theta$. Due to Hölder's inequality, we have, for sufficiently small $\delta>0$, $$\|(\widetilde{u},\widetilde{\theta})\|_{H^1}^2-\delta(\widetilde{u_2},\widetilde{\theta})\ge 0.$$
Using the first equation of (\ref{*}) and $\overline{u_2} = 0$, we write
\begin{align}
\partial_t\widetilde{u_2}+\widetilde{u\cdot\nabla\widetilde{u_2}}+\underbrace{\widetilde{u_2}\partial_2\overline{u_2}}_{=0}-\nu\partial_2^2\widetilde{u_2}+\partial_2 \widetilde{p}=g_0\widetilde{\theta}\label{n_2}.
\end{align}
Applying $\nabla\cdot$ to the first equation of (\ref{*}), we obtain
\begin{align}
\nabla\cdot(\widetilde{u\cdot\nabla\widetilde{u}})+\nabla\cdot(\widetilde{u_2}\partial_2\overline{u})+\Delta \widetilde{p}=g_0\partial_2\widetilde{\theta}.\label{n_1}
\end{align}
Making use of (\ref{n_1}), we have 
\begin{align*}
\widetilde{p}=-\Delta^{-1}\nabla\cdot(\widetilde{u\cdot\nabla\widetilde{u}})-\Delta^{-1}\nabla\cdot(\widetilde{u_2}\partial_2\overline{u})+g_0\Delta^{-1}\partial_2\widetilde{\theta}.
\end{align*}
Then,
\begin{align}
\partial_2\widetilde{p}=-\partial_2\Delta^{-1}\nabla\cdot(\widetilde{u\cdot\nabla\widetilde{u}})-\partial_2\Delta^{-1}\nabla\cdot(\widetilde{u_2}\partial_2\overline{u})+g_0\partial_2\partial_2\Delta^{-1}\widetilde{\theta}.\label{p-2}
\end{align}
By (\ref{n_2}) and the second equation of (\ref{*}), we write
\begin{align}
-\delta\frac{d}{dt}(\widetilde{u_2},\widetilde{\theta})&=-\delta(\partial_t\widetilde{u_2},\widetilde{\theta})-\delta(\widetilde{u_2},\partial_t\widetilde{\theta})\nonumber\\&=-\delta(g_0\widetilde{\theta}-\partial_2 \widetilde{p}+\nu\partial_2^2\widetilde{u_2}-\widetilde{u\cdot\nabla\widetilde{u_2}},\widetilde{\theta})
\notag\\
&\quad -\delta(\widetilde{u_2},-g_0\widetilde{u_2}+\eta\partial_1^2\widetilde{\theta}-\widetilde{u_2}\partial_2\overline{\theta}-\widetilde{u\cdot\nabla\widetilde{\theta}})\nonumber\\&=-g_0\delta\|\widetilde{\theta}\|_{L^2}^2+\int\partial_2 \widetilde{p}\widetilde{\theta}dx-\delta\nu\int\partial_2^2\widetilde{u_2}\widetilde{\theta}dx+\delta\int\widetilde{u\cdot\nabla\widetilde{u_2}}\widetilde{\theta}dx\nonumber\\&\quad+g_0\delta\|\widetilde{u_2}\|_{L^2}^2-\delta\eta\int\partial_1^2\widetilde{\theta}\widetilde{u_2}dx+\delta\int\widetilde{u_2}\widetilde{u_2}\partial_2\overline{\theta}dx+\delta\int\widetilde{u\cdot\nabla\widetilde{\theta}}\widetilde{u_2}dx\nonumber\\&:=N_1+\cdots+N_8.\label{A}\end{align}
The terms $N_1$ through $N_8$ obey the following bounds. For $N_2$, we use (\ref{p-2}) to rewrite it as,
\begin{align}
N_2&:=\delta\int\partial_2 \widetilde{p}\widetilde{\theta}dx\nonumber\\&=-\delta\int\partial_2\Delta^{-1}\nabla\cdot(\widetilde{u\cdot\nabla\widetilde{u}})\cdot\widetilde{\theta}dx-\delta\int\partial_2\Delta^{-1}\nabla\cdot(\widetilde{u_2}\partial_2\overline{u})\cdot\widetilde{\theta}dx\nonumber\\&\quad+g_0\delta\int\partial_2\partial_2\Delta^{-1}\widetilde{\theta}\cdot\widetilde{\theta}dx\nonumber\\&:=N_{21}+N_{22}+N_{23}.\label{N-2}
\end{align}
By Lemma \ref{l1} and integration by parts we split $N_{21}$ into three pieces
\begin{align}
N_{21}&:=-\delta\int\partial_2\Delta^{-1}\nabla\cdot(\widetilde{u\cdot\nabla\widetilde{u}})\cdot\widetilde{\theta}dx\nonumber\\&=-\delta\int\partial_2\Delta^{-1}\nabla\cdot(u\cdot\nabla\widetilde{u})\cdot\widetilde{\theta}dx+\underbrace{\delta\int\partial_2\Delta^{-1}\nabla\cdot(\overline{u\cdot\nabla\widetilde{u}})\cdot\widetilde{\theta}dx}_{=0}\nonumber\\&=-\delta\int\partial_2\Delta^{-1}\partial_1(u_1\partial_1\widetilde{u})\cdot\widetilde{\theta}dx-\delta\int\partial_2\Delta^{-1}\partial_2(u_2\partial_2\widetilde{u})\cdot\widetilde{\theta}dx\nonumber\\&=-\delta\int(u_1\partial_1\widetilde{u})\cdot\partial_2\Delta^{-1}\partial_1\widetilde{\theta}dx-\delta\int(u_2\partial_2\widetilde{u})\cdot\partial_2\Delta^{-1}\partial_2\widetilde{\theta}dx\\&\nonumber=-\delta\int\partial_1\widetilde{u_1}\widetilde{u}\cdot\partial_2\Delta^{-1}\partial_1\widetilde{\theta}dx-\delta\int u_1\widetilde{u}\cdot\partial_1\partial_2\Delta^{-1}\partial_1\widetilde{\theta}dx\\&\nonumber\quad\quad-\delta\int(u_2\partial_2\widetilde{u})\cdot\partial_2\Delta^{-1}\partial_2\widetilde{\theta}dx\nonumber\\&=N_{211}+N_{212}+N_{213}.\label{N-21}
\end{align}
Due to $\nabla\cdot u=0$, Lemma \ref{l3} and the boundedness of the Riesz transform,
\begin{align}
N_{211}&=-\delta\int\partial_1\widetilde{u_1}\widetilde{u}\cdot\partial_2\Delta^{-1}\partial_1\widetilde{\theta}dx\nonumber\\&=\delta\int\partial_2\widetilde{u_2}\widetilde{u}\cdot\partial_2\Delta^{-1}\partial_1\widetilde{\theta}dx\nonumber\\&\le c\|\widetilde{u}\|_{L^2}\|\partial_2\widetilde{u_2}\|_{L^2}^{\frac12}\|\partial_2\partial_2\widetilde{u_2}\|_{L^2}^{\frac12}\|\partial_2\Delta^{-1}\partial_1\widetilde{\theta}\|_{L^2}^{\frac12}\|\partial_1\partial_2\Delta^{-1}\partial_1\widetilde{\theta}\|_{L^2}^{\frac12}\nonumber\\&\le c\|u\|_{H^2}\|\partial_2\widetilde{u}\|_{H^1}\|\widetilde{\theta}\|_{L^2}^{\frac12}\|\partial_1\widetilde{\theta}\|_{L^2}^{\frac12}\nonumber\\&\le c\|u\|_{H^2}\|\partial_2\widetilde{u}\|_{H^1}\|\partial_1\widetilde{\theta}\|_{H^1}\nonumber\\&\le c\|u\|_{H^2}\Big(\|\partial_2\widetilde{u}\|_{H^1}^2+\|\partial_1\widetilde{\theta}\|_{H^1}^2\Big).\label{N211}
\end{align}
According to Lemma \ref{l3}, the boundedness of the Riesz transform, Lemma \ref{l4} and $\nabla\cdot u=0$,
\begin{align}
N_{212}&=-\delta\int u_1\widetilde{u}\cdot\partial_1\partial_2\Delta^{-1}\partial_1\widetilde{\theta}dx\nonumber\\&\le c\|u_1\|_{L^2}\|\widetilde{u}\|_{L^2}^{\frac12}\|\partial_2\widetilde{u}\|_{L^2}^{\frac12}\|\partial_1\partial_2\Delta^{-1}\partial_1\widetilde{\theta}\|_{L^2}^{\frac12}\|\partial_1\partial_1\partial_2\Delta^{-1}\partial_1\widetilde{\theta}\|_{L^2}^{\frac12}\nonumber\\&\le c\|u\|_{H^2}\|\widetilde{u}\|_{L^2}^{\frac12}\|\partial_2\widetilde{u}\|_{L^2}^{\frac12}\|\partial_1\widetilde{\theta}\|_{L^2}^{\frac12}\|\partial_1\partial_1\widetilde{\theta}\|_{L^2}^{\frac12}\nonumber\\&\le c\|u\|_{H^2}(\|\widetilde{u_1}\|_{L^2}+\|\widetilde{u_2}\|_{L^2})^{\frac12}\|\partial_2\widetilde{u}\|_{H^1}^{\frac12}\|\partial_1\widetilde{\theta}\|_{H^1}\nonumber\\&\le c\|u\|_{H^2}\Big(\underbrace{\|\widetilde{u_1}\|_{L^2}^2}_{\le \|\partial_1\widetilde{u_1}\|_{L^2}^2=\|\partial_2\widetilde{u_2}\|_{L^2}^2}+\|\widetilde{u_2}\|_{L^2}^2+\|\partial_2\widetilde{u}\|_{H^1}^2+\|\partial_1\widetilde{\theta}\|_{H^1}^2\Big)\nonumber\\&\le c\|u\|_{H^2}\Big(\|\widetilde{u_2}\|_{L^2}^2+\|\partial_2\widetilde{u}\|_{H^1}^2+\|\partial_1\widetilde{\theta}\|_{H^1}^2\Big).\label{N212}
\end{align}
Applying Lemma \ref{l3}, the boundedness of the Riesz transform  and then Lemma \ref{l4},
\begin{align}
N_{213}&=-\delta\int(u_2\partial_2\widetilde{u})\cdot\partial_2\Delta^{-1}\partial_2\widetilde{\theta}dx\nonumber\\&\le c\|u_2\|_{L^2}\|\partial_2\widetilde{u}\|_{L^2}^{\frac12}\|\partial_2\partial_2\widetilde{u}\|_{L^2}^{\frac12}\|\partial_2\Delta^{-1}\partial_2\widetilde{\theta}\|_{L^2}^{\frac12}\|\partial_1\partial_2\Delta^{-1}\partial_2\widetilde{\theta}\|_{L^2}^{\frac12}\nonumber\\&\le c\|u\|_{H^2}\|\partial_2\widetilde{u}\|_{H^1}\|\widetilde{\theta}\|_{L^2}^{\frac12}\|\partial_1\widetilde{\theta}\|_{L^2}^{\frac12}\nonumber\\&\le c\|u\|_{H^2}\|\partial_2\widetilde{u}\|_{H^1}\|\partial_1\widetilde{\theta}\|_{H^1}\nonumber\\&\le c\|u\|_{H^2}\Big(\|\partial_2\widetilde{u}\|_{H^1}^2+\|\partial_1\widetilde{\theta}\|_{H^1}^2\Big).\label{N213}
\end{align}
The bounds in (\ref{N211}), (\ref{N212}) and (\ref{N213}) lead to
\begin{align}
N_{21} \le c\|u\|_{H^2}\Big(\|\widetilde{u_2}\|_{L^2}^2+\|\partial_2\widetilde{u}\|_{H^1}^2+\|\partial_1\widetilde{\theta}\|_{H^1}^2\Big).
\end{align}
Now we turn to the next term $N_{22}$. Using Hölder's inequality, the boundedness of the Riesz transform and Lemmas \ref{l1}, \ref{delta216} and \ref{l4}
\begin{align}
N_{22}&:=-\delta\int\partial_2\Delta^{-1}\nabla\cdot(\widetilde{u_2}\partial_2\overline{u})\cdot\widetilde{\theta}dx\nonumber\\&\le c\delta\|\partial_2\Delta^{-1}\nabla\cdot(\widetilde{u_2}\partial_2\overline{u})\|_{L^2}\|\widetilde{\theta}\|_{L^2}\nonumber\\&\le c\delta\|\widetilde{u_2}\partial_2\overline{u}\|_{L^2}\|\widetilde{\theta}\|_{L^2}\nonumber\\&\le c\delta\|\partial_2\overline{u}\|_{L^{\infty}_{x_2}}\|\widetilde{u_2}\|_{L^2}\|\widetilde{\theta}\|_{L^2}\nonumber\\&\le c\delta\|\partial_2\overline{u}\|_{H^1}\|\widetilde{u_2}\|_{L^2}\|\widetilde{\theta}\|_{L^2}\nonumber\\&\le c\delta\|u\|_{H^2}\|\widetilde{u_2}\|_{L^2}\|\partial_1\widetilde{\theta}\|_{L^2}\nonumber\\&\le c\delta\|u\|_{H^2}\Big(\|\widetilde{u_2}\|_{L^2}^2+\|\partial_1\widetilde{\theta}\|_{H^1}^2\Big).\label{N-22}
\end{align}
To deal with $N_{23}$, we integrate by parts, use Plancherel's theorem and then Lemma \ref{l4},
\begin{align}
N_{23}&:=g_0\delta\int\partial_2\partial_2\Delta^{-1}\widetilde{\theta}\cdot\widetilde{\theta}dx\nonumber\\&=g_0\delta\int\partial_2\Delta^{-\frac12}\widetilde{\theta}\cdot\partial_2\Delta^{-\frac12}\widetilde{\theta}dx\nonumber\\&= g_0\delta \|\p_2 \Lambda^{-1} \widetilde \theta\|^2_{L^2}\nonumber\\&= g_0\delta\sum\limits_{\underset{k\not= 0}{k\in\mathbb{Z}}} \int_{\mathbb R}  \frac{\xi_2^2}{k^2 + \xi_2^2} |\widehat{\widetilde \theta}(k, \xi_2)|^2 d\xi_2\nonumber\\&\le c\delta\sum\limits_{\underset{k\not= 0}{k\in\mathbb{Z}}} \int_{\mathbb R}  \xi_2^2|\widehat{\widetilde \theta}(k, \xi_2)|^2 d\xi_2= c\delta\|\p_2 \widetilde \theta\|_{L^2} ^2\le c\delta\|\partial_1\p_2 \widetilde \theta\|_{L^2} ^2\le c\delta\|\partial_1 \widetilde \theta\|_{H^1} ^2,\label{N-23}
\end{align}
where we denote $\Lambda=(-\Delta)^{\frac12}$ and we have used the fact that the oscillation part $\widehat{\theta}(0, \xi_2)$ has the horizontal mode equal to $0$, namely $\widehat{\widetilde{\theta}}(0, \xi_2)=0$. 
\\
Collecting (\ref{N-21}), (\ref{N-22}), (\ref{N-23}) and (\ref{N-2}), we find
\begin{align}
N_2\le c\delta\|(u,\theta)\|_{H^2}\Big(\|\partial_2\widetilde{u}\|_{H^1}^2+\|\widetilde{u_2}\|_{L^2}^2+\|\partial_1\widetilde{\theta}\|_{H^1}^2\Big)+c\delta\|\partial_1\widetilde{\theta}\|_{H^1}^2.\label{kk2}
\end{align}
To deal with $N_3$ we use $\nabla\cdot u=0$, integration by parts, H\"{o}lder's inequality and Lemma \ref{l4}, 
\begin{align}
N_3&:=-\delta\nu\int\partial_2^2\widetilde{u_2}\widetilde{\theta}dx=\delta\nu\int\partial_2\partial_1\widetilde{u_1}\widetilde{\theta}dx\nonumber\\&=-\delta\nu\int\widetilde{u_1}\partial_2\partial_1\widetilde{\theta}dx\nonumber\\&  \le \delta\nu\|\widetilde{u_1}\|_{L^2}\|\partial_2\partial_1\widetilde{\theta}\|_{L^2}\nonumber\\& \le c\delta\Big(\underbrace{\|\widetilde{u_1}\|_{L^2}^2}_{\le \|\partial_1\widetilde{u_1}\|_{L^2}^2= \|\partial_2\widetilde{u_2}\|_{L^2}^2}+\|\partial_2\partial_1\widetilde{\theta}\|_{L^2}^2\Big)\nonumber\\& \le c\delta\Big(\|\partial_2\widetilde{u}\|_{H^1}^2+\|\partial_1\widetilde{\theta}\|_{H^1}^2\Big).\label{kk3}
\end{align}
To estimate $N_4$, we make use of Lemma \ref{l1} and integration by parts, to write it as
\begin{align}
N_4&:=\delta\int\widetilde{u\cdot\nabla\widetilde{u_2}}\widetilde{\theta}dx\nonumber\\&=\delta\int u\cdot\nabla\widetilde{u_2}\widetilde{\theta}dx-\underbrace{\delta\int\overline{u\cdot\nabla\widetilde{u_2}}\widetilde{\theta}dx}_{=0}\nonumber\\&=\delta\int u\partial_1\widetilde{u_2}\widetilde{\theta}dx+\delta\int u\partial_2\widetilde{u_2}\widetilde{\theta}dx\nonumber\\&=-\delta\int \partial_1\widetilde{u}\widetilde{u_2}\widetilde{\theta}dx+\delta\int u\partial_2\widetilde{u_2}\widetilde{\theta}dx\nonumber\\&=N_{41}+N_{42}.\label{kk4}
\end{align}
By Lemmas \ref{l3} and \ref{l4}
\begin{align}
N_{41}&:=-\delta\int \partial_1\widetilde{u}\widetilde{u_2}\widetilde{\theta}dx\nonumber\\&\le c\|\partial_1\widetilde{u}\|_{L^2}\|\widetilde{u_2}\|_{L^2}^{\frac12}\|\partial_2\widetilde{u_2}\|_{L^2}^{\frac12}\|\widetilde{\theta}\|_{L^2}^{\frac12}\|\partial_1\widetilde{\theta}\|_{L^2}^{\frac12}\nonumber\\&\le c\|u\|_{H^2}\|\widetilde{u_2}\|_{L^2}^{\frac12}\|\partial_2\widetilde{u}\|_{L^2}^{\frac12}\|\partial_1\widetilde{\theta}\|_{H^1}\nonumber\\&\le c\|u\|_{H^2}\Big(\|\partial_2\widetilde{u}\|_{H^1}^2+\|\partial_1\widetilde{\theta}\|_{H^1}^2+\|\widetilde{u_2}\|_{L^2}^2\Big).\label{N41}
\end{align}
Similarly,
\begin{align}
N_{42}&=\delta\int u\partial_2\widetilde{u_2}\widetilde{\theta}dx\nonumber\\&\le c\|u\|_{L^2}\|\partial_2\widetilde{u_2}\|_{L^2}^{\frac12}\|\partial_2\partial_2\widetilde{u_2}\|_{L^2}^{\frac12}\|\widetilde{\theta}\|_{L^2}^{\frac12}\|\partial_1\widetilde{\theta}\|_{L^2}^{\frac12}\nonumber\\&\le c\|u\|_{H^2}\|\partial_2\widetilde{u}\|_{H^1}\|\partial_1\widetilde{\theta}\|_{H^1}\nonumber\\&\le c\|u\|_{H^2}\Big(\|\partial_2\widetilde{u}\|_{H^1}^2+\|\partial_1\widetilde{\theta}\|_{H^1}^2\Big).\label{N42}
\end{align}
Inserting (\ref{N41}) and (\ref{N42}) in (\ref{kk4}) we find
\begin{align}
N_4\le c\delta\|u\|_{H^2}\Big(\|\partial_2\widetilde{u}\|_{H^1}^2+\|\partial_1\widetilde{\theta}\|_{H^1}^2+\|\widetilde{u_2}\|_{L^2}^2\Big).
\end{align}
Clearly, the term $N_5$ can be bounded via Lemma \ref{l4},
\begin{align}
N_5&:=-g_0\delta\|\widetilde{\theta}\|_{L^2}^2\le c\delta\|\partial_1\widetilde{\theta}\|_{L^2}^2\le c\delta\|\partial_1\widetilde{\theta}\|_{H^1}^2.\label{kk5}
\end{align}
Applying H\"{o}lder's inequality and Young's inequality,
\begin{align}
N_6&:=-\delta\eta\int\partial_1^2\widetilde{\theta}\widetilde{u_2}dx\nonumber\\&\le c\delta\|\partial_1^2\widetilde{\theta}\|_{L^2}\|\widetilde{u_2}\|_{L^2}\nonumber\\&\le c\delta\|\partial_1\widetilde{\theta}\|_{H^1}\|\widetilde{u_2}\|_{L^2}\nonumber\\&\le c\delta\|\partial_1\widetilde{\theta}\|_{H^1}^2-g_0\frac{\delta}{4}\|\widetilde{u_2}\|_{L^2}^2.\label{kk6}
\end{align}
Using integration by parts and Lemma \ref{l3}, we obtain
\begin{align}
N_7&:=\delta\int\widetilde{u_2}\widetilde{u_2}\partial_2\overline{\theta}dx=2\delta\int\partial_2\widetilde{u_2}\widetilde{u_2}\overline{\theta}dx\nonumber\\&\le c\delta\|\partial_2\widetilde{u_2}\|_{L^2}^{\frac12}\|\partial_1\partial_2\widetilde{u_2}\|_{L^2}^{\frac12}\|\widetilde{u_2}\|_{L^2}^{\frac12}\|\partial_2\widetilde{u_2}\|_{L^2}^{\frac12}\|\overline{\theta}\|_{L^2}\nonumber\\&\le c\delta\|\partial_2\widetilde{u}\|_{L^2}^{\frac32}\|\widetilde{u_2}\|_{L^2}^{\frac12}\|\theta\|_{H^2}\nonumber\\&\le c\delta\|\theta\|_{H^2}(\|\partial_2\widetilde{u}\|_{H^1}^2+\|\widetilde{u_2}\|_{L^2}^2).\label{kk7}
\end{align}
It remain to bound the last term $N_8$. Making use of Lemma \ref{l1}, we divide it into three parts
\begin{align}
N_8&:=\delta\int\widetilde{u\cdot\nabla\widetilde{\theta}}\widetilde{u_2}dx\nonumber\\&=\delta\int u\cdot\nabla\widetilde{\theta}\widetilde{u_2}dx-\underbrace{\delta\int\overline{u\cdot\nabla\widetilde{\theta}}\widetilde{u_2}dx}_{=0}\nonumber\\&=\delta\int \widetilde{u_1}\partial_1\widetilde{\theta}\widetilde{u_2}dx+ \delta\int \overline{u_1}\partial_1\widetilde{\theta}\widetilde{u_2}dx +\delta\int u_2\partial_2\widetilde{\theta}\widetilde{u_2}dx\nonumber\\&:=N_{81}+N_{82} + N_{83}.\label{N-8}
\end{align}
Due to Lemmas \ref{l3}, \ref{l4} and divergence-free condition of $u$, we have
\begin{align}
N_{81}&:=\delta\int \widetilde{u_1}\partial_1\widetilde{\theta}\widetilde{u_2}dx\nonumber\\& \le c\delta\|\widetilde{u_1}\|_{L^2}^{\frac12}\|\partial_1\widetilde{u_1}\|_{L^2}^{\frac12}\|\widetilde{u_2}\|_{L^2}^{\frac12}\|\partial_2\widetilde{u_2}\|_{L^2}^{\frac12}\|\partial_1\widetilde{\theta}\|_{L^2}\nonumber\\& \le c\delta\|\partial_2\widetilde{u}\|_{L^2}^{\frac32}\|\widetilde{u_2}\|_{L^2}^{\frac12}\|\partial_1\widetilde{\theta}\|_{L^2}\nonumber\\&\le c\delta\|\theta\|_{H^2}(\|\partial_2\widetilde{u}\|_{H^1}^2+\|\widetilde{u_2}\|_{L^2}^2).\label{N-811}
\end{align}
By Lemma \ref{l1}, Hölder's inequality and Lemma \ref{delta216}, 
\begin{align}
N_{82}&:=\delta\int \overline{u_1}\partial_1\widetilde{\theta}\widetilde{u_2}dx \nonumber\\&\le \delta \|\overline{u_1}\|_{L^{\infty}_{x_2}}\|\partial_1\widetilde{\theta}\widetilde{u_2}\|_{L^1}\nonumber\\&\le c\delta \|\overline{u_1}\|_{L^{\infty}_{x_2}}\|\partial_1\widetilde{\theta}\|_{L^2}\|\widetilde{u_2}\|_{L^2}\nonumber\\&\le c\delta \|u\|_{H^1}\|\partial_1\widetilde{\theta}\|_{L^2}\|\widetilde{u_2}\|_{L^2}\nonumber\\&\le c\delta \|u\|_{H^2}\Big(\|\partial_1\widetilde{\theta}\|_{H^1}^2+\|\widetilde{u_2}\|_{L^2}^2\Big).\label{N-812}
\end{align}
Due to $\overline{u_2}=0$, integration by parts, Lemma \ref{l3} and Young's inequality
\begin{align}
N_{83}&:=\delta\int u_2\partial_2\widetilde{\theta}\widetilde{u_2}dx=\delta\int \widetilde{u_2}\partial_2\widetilde{\theta}\widetilde{u_2}dx\nonumber\\&=2\delta\int \partial_2\widetilde{u_2}\widetilde{\theta}\widetilde{u_2}dx\nonumber\\&\le c\delta\|\partial_2\widetilde{u_2}\|_{L^2}^{\frac12}\|\partial_1\partial_2\widetilde{u_2}\|_{L^2}^{\frac12}\|\widetilde{u_2}\|^{\frac12}_{L^2}\|\partial_2\widetilde{u_2}\|^{\frac12}_{L^2}\|\partial_2\widetilde{\theta}\|_{L^2}\nonumber\\&\le c\delta\|\partial_2\widetilde{u}\|^{\frac32}_{H^1}\|\widetilde{u_2}\|^{\frac12}_{L^2}\|\theta\|_{H^2}\nonumber\\&\le c\delta\|\theta\|_{H^2}\Big(\|\partial_2\widetilde{u}\|^{2}_{H^1}+\|\widetilde{u_2}\|^{2}_{L^2}\Big).\label{N-82}
\end{align}
Inserting (\ref{N-811}), (\ref{N-812}) and (\ref{N-82}) in (\ref{N-8}) leads to
\begin{align}
N_8\le c\delta \|(u,\theta)\|_{H^2}\Big(\|\partial_2\widetilde{u}\|_{H^1}^2+\|\widetilde{u_2}\|_{L^2}^2+\|\partial_1\widetilde{\theta}\|_{H^1}^2\Big).\label{kk8}
\end{align}
Considering (\ref{A}) and collecting (\ref{kk2}), (\ref{kk3}), (\ref{kk4}), (\ref{kk5}), (\ref{kk6}), (\ref{kk7}) and (\ref{kk8}), we obtain
\begin{align}
-\delta\frac{d}{dt}(\widetilde{u_2},\widetilde{\theta})&\le g_0\delta\|\widetilde{u_2}\|_{L^2}^2+c\delta \|(u,\theta)\|_{H^2}\Big(\|\partial_2\widetilde{u}\|_{H^1}^2+\|\widetilde{u_2}\|_{L^2}^2\Big)\nonumber\\&\quad-g_0\frac{\delta}{4}\|\widetilde{u_2}\|_{L^2}^2+c\delta\Big(\|\partial_2\widetilde{u}\|_{H^1}^2+\|\partial_1\widetilde{\theta}\|_{H^1}^2\Big).\label{r}
\end{align}
It then follows from (\ref{M2}), (\ref{M1}) and (\ref{r}) that
\begin{align*}
\frac{d}{dt}\Big(\|\widetilde{u}\|_{H^1}^2&+\|\widetilde{\theta}\|_{H^1}^2-\delta(\widetilde{u_2},\widetilde{\theta})\Big)+2\nu\|\partial_2\widetilde{u}\|_{H^1}^2+2\eta\|\partial_1\widetilde{\theta}\|_{H^1}^2\nonumber\\&\le c\|(u,\theta)\|_{H^2}\Big(\|\partial_2\widetilde{u}\|_{H^1}^2+\|\partial_1\widetilde{\theta}\|_{H^1}^2+\|\widetilde{u_2}\|_{L^2}^2\Big)\nonumber\\&\quad+g_0\frac{3\delta}{4}\|\widetilde{u_2}\|_{L^2}^2+c\delta \|(u,\theta)\|_{H^2}\Big(\|\partial_2\widetilde{u}\|_{H^1}^2+\|\widetilde{u_2}\|_{L^2}^2\Big)\nonumber\\&\quad\quad\quad\quad\quad\;+c\delta\Big(\|\partial_2\widetilde{u}\|_{H^1}^2+\|\partial_1\widetilde{\theta}\|_{H^1}^2\Big).
\end{align*}
Using Theorem \ref{TH}, if $\varepsilon>0$ is sufficiently small and $\|u_0\|_{L^2}+\|\theta_0\|_{L^2}\le \varepsilon$, then $\|(u(t),\theta(t))\|_{H^2}\le c\varepsilon$ and so,
\begin{align*}
\frac{d}{dt}\Big(\|\widetilde{u}\|_{H^1}^2&+\|\widetilde{\theta}\|_{H^1}^2-\delta(\widetilde{u_2},\widetilde{\theta})\Big)+2\nu\|\partial_2\widetilde{u}\|_{H^1}^2+2\eta\|\partial_1\widetilde{\theta}\|_{H^1}^2\nonumber\\&\le c\epsilon\Big(\|\partial_2\widetilde{u}\|_{H^1}^2+\|\partial_1\widetilde{\theta}\|_{H^1}^2+\|\widetilde{u_2}\|_{L^2}^2\Big)\nonumber\\&\quad+g_0\frac{3\delta}{4}\|\widetilde{u_2}\|_{L^2}^2+c\delta \epsilon\Big(\|\partial_2\widetilde{u}\|_{H^1}^2+\|\widetilde{u_2}\|_{L^2}^2\Big)\nonumber\\
&\quad+c\delta\Big(\|\partial_2\widetilde{u}\|_{H^1}^2+\|\partial_1\widetilde{\theta}\|_{H^1}^2\Big).
\end{align*}
Choosing $\epsilon>0$ such that $c\epsilon\le -g_0\min(\frac{1}{4},\frac{\delta}{4})$, we obtain 
\begin{align*}
\frac{d}{dt}\Big(\|\widetilde{u}\|_{H^1}^2&+\|\widetilde{\theta}\|_{H^1}^2-\delta(\widetilde{u_2},\widetilde{\theta})\Big)+2\nu\|\partial_2\widetilde{u}\|_{H^1}^2+2\eta\|\partial_1\widetilde{\theta}\|_{H^1}^2\nonumber\\&\le \frac{\delta}{4}\Big(\|\partial_2\widetilde{u}\|_{H^1}^2+\|\partial_1\widetilde{\theta}\|_{H^1}^2\Big)-g_0\frac{\delta}{4}\|\widetilde{u_2}\|_{L^2}^2\nonumber\\&\quad+g_0\frac{3\delta}{4}\|\widetilde{u_2}\|_{L^2}^2-g_0\frac{\delta}{4}\Big(\|\partial_2\widetilde{u}\|_{H^1}^2+\|\widetilde{u_2}\|_{L^2}^2\Big)\nonumber\\&\quad\quad\quad\quad\quad\quad+c\delta\Big(\|\partial_2\widetilde{u}\|_{H^1}^2+\|\partial_1\widetilde{\theta}\|_{H^1}^2\Big)
\nonumber\\&\le g_0\frac{\delta}{4}\|\widetilde{u_2}\|_{L^2}^2+c\delta\Big(\|\partial_2\widetilde{u}\|_{H^1}^2+\|\partial_1\widetilde{\theta}\|_{H^1}^2\Big).
\end{align*}
Choosing $\delta>0$ such that $c\delta\le\min(\nu,\eta, \frac{c}{2})$, we get
\begin{align}
\frac{d}{dt}\Big(\|\widetilde{u}\|_{H^1}^2&+\|\widetilde{\theta}\|_{H^1}^2-\delta(\widetilde{u_2},\widetilde{\theta})\Big)+\nu\|\partial_2\widetilde{u}\|_{H^1}^2+\eta\|\partial_1\widetilde{\theta}\|_{H^1}^2-g_0\frac{\delta}{4}\|\widetilde{u_2}\|_{L^2}^2\le 0.\label{:}
\end{align}
Due to the above choice of $\delta$, we obtain 
\begin{align*}
	\frac{1}{2} \Big(\|\widetilde{u}\|_{H^1}^2 + \|\widetilde{\theta}\|_{H^1}^2\Big) - \delta (\widetilde{u_2},\widetilde{\theta}) \ge 0.
\end{align*}
or 
$$
\frac12 (\|\widetilde{u}\|_{H^1}^2+\|\widetilde{\theta}\|_{H^1}^2) \le \|\widetilde{u}\|_{H^1}^2+\|\widetilde{\theta}\|_{H^1}^2-\delta(\widetilde{u_2},\widetilde{\theta}) \le \frac32 (\|\widetilde{u}\|_{H^1}^2+\|\widetilde{\theta}\|_{H^1}^2).
$$
For any $0 \le s \le t$, integrating (\ref{:}) in time leads to
\begin{align*}
	&\frac12 (\|\widetilde{u}(t)\|_{H^1}^2+\|\widetilde{\theta}(t)\|_{H^1}^2) + \int_{s}^t (\nu\|\partial_2\widetilde{u}\|_{H^1}^2+\eta\|\partial_1\widetilde{\theta}\|_{H^1}^2-g_0\frac{\delta}{4}\|\widetilde{u_2}\|_{L^2}^2)\, d\tau \\
	&\le \frac32 (\|\widetilde{u}(s)\|_{H^1}^2+\|\widetilde{\theta}(s)\|_{H^1}^2). 
\end{align*}
Then, for any $0 \le s \le t$, we have
\beq\label{bi}
\|\widetilde{u}(t)\|_{H^1}^2+\|\widetilde{\theta}(t)\|_{H^1}^2 \le 3(\|\widetilde{u}(s)\|_{H^1}^2+\|\widetilde{\theta}(s)\|_{H^1}^2)
\eeq
and 
$$
\int_{0}^\infty (\nu\|\partial_2\widetilde{u}\|_{H^1}^2+\eta\|\partial_1\widetilde{\theta}\|_{H^1}^2-g_0\frac{\delta}{4}\|\widetilde{u_2}\|_{L^2}^2)\, d\tau \le C <\infty. 
$$
Combining with the time integral bounds from Theorem \ref{TH},
\begin{align*}
	\int_0^{\infty} \|\partial_2 u\|_{H^2}^2\,dt < \infty, \quad \int_0^{\infty} \|\partial_1 u_2\|_{L^2}^2\,dt < \infty \quad \text{and} \quad \int_0^{\infty} \|\partial_1 \theta\|_{H^2}^2\,dt < \infty,
\end{align*} 
we get
\beq\label{bi2}
\int_0^\infty (\|\widetilde{u}(t)\|_{H^1}^2+\|\widetilde{\theta}(t)\|_{H^1}^2)\,dt < \infty. 
\eeq
Finally, applying Lemma \ref{special5} to (\ref{bi}) and (\ref{bi2}) leads to 
\begin{align*}
\|\widetilde{u}(t)\|_{H^1}^2+\|\widetilde{\theta}(t)\|_{H^1}^2 \le c (1 +t)^{-1},
\end{align*}
and the asymptotic behavior, as $t\to \infty$, 
$$
t\, (\|\widetilde{u}(t)\|_{H^1}^2+\|\widetilde{\theta}(t)\|_{H^1}^2) \to 0. 
$$
This completes the proof of Theorem~\ref{TH1}.
\end{proof}


\vskip .3in




\vskip .4in
\begin{thebibliography}{89}
	
\bibitem{BPW} O. Ben Said, U. Pandey and J. Wu. \emph{The stabilizing effect of the temperature on buoyancy-
driven fluids}. Indiana Univ. Math. J., 71 (2022), 2605-2645. doi: 10.48550/arXiv.2005.11661. 
 
\bibitem{ABPW} D. Adhikari, O. Ben Said, U. Pandey and J. Wu. \emph{Stability and large-time behavior for the $2$D Boussinesq system with horizontal dissipation and vertical thermal diffusion.} Nonlinear Differential Equations and
Applications (NoDEA), 29 (2022), No.4, Paper No. 42. 
\bibitem{BrS} L. Brandolese and M.E. Schonbek, \emph{Large time decay and growth for solutions of a viscous Boussinesq system}, Trans. Amer. Math. Soc. 364 (2012), 5057-5090.


\bibitem{Blu}H. Bluestein, \emph{Severe Convective Storms and Tornadoes: Observations and Dynamics}, 
Springer, (2013).

\bibitem{CCL} A. Castro, D. C\'ordoba and D. Lear, \emph{On the asymptotic stability of stratified solutions for the 2D Boussinesq equations with a velocity damping term}, Math. Models Methods Appl. Sci. 29 (2019), 1227-1277.

\bibitem{CaoWu} C. Cao and J. Wu, \emph{Global regularity for the 2D MHD equations with mixed partial dissipation and magnetic diffusion}, Adv.  Math.  226 (2011), 1803-1822.

\bibitem{ConW}P. Constantin, J. Wu,  J. Zhao and Y. Zhu, \emph{High Reynolds number and high 
	Weissenberg number Oldroyd-B model with dissipation}, J. Evolution Equations,  special issue in honor of the 60th birthday of Professor
	Matthias Hieber, https://doi.org/10.1007/s00028-020-00616-8, in press.

\bibitem{Den} S. Denisov, \emph{Double-exponential growth of the vorticity gradient for the two-dimensional Euler equation}, Proc. Amer. Math. Soc. 143 (2015), 1199--1210.


\bibitem{DWZZ} C. R. Doering, J. Wu, K. Zhao and X. Zheng, \emph{Long time behavior of the
two-dimensional Boussinesq equations
without buoyancy diffusion}, Physica D  376/377 (2018), 144-159.

\bibitem{DWXZ} B. Q. Dong, J. Wu, X. Xu and N. Zhu, \emph{Stability and exponential decay for the 2D anisotropic Boussinesq
equations with horizontal dissipation}, Calculus of Variations and Partial Differential Equations, 60 (2021), No.3,
Paper No. 116, 21 pp.

	\bibitem{ER}
	T.M. Elgindi and F. Rousset, \emph{Global regularity for some Oldroyd-B type models}, Comm. Pure Appl. Math. 68 (2015), 2005-2021.
	

\bibitem{Hol} J.R. Holton and G.J. Hakim, \emph{An Introduction to Dynamic Meteorology}, Academic press, Oxford, UK, (2013).

\bibitem{Kis} A. Kiselev and V. Sverak, \emph{Small scale creation for solutions of the incompressible two-dimensional Euler equation}, Ann. Math. 180 (2014),  1205-1220.

\bibitem{Luk1} G. Lukaszewicz, \emph{On nonstationary flows of asymmetric fluids}, Rend. Accad. Naz. Sci. XL Mem. Mat. (5) 12 (1988): 83–97.

\bibitem{Luk2} G. Lukaszewicz, \emph{On the existence, uniqueness and asymptotic properties for solutions of flows of asymmetric fluids}, Rend. Accad. Naz. Sci. XL Mem. Mat. (5) 13 (1989): 105–120.

\bibitem{Luk3} G. Lukaszewicz, \emph{Micropolar Fluids Theory and Applications}, Model. Simul. Sci. Eng. Technol. Birkhäuser, Boston., (1999).

\bibitem{Lad} O.A. Ladyzhenskaya, \emph{Solution “in the large” of the nonstationary boundary value problem for the Navier-Stokes system with two space variables}, Comm Pure Appl Math. 12(4) (1959): 427–433.

\bibitem{Li} Y. Li, \emph{Global regularity for the viscous Boussinesq equations}, Math. Meth. Appl. Sci. 27 (2004): 363-369.

\bibitem{LL} A. Larios, E. Lunasin and E.S. Titi, \emph{Global well-posedness for the 2D Boussinesq system with anisotropic viscosity and without heat diffusion}, J. Differential Equations 255 (2013): 2636–2654.

\bibitem{LWXZZ} S. Lai, J. Wu, X. Xu, J. Zhang and Y Zhong, \emph{Optimal decay estimates for the 2D Boussinesq
equations with partial dissipation}, Journal of Nonlinear Science, 31 (2021), No.1, 16.

\bibitem{LWZ} S. Lai, J. Wu and Y. Zhong, \emph{Stability and large-time behavior of the 2D Boussinesq equations with partial
dissipation}, Journal of Differential Equations, 271 (2021), 764-796.

\bibitem{Maj} A. Majda, \emph{ Introduction to PDEs and Waves for the Atmosphere and Ocean}, Courant Lecture Notes 9, Courant Institute of Mathematical Sciences and American Mathematical Society, (2003).

\bibitem{MaBe} A. Majda and A. Bertozzi,  \emph{ Vorticity and Incompressible Flow}, Cambridge University Press, (2002).

\bibitem{Ped} J.  Pedlosky \emph{ Geophysical Fluid Dynamics}, 2nd Edition, Springer-Verlag,
Berlin Heidelberg-New York, (1987).

\bibitem{Sch0} M. Schonbek,  \emph{$L^2$ decay for weak solutions of the Navier-Stokes equations}, Arch. Ration. Mech. Anal. 88 (1985), 209--222.
	
	\bibitem{Sch} M. Schonbek and M. Wiegner, \emph{On the decay of higher-order norms of the solutions of Navier-Stokes equations}, Proc. Roy. Soc. Edinburgh Sect. A 126 (1996), 677--685.
	
	\bibitem{Stein} E.M. Stein, 
	\emph{Singular Integrals and Differentiability Properties of Functions},  Princeton Univ. Press, Princeton, New Jersey, (1970).

\bibitem{TWZZ} L. Tao, J. Wu,  K. Zhao and  X. Zheng, \emph{Stability near hydrostatic equilibrium to the 2D Boussinesq equations without thermal diffusion}, Arch. Ration. Mech. Anal. 237 (2020), 585-630.


\bibitem{Tao} T. Tao, \emph{Nonlinear Dispersive Equations: Local and Global Analysis},
CBMS regional conference series in mathematics, (2006).

\bibitem{WuZhu} J. Wu and Y. Zhu, \emph{Global solutions of 3D incompressible MHD system with mixed partial dissipation and magnetic diffusion near an equilibrium},  Adv. Math. 377 (2021), 107466.

\bibitem{Zla}
	A. Zlatos, \emph{Exponential growth of the vorticity gradient for the Euler equation on the torus}, Adv. Math. 268 (2015), 396-403.

\end{thebibliography}
\end{document}